\documentclass[letterpaper,11pt,reqno]{amsart}
\usepackage[margin=1.2in]{geometry}
\usepackage{eucal}
\usepackage{tikz}
\usepackage{etex}
\usepackage{stackrel}
\usepackage{sseq}
\usepackage{tikz-cd}
\usepackage{amsmath,amssymb}               % great math stuff
\usepackage{verbatim}
\usepackage{xypic}
\usepackage{amsfonts}
\usepackage{epsfig}
\usepackage{enumerate}
\usepackage{enumitem}
\usepackage{ mathrsfs }
\usepackage{amsthm}
\usepackage{mathtools}
\usetikzlibrary{decorations.pathreplacing,calligraphy,arrows}
\usepackage[breaklinks,colorlinks,citecolor=teal,linkcolor=teal,urlcolor=teal,pagebackref,hyperindex]{hyperref}

\newtheorem{theorem}{Theorem}[section]
\newtheorem{lemma}[theorem]{Lemma}

\newtheorem*{question*}{Question}
\newtheorem{prop}[theorem]{Proposition}
\newtheorem{corollary}[theorem]{Corollary}

\theoremstyle{definition}
\newtheorem{defn}[theorem]{Definition}
\newtheorem{remark}[theorem]{Remark}

\newtheorem*{notation*}{Notation}

\newcommand{\Ima}{\text{Im}}

\newcommand{\C}{\mathbb C}

\newcommand{\Z}{\mathbb{Z}}

\newcommand{\R}{\mathbb{R}}
\newcommand{\p}{{\partial}}
\newcommand{\K}{\mathbb{K}}

\newcommand{\fa}{\mathbf{a}}
\newcommand{\fz}{\mathbf{z}}
\newcommand{\clL}{\mathcal{L}}
\newcommand{\cM}{\mathcal{M}}

\newcommand{\cP}{\mathcal{P}}

\newcommand{\ccc}{{\mathcal{C}^2}}
\newcommand{\cc}{{\mathcal{C}^1}}
\newcommand{\im}{\mbox{im}\,}

\newcommand{\wh}{\widehat}

\numberwithin{equation}{section}

\title{STEIN DOMAINS WITH EXOTIC CONTACT BOUNDARIES.}

\author{MU ZHAO}

\begin{document}
	
	\begin{abstract}
	 We introduce a new invariant, the \textit{positive idempotent group},
	 for strongly asymptotically dynamically convex contact manifolds. This invariant can be used to distinguish different contact structures. As an application, for any complex dimension $n>8$ and any positive integer $k$, we can construct $n-$dimensional Stein manifolds $V_0,V_1,\cdots,V_k$ such that $\tilde{H}_j(V_i)=0, j\neq n-1,n$,
	 $V_i's$ are almost symplectomorphic, their boundaries are in the same almost contact class but not contactomorphic.
 
	\end{abstract}
	
	\maketitle

	\tableofcontents
	\section{Introduction}
	In this paper, we will introduce a new invariant $I_+(\Sigma)$, the \textit{positive idempotent group}, for strongly asymptotically dynamically convex contact manifolds $(\Sigma,\xi,\Phi)$(see definition in Section~\ref{section : def of ADC}). The definition of positive idempotent group $I_+(W)$ depends on the filling $W$: it is well defined when $SH_*(W)\neq 0$ for some Liouville filling $W$, and it is independent of filling when $(\Sigma,\xi,\Phi)$ is a strongly ADC contact manifold.
	
	 The main purpose of this paper is to prove the following theorem:
	
	\begin{theorem}\label{Main Thm}
		If $(\Sigma,\xi,\Phi)$ is a strongly asymptotically dynamically convex contact structure with a Liouville filling $W$ such that $SH_*(W)\neq 0$,  then all connected Liouville fillings of $(\Sigma,\xi,\Phi)$ with nonzero symplectic homology have isomorphic positive idempotent group $I_+$.
	\end{theorem}
    \begin{remark}
    Here a Liouville filling $W$ of $(\Sigma,\xi,\Phi)$ means that $W$ is a filling of $(\Sigma,\xi)$ and the trivialization $\Phi$ of the canonical bundle extends over $W$.
    Now that all these Liouville fillings have isomorphic positive idempotent group, we can regard $I_+$  as an invariant for strongly ADC contact manifold. We will prove the result in section~\ref{section:indepnedence of I+}.
    \end{remark}

	 As an application, we will use the \textit{positive idempotent group} to distinguish contact boundaries of Stein manifolds, which has a long history. Y.Eliashberg   \cite{eliashberg1991symplectic} constructed an exotic contact structure representing the standard almost contact structure on $S^{4k+1}$, and I.Ustilovsky  \cite{ustilovsky1999contact} proved that every almost contact class on $S^{4k+1}$ has infinitely many different contact structures. M.McLean   \cite{mclean2007lefschetz} has shown that there are infinitely many exotic Stein structures $\mathbb{C}_k^n$ on $\mathbb{C}^n, n\geq 4$. Using flexible Weinstein structures, O.Lazarev   \cite{lazarev2016contact} proved that any contact manifold admitting an almost Weinstein filling admits infinitely many exotic contact structures with flexible fillings. We have the following theorem:
	\begin{theorem}\label{big theorem}
		For any complex dimension $n>8$ and any positive integer $k$, there are Stein domains $V_0,V_1,\cdots, V_k$ such that:
		\begin{itemize}
			\item $V_i's $ are almost symplectomorphic,
			\item the contact boundaries $\p V_i$ of $V_i$ are in the same almost contact class,
			\item  $\p V_i$ are mutually non-contactomorphic.
			\item $\tilde{H}_j(V_i)=0$ for $j\neq n,n-1$.
		\end{itemize}
	\end{theorem}
    \begin{remark}
    In Theorem 1.14 \cite{lazarev2016contact}, O.Lazarev proved that if $V$ is almost symplectomorphic to a domain containing a closed (regular) Lagrangian, then there are infinitely symplectic structures $V_k$ almost symplectomorphic to $V$ that are not symplectomorphic and their contact boundaries are not contactomorphic either. The Stein domains constructed in this paper are different from Lazarev's examples.
    \end{remark}
	\subsection{Sketch of the proof}

	The contact structure on $\p V_i$ in Theorem~\ref{big theorem} is asymptotically dynamically convex.  In the case when $I_+(\Sigma)$ is finite, we can define the \textit{positive idempotent index} $i(\Sigma):=|I_+(\Sigma)|$ (see Section~\ref{ss:definition of positive idempotent group}). The theorem~\ref{big theorem} is based on the following theorem, which will be proved in Section~\ref{Finite}:

\begin{theorem}\label{main technical theorem}
	There exists connected Weinstein domains $(W^{2n},\lambda,\psi),$ for any $n>8$ such that
	\begin{itemize}
		\item $(\p W,\lambda)$ is asymptotically dynamically convex,
		\item $SH_*(W,\Z/2\Z)\neq 0$,
		\item $|I(W,\Z/2\Z)|<\infty$.
		\item $\tilde{H}_i(W,\Z/2\Z)=0$, for $i\neq n,n-1$.
	\end{itemize}
\end{theorem}
\begin{remark}
The definition of $I$ is in equation~\ref{def of I}.
\end{remark}
The basic idea to construct the Weinstein domain is to use Brieskorn variety. First we take the complement of a specific Brieskorn variety and then attach a Weinstein 2-handle to kill the fundamental group. With the help of a covering trick we can show that the resultant manifold has asymptotically dynamically convex boundary. The full proof is at the end of this paper, see Section~\ref{Finite}.

\begin{comment}
\begin{remark}
In  \cite{mclean2007lefschetz}, for $n>4$, M.McLean constructed infinitely many exotic Stein structures $\mathbb{C}_k^n$ on $\mathbb{C}^n$, which are  Weinstein homotopic to  $(W_k,\lambda_k,\psi_k)$. So the boundaries of $\mathbb{C}_k^n$ are all ADC, and non-contactomorphic.
\end{remark}
\end{comment}
We will need the fact that any almost Weinstein domain admits a flexible Weinstein structure in the same almost symplectic class (See Section~\ref{ss:Weinstein handle & contact surgery}). Moreover, if a contact manifold admits a flexible filling, then it is asymptotically dynamically convex, as stated in the following lemma:
	\begin{lemma}[Corollary 4.1  \cite{lazarev2016contact}]\label{ADC boundary}
		If $(Y^{2n-1},\xi),n\geq 3$, has a flexible filling, then $(Y,\xi)$ is asymptotically dynamically convex.
	\end{lemma}
	\begin{proof}[Proof of Theorem~\ref{big theorem}]

         Let $(W,\lambda,\psi)$ be in Theorem~\ref{main technical theorem}.
		 There is a flexible Weinstein domain $(W_1,\lambda_1,\psi_1)$ that is almost symplectomorphic to $W$. 
		 Let
		 \[(W_i,\lambda_i,\psi_i):=\underbrace{(W,\lambda,\psi)\natural(W,\lambda,\psi)\natural\cdots\natural(W,\lambda,\psi)}_i\natural\underbrace{(W_1,\lambda_1,\psi_1)\natural(W_1,\lambda_1,\psi_1)\natural\cdots\natural(W_1,\lambda_1,\psi_1)}_{k-i}.
		 \]
		 That is, $W_i$ is the boundary connect sum of $i$ copies of $W$ and $k-i$ copies of the flexibilization of $W_1$.
		 The boundary connect sum is equivalent to attaching a Weinstein 1-handle, so $W_i$ is a Weinstein domain. By construction, they are all almost symplectomorphic, see subsection~\ref{sec: formal structures}, and their boundaries are in the same almost contact class by lemma~\ref{lemma for almost contact}. Theorem~\ref{from weinstein to stein} allows us to deform a Weinstein structure into a Stein structure, which is denoted by $V_i$. The last condition is obvious.
		 There's only the third condition left to be verified. Indeed, we have $(\p V_i,\lambda_i)$ is asymptotically dynamically convex. Furthermore, we have:  
		 \begin{prop}\label{new prop}
		  $|I_+(\p W_i)|\neq |I_+(\p W_j)|,i\neq j.$
		 \end{prop}
		 The proof of Proposition~\ref{new prop} will be defer to Subsection~\ref{subsection:effect of conatct surgery}.

	\end{proof}	
		
\section{Background}\label{section:background}

	\subsection{Conventions and notation} (See Section~\ref{section:background} for detailed definitions.)
	
	Let $\lambda$ be a Liouville 1-form on a Liouville manifold $W$.
	$$
	d\lambda(\cdot,J\cdot)=g_J \qquad\text{(Riemannian metric)},
	$$
	$$
	d\lambda(X_H,\cdot)=-dH,\qquad X_H=J\nabla H \qquad\text{(Hamiltonian vector field)},
	$$
	$$
	\clL\wh W:=C^\infty(S^1,\wh W), \qquad S^1=\R/\Z \qquad \text{(loop space)},
	$$
	$$
	A_H:\clL\wh W\to \R,\qquad A_H(x):=\int_{S^1}x^*\lambda -
	\int_{S^1}H(t,x(t))\,dt \qquad\text{(action)},
	$$
	$$
	\nabla A_H(x)=-J(x)(\dot x-X_H(t,x)) \qquad\text{($L^2$-gradient)},
	$$
	$$
	u:\R\to\clL W,\qquad \p_su=\nabla A_H(u(s,\cdot)) \qquad\text{(gradient
		line)}
	$$
	\begin{equation}\label{eq:Floer}
	\Longleftrightarrow \p_s u + J(u)(\p_t u-X_H(t,u))=0
	\qquad\text{(Floer equation)},  
	\end{equation}
	$$
	\cP(H):=\mbox{Crit}(A_H) = \{\text{$1$-periodic orbits of the Hamiltonian vector field $X_H$}\} ,
	$$
	 
    \[
	\textrm{For each }  h\in H_1(W),
	\cP^h(H):=\mbox{Crit}_h(A_H) =  
	 \{ x\in \cP(H) \big| [x]=h\in[S^1\to W]  \}
	\]

	$$
	\hspace{.5cm} \mathcal M(x_-,x_+;H,J)=\{u:\R\times S^1\to W \mid
	\p_su= \nabla A_H(u(s,\cdot)),\ u(\pm\infty,\cdot)=x_\pm\}/\R 
	$$
		$$
	\mbox{(moduli space of Floer trajectories connecting $x_\pm\in\cP(H)$)},
	$$
\[
	\hspace{.5cm} \mathcal M_h(x_-,x_+;H,J)=\{u:\R\times S^1\to W \mid
	\p_su= \nabla A_H(u(s,\cdot)),\ u(\pm\infty,\cdot)=x_\pm\in \cP^h(H)\}/\R 
\]
	$$
	\mbox{(moduli space of Floer trajectories connecting $x_\pm\in\cP^h(H)$)},
	$$

	$$
	\dim \cM(x_-,x_+;H,J)=\mu_{CZ}(x_+)-\mu_{CZ}(x_-)-1,
	$$
	$$
	A_H(x_+)-A_H(x_-) = \int_{\R\times S^1}|\p_su|^2ds\,dt = \int_{\R\times S^1}u^*(d\lambda-dH\wedge dt).  
	$$
	Here the formula expressing the dimension of the moduli space in terms
	of Conley-Zehnder indices is to be understood with respect to a
	symplectic trivialization of $u^*TW$.  
	
	Let $\K$ be a {\color{black}field} and $a<b$ with $a,b\notin\mbox{Spec}(\p W,\alpha)$. We define the filtered Floer
	chain groups with coefficients in $\K$ by  
	\[
	SC_*^{<b}(H) := \bigoplus_{\scriptsize \begin{array}{c} x\in \cP(H)
		\\ A_H(x)<b \end{array}} \K\cdot x,\qquad
	SC_*^{(a,b)}(H) = SC_*^{<b}(H)/SC_*^{<a}(H),
     \]
	with the differential $d:SC_*^{(a,b)}(H)\to SC_{*-1}^{(a,b)}(H)$
	given by
	\[
	d x_+=\sum_{\mu_{CZ}(x_-)=\mu_{CZ}(x_+)-1} \#\mathcal M(x_-,x_+;H,J)\cdot x_-.
    \]
	Here $\#$ denotes the signed count of points with respect to suitable
	orientations. 
	We think of the cylinder $\R\times S^1$ as the twice punctured Riemann
	sphere, with the positive puncture at $+\infty$ as incoming, and the
	negative puncture at $-\infty$ as outgoing. This terminology makes
	reference to the corresponding asymptote being an input, respectively
	an output for the Floer differential. Note that the differential
	decreases both the action $A_H$ and the Conley-Zehnder index. 
	The filtered Floer homology is now defined as 
	$$
	SH_*^{(a,b)}(H) = \ker d/\im d. 
	$$
	Note that for $a<b<c$ the short exact sequence
	$$
	0 \to SC_*^{(a,b)}(H) \to SC_*^{(a,c)}(H) \to SC_*^{(b,c)}(H) \to 0
	$$
	induces a {\em tautological exact triangle} 
	\begin{equation}\label{eq:taut1}
	SH_*^{(a,b)}(H) \to SH_*^{(a,c)}(H) \to SH_*^{(b,c)}(H)
	\to SH_*^{(a,b)}(H)[-1].
	\end{equation}
	
	{\color{black}\noindent {\bf Remark.} We will suppress the field $\K$ from the notation. As noted in the Introduction, the definition can also be given with coefficients in a commutative ring. In this paper, $\K=\Z_2$.}
    \begin{notation*}
   
	Let $\fa=(a_0,a_1,\cdots,a_n)$ be an $(n+1)$-tuple of integers $a_i>1,\mathbf{z}:=(z_0,z_1,\cdots,z_n)\in \mathbb{C}^{n+1}$, and set $f(\mathbf{z}):=z_0^{a_0}+z_1^{a_1}+\cdots+z_n^{a_n}$,
	and let $B(s)$ to be the closed ball of radius $s$.
	\[V_\fa(t):=\{(z_0,z_1,\cdots,z_n)\in \mathbb{C}^{n+1}| f(\mathbf{z})=t\}.\]
	We will often suppress $\fa$ from the notation. Let
	\[X_t^s=V(t)\cap B(s).
	\]
	 and let $\beta\in C^\infty(\mathbb{R})$ be a smooth monotone decreasing cut-off function with $\beta(x)=1,x\leq \frac{1}{4}$ and $\beta(x)=0,x\geq \frac{3}{4}$,
	\[U_\fa(\epsilon):=\{\mathbf{z}\in \mathbb{C}^{n+1}|z_0^{a_0}+\cdots+z_n^{a_n}=\epsilon\cdot\beta(||\mathbf{z}||^2)\}.\]
	Likewise $\fa$ will often be suppressed.
	Moreover let \[W_\epsilon^s=U(\epsilon)\cap B(s).\]

\end{notation*}

	\subsection{Symplectic and contact structures}
	
	A \textit{symplectic manifold} $(M,\omega) $ is a smooth $2n$-dimensional manifold $M$ together with a nondegenerate, closed 2-form $\omega$. A function $H\in C^{\infty}(M)$ on a symplectic manifold $(M,\omega)$ is called \textit{Hamiltonian}. We define its \textit{Hamiltonian vector field} $X_H$ via
	\[dH=-\iota_{X_H}\omega=-\omega(X_H,\cdot)=\omega(\cdot,X_H).\]
	
	A \textit{contact manifold} $\Sigma$ is a smooth $(2n-1)$-dimensional manifold together with a completely non-integrable smooth hyperplane distribution $\xi \in T\Sigma$. The distribution is called a \textit{contact structure}. It can be locally defined as $\xi=\ker\alpha$ for some local 1-form $\alpha$ such that $\alpha\wedge (d\alpha)^{n-1}\neq 0$ pointwise. If $\alpha$ is globally defined, then  $\alpha$ is called a \textit{contact form}. We will always assume  $\alpha$ is globally defined. Under this assumption  $\alpha\wedge (d\alpha)^{n-1}$ gives rise to a volume form and hence $\Sigma$ is orientable. Once an orientation is chosen we require that $\alpha\wedge (d\alpha)^{n-1} > 0$. Associated with a contact form $\alpha$ one has a \textit{Reeb vector field} $R$, uniquely defined by the equations
	\begin{align*}
	&\iota_R(d\alpha)=0, \\  
	&\iota_R \alpha=1.
	\end{align*}
	Clearly $R$ is transverse to $\xi$. If we have two different forms $\alpha, \alpha^{'}$ which define the same contact structure, then we can find a nowhere vanishing function $f$ such that $\alpha^{'}=f\cdot\alpha$. Indeed, $f=\alpha^{'}(R)$. The flow of a Reeb vector field is called Reeb flow, and closed trajectories of Reeb flow are called the \textit{Reeb orbits}. The \textit{action} of a Reeb orbit $\gamma$ is defined as 
	\[A(\gamma):=\int_{S^1}\gamma^{*}\alpha
	\]
	Note that $A(\gamma)$ is always positive and equals the period of $\gamma$. The \textit{spectrum} $spec(\Sigma,\alpha)$ is the set of actions of all Reeb orbits of $\alpha$.
	We will need the following definition for Reeb trajectories which is part of a closed Reeb orbit.
	\begin{defn}
	 $\gamma:[0,T]\to X$ is called a \textit{fractional Reeb orbit} for contact manifold $(X,\xi)$ if there is a closed Reeb orbit $\gamma_0$ of $(X,\xi)$ such that $\gamma(t)=\gamma_0(t), t\in [0,T]$.
	\end{defn}

	We say that a Reeb orbit $\gamma$  of $\alpha$ is \textit{non-degenerate} if the linearized Reeb flow along $\gamma$ from $\xi_p$ to itself for some $p \in \gamma$ has no eigenvalue 1. Moreover we say that a contact form is \textit{non-degenerate} if all Reeb orbits of $\alpha$ are non-degenerate. We can always assume a contact form is non-degenerate after a $\mathcal{C}^0$-small perturbation,  since a generic contact form is non-degenerate. Notice that when $\alpha$ is non-degenerate, $spec(\Sigma,\alpha)$ is a discrete subspace of $\mathbb{R}^{+}$.
	
	\subsection{Liouville and Weinstein domains}\label{ssec: domains}

	A \textit{Liouville domain} is a pair $(W^{2n}, \lambda)$ such that \begin{itemize}
		\item $W^{2n}$ is a compact manifold with boundary,
		\item $d\lambda$ is a symplectic form on $W$ ,
		\item the Liouville field $X_\lambda$, defined by $i_{X}d\lambda = \lambda$, is outward transverse along $\partial W$.
	\end{itemize}
	Let $\alpha:=\lambda|_{\partial W}$ be a contact one-form on $\partial W$. The negative flow of $X$ gives rise to a collar:
	\begin{align*}
	&\phi: (1-\epsilon,1]\times\partial W\rightarrow W, \\
	&\phi^*\lambda = r\alpha,\quad \phi^*X=r\partial_r.
	\end{align*}
	 We can attach an infinite cone to it, which is called the \textit{completion} of $(W,\lambda)$:
	\begin{align*}
	&\widehat{W}=W\cup_{\partial W}([1,\infty)\times \partial W),
	\qquad \hat{\lambda}|_W=\lambda\\
	&\hat{\lambda}|([1,\infty)\times \partial W)= r\alpha,\quad
	\hat{X}|([0,\infty)\times \partial W)= r\partial_r, \quad\hat{\omega}=d\hat{\lambda}.
	\end{align*}
	A \textit{Liouville isomorphism} between domains $W_0,W_1$ is a diffeomorphism $\psi: \widehat{W_0}\to\widehat{W_1}$ satisfying $\psi^*\hat{\lambda_1}=\hat{\lambda_0}+df$, for some $f$ compactly supported. We also say that $\widehat{W_0}$ and $\widehat{W_1}$ are Liouville isomorphic. Clearly $\psi$ is compatible with the Liouville flow at infinity.
	
	\begin{defn}
		A Liouville domain $(W,\lambda)$ is called $G$-equivariant if a group $G$ acts on $W$ and $\lambda$ is $G$-invariant, i.e, $g^{*}\lambda=\lambda,\forall g\in G$. A diffeomorphism $f$ between two $G$-equivariant Liouville domains is called $G$-equivariant if the following diagram commutes, for all $g \in G$:
		\[\begin{tikzcd}
		(W_1,\lambda_1) \arrow[r, "f"] \arrow[d, "g*"]
		& (W_0,\lambda_0) \arrow[d, "g*" ] \\
		(W_1,\lambda_1) \arrow[r,  "f" ]
		& (W_0,\lambda_0)
		\end{tikzcd}
		\]
	\end{defn}
	\begin{remark}
		A manifold $M$ is called $G$-equivariant if $G$ acts on it. 
	\end{remark}

	\begin{prop}[Proposition 11.8  \cite{cieliebak2012stein}]\label{Liouville homotopy theorem of Cieliebak and Eliashberg}
		Let $W$ be a compact symplectic manifold with contact type boundary and $\lambda_t,t\in [0,1]$ be a homotopy of Liouville forms on $W$. Then there exits a diffeomorphism of the completions $f:\widehat{W_0}\to \widehat{W_1}$ such that $f^*\hat{\lambda_1}-\hat{\lambda_0}=dg$ where $g$ is a compactly supported function.
	\end{prop}
	We have an immediate corollary for Proposition~\ref{Liouville homotopy theorem of Cieliebak and Eliashberg}:
	\begin{corollary}\label{liouville homotopy}
		Let $(\lambda_t)_{0\leq t\leq 1}$ be a family of  ($G-$equivariant) Liouville structures on $W$. Then all the $(W,\lambda_t)$ ($(\widehat{W},\hat{\lambda_t})$) are mutually ($G-$equivariantly) Liouville isomorphic.
	\end{corollary}
	A \textit{Weinstein domain} is a triple $(W^{2n}, \lambda, \phi)$ such that
	\begin{itemize}
		\item $(W, \lambda)$ is a Liouville domain,
		\item $\phi: W \rightarrow \mathbb{R}$ is an exhausting Morse function with $\partial W$ being a regular level set,
		\item $X_\lambda$ is a gradient-like vector field for $\phi$.
	\end{itemize}
	Since $W$ is compact and $\phi$ is an exhausting Morse function with $\partial W$ as a regular level set, $\phi$ has finitely many critical points. Liouville and Weinstein \textit{cobordisms} are defined similarly. If a contact manifold $(Y, \xi)$ is contactomorphic to $\partial (W, \lambda)$, then  we say that $(W, \lambda)$ is a Liouville or Weinstein \textit{filling} of $(Y, \xi)$.
	\begin{defn}\label{stein domain}
		A \textit{Stein manifold} $(M,J,\phi)$ is a complex manifold $(M,J)$ with an exhausting plurisubharmonic function $\phi : M\to \mathbb{R} $. A manifold of the form $\phi^{-1}((-\infty,c])$ is called a \textit{Stein domain}, where c is a regular value of $\phi$.
	\end{defn}
	
	We also have the following famous theorem by Eliashberg:
	\begin{theorem}[Theorem 1.1 \cite{cieliebak2012stein}]\label{from weinstein to stein}
		Given a Weinstein structure  $\mathfrak{M}=(\omega,X,\phi)$ on $V$, there exists a Stein structure $(J,\phi)$ on $V$ such that $\mathfrak{M}(J,\phi)$ is Weinstein homotopic to $\mathfrak{M}$ with fixed $\phi$.
	\end{theorem}

	\subsection{Symplectic homology}\label{subsec: symhom}
	This section is mainly taken out from  \cite{lazarev2016contact}. The convention used here agrees with  \cite{cieliebak2018symplectic} .
	\subsubsection{Admissible Hamiltonians and almost complex structures}\label{sss:Ad Hamiltonian}
	Let $\mathcal{H}_{std}(W)$ denote the class of \textit{admissible Hamiltonians}, which are functions on $\widehat{W}$ defined up to smooth approximation as follows:
	\begin{itemize}
		\item $H^s \equiv 0$ in $W$,
		\item $H^s$ is linear in $r$ with slope $s \not\in Spec(Y, \alpha)$ in $\widehat{W}\setminus W = Y\times [1, \infty)$. 
	\end{itemize}
	To be more precise, $H$ is a $\mathcal{C}^2$-small Morse function in $W$ and $H=h(r)$ in $\widehat{W}\setminus W$ for some function $h$ such that
	\begin{itemize}
		\item $h$ is increasing convex in a small region  $(Y \times [1, 1+\epsilon], r\alpha)$ of $Y$,
		\item $h$ is linear with slope $s$ outside this region.
	\end{itemize}
	For $H \in \mathcal{H}_{std}(W)$, the Hamiltonian vector field $X_H$ is defined by $d\hat\lambda(\cdot, X_H ) = dH$. The time-1 orbits of $X_H$ are called the Hamiltonian orbits of $H$. Depending on their location in $\widehat{W}$, we can classify them into two categories: 
	\begin{itemize}
		\item In the interior of $W$, the only Hamiltonian orbits are constants corresponding to critical points of $H|_W$.
		\item In $\widehat{W}\setminus W$, we have $X_H = h'(r)R_\alpha$, where $R_\alpha$ is the Reeb vector field of $(Y, \alpha)$. Therefore all Hamiltonian orbits lie on level sets of $r$ and corresponding to some Reeb orbit of $\alpha$ with period $h'(r)$.
	\end{itemize}
	 The slope $s$ of $H$ at infinity is not in $Spec(Y, \alpha)$, as a consequence, every non-constant Hamiltonian orbit lies in a small neighborhood of $Y$ in $\widehat{W}$.
	After a $\mathcal{C}^2$-small time-dependent perturbation of $H$, the orbits become \textit{non-degenerate}.
	These non-degenerate orbits also lie in a neighborhood of $W$ and so their number is finite. 
	
	An almost complex structure $J$ is \textit{cylindrical} on the symplectization 
	$(Y \times (0, \infty), r\alpha)$ if 
	\begin{itemize}
		\item $J$ is independent of $r$,
		\item $J(r\partial_r) = R_\alpha$,
		\item $J$ preserves 
		$\xi = \ker \alpha,$ $J|_\xi$,
		\item $J$ is compatible with $d(r\alpha)|_\xi$.
	\end{itemize}
	Now we define the \textit{admissible} almost complex structures $J$ on $\widehat{W}$, denoted by  $\mathcal{J}_{std}(W)$:
	\begin{itemize}
		\item $J$ is cylindrical on  $\widehat{W}\backslash W = (Y \times [1, \infty), r\alpha)$
		\item $J$ is compatible with $\omega$ on $\widehat{W}$.
	 
	\end{itemize}

	\subsubsection{Floer complex}
	For $H \in \mathcal{H}_{std}(W), J\in \mathcal{J}_{std}(W)$, the Floer complex $SC(W,\lambda,  H, J)$  is generated as a free abelian group by Hamiltonian orbits of $H$. In this paper we need to consider all Hamiltonian orbits, as opposed to only the contractible ones, see \cite{wendlbeginner}.
	
	First, let's fix a reference loop
	\[l_h:S^1\to W
	\]
	with $[l_h]=h\in H_1(W,\Z)$. Denote by $\mathcal{P}^h(H)$ the set of all $1-$periodic orbits of $X_{H_t}$ in the homology class $h$.

	For a fixed
	reference class $h$,
	we will often write the chain complex generated as a free abelian group by orbits in  $\mathcal{P}^h(H)$ as $SC^h(H, J)$ when we do not need to specify $(W, \lambda)$. We will suppress $h$ when it causes no confusion.
	
	The differential is given by counts of Floer trajectories. In particular, for two Hamiltonian orbits $x_-, x_+$ of $H$, let $\widehat{\mathcal{M}}(x_-, x_+; H, J)$ be the moduli space of smooth maps 
	$u: \mathbb{R}\times S^1  \rightarrow \widehat{W}$ such that $
	\underset{s\rightarrow \pm \infty}{\lim}
	u(s, \cdot) = x_\pm$ and $u$ satisfies 
	Floer's equation 
	\begin{equation}
	\partial_s u + J(\partial_t u - X_H) =0.
	\end{equation}
	Here $s, t$ denotes the $\mathbb{R},\, S^1$ coordinates on $\mathbb{R}\times S^1$ respectively. 
	Since the Floer equation is $\mathbb{R}$-invariant, there is a free $\mathbb{R}$-action on $\widehat{\mathcal{M}}(x_-, x_+; H, J)$ for $x_- \ne x_+$. 
	Let $\mathcal{M}(x_-, x_+: H, J)$ be the quotient by this $\mathbb{R}$-action, that is, 
	$\widehat{\mathcal{M}}(x_-, x_+; H, J)/ \mathbb{R}$.
	After a small time-dependent perturbation of $(H,J)$, $\mathcal{M}(x_-, x_+, H, J)$ is a smooth finite-dimensional manifold. 
	
	A maximal principle ensures us that Floer trajectories will not escape to infinity in $\widehat{W}$.
	Let $V \subset (W, \lambda_W)$ be a \textit{Liouville subdomain}, that is,  $(V, \lambda_W|_V)$ is a Liouville domain and  $(Z,\alpha_Z) = \partial (V, \lambda)$  a contact manifold. 
	Since $V$ is a Liouville subdomain, there is a collar of $Z$ in $W$ that is symplectomorphic to $(Z \times [1, 1+\delta], d(t\alpha_Z))$ for some small $\delta$. We have the following lemma:
	\begin{lemma} \cite{abouzaid2010open}\label{lem: maximal_principle}
		Consider $H: \widehat{W} \rightarrow \mathbb{R}$  such that  $H = h(r)$  is increasing near $Z$, where $r$ is the cylindrical coordinate and  $J \in \mathcal{J}_{std}(W)$ is cylindrical near $Z$. If both asymptotic orbits of a $(H, J)$-Floer trajectory $u: \mathbb{R}\times S^1 \rightarrow \widehat{W}$ are contained in $V$, then $u$ is contained in $V$.  
	\end{lemma}
	
	Apply this result to $V = W$, then we can proceed as if $W$ were closed. Therefore $\mathcal{M}(x_-, x_+; H, J)$ has a codimension one compactification by the Gromov-Floer compactness theorem. This implies that $\mathcal{M}_h(x_-, x_+; H, J)$, the zero-dimensional component of $\mathcal{M}(x_-, x_+; H, J)$, is finite and the map 
	$\p: SC(H, J) \rightarrow SC(H, J), $
	defined by
	\[
	\p x_+ := \sum_{x_-} \# \mathcal{M}_h(x_-, x_+; H, J)\cdot x_-\pmod2
	\]
	 is a differential. 
Notice that the underlying vector space $SC(H,J)$ depends only on $H$ while the differential $\p$ depends on both $H$ and $J$. 
	The resulting homology $HF(H, J)$ is independent of $J$ and compactly supported deformations of $H$. 
	\begin{remark}
	If $c_1(W, \omega) = 0$, then $HF(H, J)$ has a $\mathbb{Z}$-grading due to the fact that
	$c_1(W, \omega) = 0$ implies the canonical line bundle of $(W, \omega)$ being trivial. For all our purposes, the canonical line bundle will always be trivial in this paper.
	Once we fix a global trivialization of this bundle, 
	we can assign to each Hamiltonian orbit $x$ an integer, known as the Conley-Zehnder index $\mu_{CZ}(x)$ (see Subsection~\ref{subsection:index}).
	\end{remark}

	Generally speaking, the orbit $x$, Conley-Zehnder index  $\mu_{CZ}(x)$ depend on the choice of trivialization of the canonical bundle. For a Hamiltonian orbit 
	corresponding to a critical point $p$ of the Morse function $H|_W$, the Conley-Zehnder index $\mu_{CZ}(p)$ coincides with $n- Ind(p)$, where $Ind(p)$ is the Morse index of $H|_W$ at $p$. 
	
	\subsubsection{Continuation map}\label{sssec: continuation_map}
	Although $HF(H, J)$ is independent of $J$ and compactly supported deformations of $H$, $HF(H, J)$ does depend on the slope of $H$ at infinity and therefore is not an invariant of $W$. Indeed,  $HF(H, J)$ only sees Reeb orbits of period less than the slope of $H$ at infinity. To incorporate all Reeb orbits, we have to consider Hamiltonians with arbitrarily large slope. 
	
	More formally, this can be done by considering continuation maps between $SC(H, J)$ for different $H$. Given $H_-, H_+ \in \mathcal{H}_{std}(W)$, let $H_s \in \mathcal{H}_{std}(W), s\in \mathbb{R},$ be a family of Hamiltonians such that $H_s = H_-$ for $s \ll 0$ and $H_s = H_+$ for $s\gg 0$. Similarly, let $J_s \in \mathcal{J}_{std}(W)$ interpolate between $J_-, J_+$. For  Hamiltonian orbits $x_-, x_+$ of $H_-, H_+$ respectively, let $\mathcal{M}(x_-, x_+; H_s, J_s)$ be the moduli space of parametrized Floer trajectories, i.e. maps 
	$u:   \mathbb{R}\times S^1 \rightarrow \widehat{W}$ 
	$$
	\partial_s u + J_s(\partial_t u - X_{H_s}) = 0
	$$

	To ensure that parametrized Floer trajectories do not escape to infinity, we have to use a maximal principle. For this principle to hold, it is crucial that the homotopy of Hamiltonian functions is decreasing, that is, $\partial H_s/ \partial s \le 0$. 
	If $J_s$ is $s$-independent, we use the following parametrized version of `no escape' Lemma \ref{lem: maximal_principle}, which is proven in Proposition 3.1.10 of  \cite{gutt2015positive}. 
	If $J_s$ does depend on $s$ and $V = W$, then we use the maximal principle from  \cite{seidel2006biased}.
	\begin{lemma}[\cite{gutt2015positive},  \cite{seidel2006biased}]\label{lem: maximal_principle_param}
		Consider a decreasing homotopy $H_s: \widehat{W} \rightarrow \mathbb{R}$ such that $H_s = h_s(t)$ 
		is increasing in $t$ near $Z = \partial V$ and $H_s|_Z$ is $s$-independent; let $J \in \mathcal{J}_{std}(W)$ be cylindrical near $Z$. If 
		$u: \mathbb{R}\times S^1 \rightarrow \widehat{W}$ 
		is a $(H_s, J)$-Floer trajectory with both asymptotes in $V$, then $u$  is contained in $V$. If $V = W$, the same claim also holds for a homotopy $J_s \in J_{std}(W)$  that is cylindrical near $Z$.
	\end{lemma}
	By applying the second part of Lemma \ref{lem: maximal_principle_param}, we can conclude that $\mathcal{M}(x_-, x_+; H_s, J_s)$ has a codimension one compactification. The continuation map 
\[
	\phi_{H_s, J_s}: SC(H_+, J_+) \rightarrow SC(H_-, J_-)
	\]
	is defined by
\[
	\phi_{H_s, J_s}(x_+) = \sum_{x_-} \#\mathcal{M}_h(x_-, x_+; H_s, J_s) x_-\pmod 2.
\]
	 This map is independent of $J_s$ and $H_s$, up to chain homotopy. 
	Notice that there is no $\mathbb{R}$-action since the parametrized Floer equation is not $\mathbb{R}$-invariant. As a result, $\phi_{H_s, J_s}$ is \textit{degree-preserving}.  
	Now, we can define symplectic homology as the direct limit (taken over continuation maps 
	$\phi_{H_s, J_s}: HF(H_+, J_+) \rightarrow HF(H_-, J_-)$):
\[
	SH(W, \lambda):= \lim_{\rightarrow} HF(H, J). 
\]
	It is worth mentioning that $SH(W, \lambda)$ depends only on the symplectomorphism type of $(\widehat{W}, d\hat{\lambda})$
     \cite{seidel2006biased}.

	\subsection{Positive symplectic homology}\label{ssec: postive_sym_hom}
	 
	For a small time-dependent perturbation of 
	$H\in\mathcal{H}_{std}(W)$, the action functional $A_H: C^\infty(S^1, \widehat{W}) \rightarrow \mathbb{R}$ is 
	$$
	A_H(x) :=  \int_{S^1} x^* \lambda - \int_{S^1} H(x(t)) dt.
	$$
	Under our conventions, the Floer equation is the \textit{positive} gradient flow of the action functional which means if $u\in \mathcal{M}(x_-, x_+)$ is a non-constant Floer trajectory, then $A_H(x_+) > A_H(x_-)$. Let $SC^{<a}(H, J)$ be generated by orbits of action less than $a$. Since action increases along Floer trajectories, the differential decreases action and therefore
	$SC^{<a}(H, J)$ is a subcomplex of $SC(H, J)$.
	 we define 
	 \[SC^{>a}(H, J):=SC(H, J)/ SC^{<a}(H, J).
	 \] 
	For $H\in \mathcal{H}_{std}(W)$, the constant orbits corresponding to Morse critical points $p \in W$ have action $-H(p)$.
	The non-constant orbits corresponding to Reeb orbits 
	have \emph{positive} action close to the action of the corresponding Reeb orbit.
	Indeed, for sufficiently small $\epsilon$, $SC^{< \epsilon}(H, J)$ corresponds to the Morse complex of $-H|_W$(with a grading shift). To be specific, 
	\[H_k(SC^{< \epsilon}(H, J)) \cong H^{n-k}(W; \mathbb{Z}).
	\]
	Let's define $SC^+(H, J):=SC(H, J)/SC^{< \epsilon}(H, J)$ to be the quotient complex 
	and  $HF^+(H, J)$ the resulting homology. We can also define 
	$HF^+(W)$ by a  direct limit construction. 
	
	More precisely, suppose   $H_s$ satisfies:  
	\begin{itemize}
		\item $H_s$ is a decreasing homotopy,
		\item  $H_s = H_+,\, s\gg 0$,
		\item  $H_s = H_-,\, s\ll 0$.
	\end{itemize}Then the continuation Floer trajectories are also action increasing and induce chain map 
	\[\phi_{H_s, J_s}^+: SC^+(H_+, J_+) \rightarrow SC^+(H_-, J_-).\]
    We define $SH^+(W)$ by
	\begin{equation}
	SH^+(W, \lambda):= \lim_{\rightarrow} HF^+(H, J).
	\end{equation}
	The direct limit is taken over the  continuation maps \[
	\phi_{H_s, J_s}^+: HF^+(H_+, J_+) \rightarrow HF^+(H_-, J_-)
	\] on homology.

	$SC^+(H, J)$ is essentially dependent only on $(Y, \alpha)$ and not on the interior $(W, \lambda)$. This is due to the fact that $SC^+(H,J)$ is generated by non-constant Hamiltonian orbits, which live in the cylindrical end of $W$ and correspond to Reeb orbits of $(Y, \alpha)$.
	On the other hand, the differential for $SC^+(H, J)$ may depend on the filling $W$ of $(Y, \alpha)$ since Floer trajectories between non-constant orbits may go into the filling, so different Liouville fillings of $(Y, \xi)$ might have different $SH^+$. 

	The short exact sequence on chain-level 
	\begin{equation}
	0 \rightarrow SC^{< \epsilon}(H, J) \rightarrow SC(H, J) \rightarrow SC^+(H, J) \rightarrow 0
	\end{equation}
	induces  ``tautological" long exact sequence in homology
	\begin{equation}
	\cdots \rightarrow H^{n-k}(W; \mathbb{Z}) \rightarrow 
	SH_k(W, \lambda) \rightarrow SH_k^+(W, \lambda) \rightarrow 
	H^{n-k+1}(W; \mathbb{Z})\rightarrow \cdots .
	\end{equation}

	\subsection{Summary of the TQFT structure on $SH_*(W)$}\label{Subsection TQFT on SH summary}
	This is taken out of chapter 6 in \cite{ritter2013topological}. For a detailed construction, see chapter 16 of  \cite{ritter2013topological}. Note that both the grading and action functional differ from ours by a negative sign, and our homology $SH_*(W^{2n})$ is cohomology $SH^*(W^{2n})$ in  \cite{ritter2013topological}.
    We summarize here the TQFT structure. Suppose we are given:
	\begin{enumerate}
		\item\label{TQFTitem1}\label{TQFTitem2} a Riemann surface $(S,j)$ with $p+q$ punctures, with fixed complex structure $j$;
		\item\label{TQFTitem3} \emph{ends}: a cylindrical parametrization $s+it$ near each puncture, with $j\partial_s = \partial_t$;
		\item\label{TQFTitem1b} $p\geq 1$ of the punctures are \emph{negative} (i.e, we converge to the puncture as $s\to -\infty$), they are indexed by $a=1,\ldots, p$; 
		\item\label{TQFTitem1c} $q\geq 0$ of the punctures are \emph{positive} (i.e, we converge to the puncture as $s\to +\infty$), they are indexed by $b=1,\ldots,q$; 
		\item\label{TQFTitem4} \emph{weights}: constants $A_a,B_b>0$ satisfying $\sum A_a - \sum B_b\geq 0$;
		\item\label{TQFTitem5} a $1$-form $\beta$ on $S$ with $d\beta \leq 0$, and on the ends $\beta=A_a\,dt$, $\beta=B_b\,dt$ for large $|s|$.
	\end{enumerate}
	\begin{remark}
 Negative/positive parametrizations are modelled on $(-\infty,0]\times S^1$ and $[0,\infty)\times S^1$, respectively. In (\ref{TQFTitem5}), $d\beta\leq 0$ means $d\beta(v,jv)\leq 0$ for all $v\in TS$. By Stokes' theorem, $\sum A_a - \sum B_b = -\int_S d\beta \geq 0$. This forces $p\geq 1$ and (\ref{TQFTitem4}). Subject to this inequality, such $\beta$ exists. See Lemma 16.1  \cite{ritter2013topological}.
	\end{remark}

 Fix a Hamiltonian $H:\widehat{W}\to \R$ linear at infinity with $H\geq 0$ (required in Section 16.3  \cite{ritter2013topological}), this defines $X=X_H$. Fix an almost complex structure $J$ on $W$ of contact type at infinity.
	
	The moduli space $\mathcal{M}(x_a; y_b;S,\beta)$ of \emph{Floer
		solutions} consists of smooth maps $u:S\to \widehat{W}$ such that 
	$du-X\otimes \beta$ is $(j,J)$-holomorphic, and $u$ converges on the ends to
	$1$-orbits $x_a,y_b$ of $A_a H$, $B_b H$ which we call the \emph{asymptotics}. 
	
	After a small generic $S$-dependent perturbation $J_{z}$ of $J$, $\mathcal{M}(x_a; y_b;S,\beta)$ is a smooth manifold. One can ensure that on the ends $J_z$ does not depend on $z\!=\!s+it\!\in\! S$ for $|s|\gg 0$. Just as for Floer continuations maps (\ref{sssec: continuation_map}), a maximum principle and an a priori energy estimate $E(u) =  \sum \mathcal{A}_{B_b H}(y_b)-\sum \mathcal{A}_{A_a H}(x_a)$ holds, so the $\mathcal{M}(x_a; y_b;S,\beta)$ have compactifications by broken Floer solutions: Floer trajectories for $A_a H,B_bH$ can break off at the respective ends. When gradings are defined (\ref{subsection:index}), 
	\begin{align}
	\dim \mathcal{M}(x_a;y_b;S,\beta) &= -\sum \mu_{CZ}(x_a) +\sum \mu_{CZ}(y_b)
	+n\chi(S)\\
    &=\sum \mu_{CZ}(y_b)-\sum \mu_{CZ}(x_a)+n(2-2g-p-q) 
	.\label{dimension of moduli space of product}
	\end{align}
	Define
	$\psi_S: \otimes_{b=1}^q SC_*(B_b H) \to \otimes_{a=1}^p SC_*(A_a
	H)$ on generators by counting isolated Floer solutions
\[
	\psi_S(y_1 \otimes\cdots \otimes y_q) = \sum_{u\in
		\mathcal{M}_0(x_a;y_b;S,\beta)} \epsilon_u
	\; x_1\otimes \cdots \otimes  x_p,
\]

	where $\epsilon_u \in \{ \pm 1 \}$ are orientation signs (In this paper we use $\Z_2$ coefficients, so these signs don't matter. In general, see Section 17 of  \cite{ritter2013topological}). Then extend $\psi_S$
	linearly. 
	
	  	The $\psi_S$ are chain maps. On homology,
		
		$ \psi_S: \otimes_{b=1}^q SH_*(B_b H) \to \otimes_{a=1}^p SH_*(A_a
		H) $
		
		is independent of the choices $(\beta,j,J)$ relative to the ends. Taking direct limits, we get induced maps:

		$$\psi_S: SH_*(W)^{\otimes q}
		\to SH_*(W)^{\otimes p} \qquad (p\geq 1, q\geq 0).$$
		So $SH_*(W)$ has a unit $\psi_C(1)$.

		%
		%%%%%%%%%%%%%%%%%%%%%%%%%%%%%%%%%%%%%%%%%%%%%%%%%%%%%%%%%%%
		%%%%%%%%%%%%%%%%%%%%%%%%%%%%%%%%%%%%%%%%%%%%%%%%%%%%%%%%%%%
		\subsubsection{The product}\label{Subsection Product}\label{Section Ring structure}
	  \begin{figure}
		\centering
		\begin{tikzpicture}

		\draw  (5,2)ellipse (0.4 and 1);
		\draw  (5,-2)ellipse (0.4 and 1);
		\draw  (0,0 )ellipse (0.4 and 1);
		\draw node at (2,-2) {$P$};
		
		\draw  plot[smooth, tension=.7] coordinates {(0,1)(2.5,1.5) (5,3)};
		
		\draw  plot[smooth, tension=.7] coordinates {(0,-1)(2.5,-1.5) (5,-3)};
		\draw  plot[smooth, tension=.7] coordinates {(5,1)(4,0.5)(4,-0.5) (5,-1)};
		\end{tikzpicture}
		\caption{Pair of pants product: the operation
			$\psi_P$ receives inputs at positive punctures of $P$ and emits output at the negative puncture. So it goes ``from right to left".}\label{The pair of pants p }
	\end{figure}
		The pair of pants surface $P$ 
		(Figure \ref{The pair of pants p })
		defines the product
		\\[1mm]
		\begin{tabular*}{\textwidth}{l@{\extracolsep{\fill}}cr@{\extracolsep{0pt}}} 
			\strut & 
			$
			\psi_P: SH_i(W)\otimes SH_j(W) \to SH_{i+j}(W),\; x \cdot y =
			\psi_P(x,y),
			$
			& \strut 
		\end{tabular*}
		which is graded-commutative and associative. 
        \begin{remark}\label{NOproduct}
        	The pair of pants product also respects the action filtration. As mentioned in   \cite{uebele2015periodic} and  in Section 16.3 of  \cite{ritter2013topological}, we have
        	\[\mathcal{A}_{2H}(x_3)\leq \mathcal{A}_{H}(x_1)+\mathcal{A}_{H}(x_2).
        	\]
        	Hence the product restricts to a map
        	\[SH_*^{[a,b)}(W)\times SH_*^{[a',b')}(W)\rightarrow SH_*^{[\max\{a+b',a'+b\},b+b')}(W),
        	\]
        	where on the right hand side it is necessary to divide out all generators with action less than $\max\{a+b',a'+b\}$ to make the map well defined. So one does not get a product on the whole positive symplectic homology, but we can define maps:
        	\[ SH_*^{[\delta,b)}(W)\times SH_*^{[\delta,b)}(W)\rightarrow SH_*^{[b+\delta,2b)}(W)
        	\]
        \end{remark}
		\subsubsection{The unit}
		\label{Subsection Definition of the unit}
		%%%%%%%%%%%%%%%%%%%%%%%%%%%%%%%%%%%%%%%%%%%%%%%%%%%%%%%%%%%
		%%%%%%%%%%%%%%%%%%%%%%%%%%%%%%%%%%%%%%%%%%%%%%%%%%%%%%%%%%%
		%
		Let $C=\C$ with $p=1$, $q=0$. The end is parametrized by
		$(-\infty,0]\times S^1$ via $s+it \mapsto e^{ - 2\pi(s+it) }$. On
		this end, $\beta=f(s)dt$ with $f'(s)\leq 0$, $f(s)=1$ for
		$s\leq -2$ and $f(s)=0$ for $s\geq -1$. Extend by $\beta=0$ away
		from the end (See Figure~\ref{Figure Unit}). Thus we get a map $\psi_C: \K \to SH_*(H)$.

		\begin{figure}[ht]
		\centering
		\begin{tikzpicture}
		\draw  (0,0)ellipse (0.35 and 1);
		\draw node at(1,-1.25) {$C$};
		\draw  plot[smooth, tension=.7] coordinates {(0,1)(1,0.9)(2,0.8)(3,0.5)(3,-0.5)(2,-0.8)(1,-0.9)(0,-1)};
		\draw  plot[smooth, tension=.7] coordinates {(1,0.9) (1.18,0)(1,-0.9)};
		\draw plot[smooth,tension=0.6]coordinates{(2,0.8)(2.15,0)(2,-0.8)};
		\draw (0.26,0.5)--(2.1,0.45);
		\draw (0.26,-0.5)--(2.1,-0.45);
		\draw (0.35,0)--(2.15,0);
		\draw [-latex] (-1,-1)node[left] {$\beta=dt$}--(0.5,-0.7);
		\draw [-latex] (-0.5,-0.35)node[left] {$\beta=f(s)dt$}--(1.5,-0.2);
		\draw [-latex] (-0.5,0.35)node[left] {$\beta=0$}--(2.45,0.27);
		\draw [-latex] (2,1.5)node[right] {Floer's equation for $H$}--(0.5,0.7);
		\draw [-latex] (2,-1.5)node[right] {Floer's continuation equation $f(s)\cdot H$}--(1.5,-0.6);
		\draw [-latex] (4,0)node[right] {$J-$holomorphic since $0\cdot H=0$}--(2.5,0);
		\end{tikzpicture}
		\caption{A cap $C$, and its interpretation as a continuation cylinder.}\label{Figure Unit}
		\end{figure}

		\begin{defn}
		Let $e_H\!=\!\psi_{C}(1)\!\in\! SH_n(H)$. We can define $e\! =\! \varinjlim e_H \!\in\! SH_n(W).$
		\end{defn}
	
		\begin{theorem}[Theorem 6.1  \cite{ritter2013topological}]\label{Theorem unital ring structure}
			$e$ is the unit for the production on $SH_*(W)$.
		\end{theorem}
		\begin{proof}
		
			By the gluing illustrated in the Figure~\ref{Fig:unit for pair of pants}, $\psi_P(e,\cdot) = \psi_{P\# C}(\cdot) = \psi_Z(\cdot)=\textrm{id}$.
		\end{proof}
	\begin{figure}[ht]
		\centering
		\begin{tikzpicture}
		\draw (-4,0) ellipse (0.25cm and 0.8cm);
		\draw (0,-1.5) ellipse (0.25cm and 0.8cm);
		\draw (0,1.5) ellipse (0.25cm and 0.8cm);
		
		\draw  plot[smooth, tension=.75] coordinates {(-4,-0.8) (-2,-1) (0,-2.3)};
		\draw  plot[smooth, tension=.75] coordinates {(-4,0.8) (-2,1) (0,2.3)};
		
		\draw  plot[smooth, tension=1] coordinates {(0,0.7) (-1,0) (0,-0.7)};
		
		\draw  plot[smooth, tension=.7] coordinates {(0,0.7) (1.2,1.15)(1.2,1.85) (0,2.3)};
		
		\draw (4,0) ellipse (0.25cm and 0.8cm);
		\draw (8,-1.5) ellipse (0.25cm and 0.8cm);
		
		\draw  plot[smooth, tension=.75] coordinates {(4,-0.8) (6,-1) (8,-2.3)};
		
		\draw  plot[smooth, tension=.75] coordinates {(4,0.8) (5.5,0.9) (7,1.5)(7.5,1) (7,-.04) (8,-0.7)};
		\end{tikzpicture}
		\caption{Unit for pair of pants product}\label{Fig:unit for pair of pants}
	\end{figure}
			\begin{remark}
			 For ``gluing = compositions'' results, see Theorems 16.10, 16.12, 16.14 in  \cite{ritter2013topological}. Before taking direct limits, the above is the continuation map \[SH_*(H) \stackrel{e_H\otimes \cdot}{\longrightarrow} SH_*(H)^{\otimes 2} \stackrel{\psi_P}{\longrightarrow} SH_*(2H).\]
			\end{remark}

		\begin{lemma}[Lemma    6.2  \cite{ritter2013topological}]\label{Lemma unit is a count of continuation
				solutions} $e_H$ is a count of the isolated finite energy Floer
			continuation solutions $u:\R\times S^1 \to \widehat{W}$ for the
			homotopy $f(s)H$ from $H$ to $0$.
		\end{lemma}

		\begin{lemma}[Lemma 6.3  \cite{ritter2013topological}]\label{Lemma unit is sum of minima}
			For $H$ as in Section \ref{sss:Ad Hamiltonian}, $e_H = $ sum of the local minima of $H$.
		\end{lemma}

		\begin{theorem}[Theorem 6.4  \cite{ritter2013topological}]\label{Theorem unit is image of 1}
			$e=\varinjlim e_H$ is the image of $1$ under $c_*:H_*(W) \to
			SH_*(W)$, and $e_H = c_{*,H}(1)$ where $c_{*,H}:H_*(M)\cong SH_*^{<\delta}(H) \to SH_*(H)$ is the inclusion map.
		\end{theorem}

		\subsubsection{The TQFT structure on ${SH_*(W)}$ is compatible with the grading by ${H_1(W)}$}
		\label{Subsection TQFT is compatible with filtrations}
		%%%%%%%%%%%%%%%%%%%%%%%%%%%%%%%%%%%%%%%%
		%
		We can grade $SC_*(H)=\bigoplus\limits_{h\in H_1(W)} SC^{h}_*(H)$ by the homology classes
		$h\in H_1(\widehat{W})$ of the generators. The Floer
		differential preserves the $H_1$ grading, and so do Floer operations
		on a cylinder and a cap. The pair of pants product respects this grading as follows: $\psi_S: SH^{h_1}_*(W)\otimes SH^{h_2}_*(W) \to
		SH^{h_1+h_2}_*(W)$.
		We can also grade $SH_*(W)=\bigoplus\limits_{h}SH^{h}_*(M)$ by the free homotopy classes
		$h\in [S^1,M]$ of the generators. The TQFT operations for genus zero surfaces are compatible with the grading (the equation above holds after
		replacing $\sum$ by concatenation of free loops).
		\begin{remark}\label{remark on pair of pants product about contractible loops}
			Let $SH_*^0(W)$ denote the summand corresponding to the contractible loops. Considering only contractible loops determines a TQFT with operations $\psi_S:SH_*^0(W)^{\otimes q} \to
		SH_*^0(W)^{\otimes p}$ $(p\geq 1,q\geq 0)$. Also $c_*:H_*(W)\to SH_*^0(W)\subset SH_*(W)$ naturally lands in $SH_*^0(W)$. 
		\end{remark}
	
		%
		%
		%%%%%%%%%%%%%%%%%%%%%%%%%%%%%%%%%%%%%%%%
		%%%%%%%%%%%%%%%%%%%%%%%%%%%%%%%%%%%%%%%%

		\subsubsection{Viterbo Functoriality}
		\label{Subsection Viterbo Functoriality}
		%%%%%%%%%%%%%%%%%%%%%%%%%%%%%%%%%%%%%%%%%%%%%%%%%%%%%%%%%%%
		%
		%
		For Liouville subdomains $W\subset \widehat{M}$, Viterbo  \cite{viterbo1999functors}
		constructed a restriction map $SH_*(M)\to SH_*(W)$ and McLean
		 \cite{mclean2007lefschetz} proved that it is a ring homomorphism.
		 \begin{comment}
		  Ritter proved a stronger statement:
		
		\begin{theorem}[Theorem 9.5  \cite{ritter2013topological}]\label{Theorem Viterbo Functoriality}
			Let $ i:(W,\theta_W) \hookrightarrow (\widehat{M},\theta) $ be a Liouville
			subdomain. Then there exists a $\mathrm{restriction\; map}$,
			%
			%
			$SH_*(i)_{\eta}:SH_*(M)_{\eta} \to SH_*(W)_{i_*\eta},$
			%
			%
			which is a TQFT map fitting into a commutative diagram which respects TQFT
			structures:
			%
			%
			$$
			\xymatrix@R=12pt{SH^*(W) \ar@{<-}[r]^-{SH_*(i)} \ar@{<-}[d]_-{c_*}
				& SH_*(M) \ar@{<-}[d]^-{c_*} & & SH_*(W)_{i_*\eta}
				\ar@{<-}[r]^-{SH_*(i)_{\eta}} \ar@{<-}[d]^-{c_*} &
				SH_*(M)_{\eta} \ar@{<-}[d]_-{c_*} \\
				H_*(W) \ar@{<-}[r]^-{i_*} & H_*(M) & & H_*(W)\otimes \Lambda
				\ar@{<-}[r]^-{i_*} & H_*(M)\otimes \Lambda}
			$$
			%
			In particular, all maps are unital ring homomorphisms.
		\end{theorem}
	Therefore we have the following corollary:
		\end{comment}
		\begin{theorem}[\cite{mclean2007lefschetz} \cite{cieliebak2018symplectic}]\label{invariant of SH}
			Let $W$ and $V$ be compact symplectic manifolds with contact type boundary and assume that the Conley-Zehnder index is well-defined on $W$. If $V$ is obtained from $W$ by attaching to $\partial W\times [0,1]$ a subcritical symplectic handle $H_k^{2n}$, $k<n$, then it holds that
			\[SH_*(V,\Z_2)\cong SH_*(W,\Z_2) \]
			as rings.
		\end{theorem}
		\begin{remark}
		A.Ritter proved a stronger statement in Theorem 9.5 of \cite{ritter2013topological}.
		\end{remark}

		%%%%%%%%%%%%%%%%%%%%%%%%%%%%%%%%%%%%%%%%%%%%%%%%%%%%%%%%
		
		\subsection{Conley-Zehnder index}\label{subsection:index}
		In this section we discuss Conley-Zehnder index as in Fauck \cite{fauck2016rabinowitz}. To define $\mu_{CZ}$, let $Sp(2n)$ denote the group of $2n\times2n$ symplectic matrices. We will discuss a generalization, called the Robbin-Salamon index as follows: any smooth path $\Psi:[a,b]\to Sp(2n)$ satisfies an ordinary differential equation
		\[\Psi'(t)=J_0S(t)\Psi(t),\qquad \Psi(a)\in Sp(2n),
		\] 
		Where $t\to S(t)=S(t)^{T}$ is a smooth path of symmetric matrices and $J_0$ is the standard almost complex structure. We say $t\in [a,b]$ is called a crossing if $\det(id-\Psi(t))=0$. The crossing form at time $t$ is a quadratic form $\Gamma(\Psi,t)$ defined for $v\in \ker(id-\Psi(t))$ by
		\[\Gamma(\Psi,t)v=<v,S(t)v>
		\]
		
		A crossing $t$ is called regular if $\Gamma(\Psi,t)$ is non-degenerate. For a path with only regular crossings, the Robbin-Salamon index is defined by
		\[\mu_{CZ}(\Psi,a,b):=\frac{1}{2}sign \Gamma(\Psi,a)+\sum_{a<t<b} sign \Gamma(\Psi,t)+\frac{1}{2} sign \Gamma(\Psi,b)
		\]
		where the sum runs all over crossings $t\in(a,b)$, and sign($M$) denotes the signature of the matrix $M$, which equals the number of positive eigenvalues minus the number of negative eigenvalues. Here we use $\mu_{CZ}$ to denote the Robbin-Salamon index.The fact that the Robbin-Salamon index coincides with Conley-Zehnder index when $\det(id-\Psi(b))\ne 0$ sort of justifies this abuse of notation.
		
		We have the following properties for $\mu_{CZ}$:
		\begin{itemize}
			
			\item (\textit{Naturality}) For any path $\Phi:[a,b]\to Sp(2n)$, $\mu_{CZ}(\Phi\Psi\Phi^{-1})=\mu_{CZ}(\Psi)$
			\item (\textit{Homotopy}) $\mu_{CZ}(\Psi_s)$ is constant for any homotopy $\Psi_s$ with fixed endpoints. 
			\item (\textit{Product}) If $Sp(2n)\oplus Sp(2n')$ is identified with a subgroup of $Sp(2(n+n'))$ in the natural way, then $\mu_{CZ}(\Psi\oplus \Psi')=\mu_{CZ}(\Psi)+\mu_{CZ}(\Psi').$
		\end{itemize}
		
		The homotopy property allows us to define $\mu_{CZ}(\Psi,a,b)$ also for paths with non-regular crossings, given that having regular crossings is a $\mathcal{C}^\infty$ generic property among paths with fixed endpoints.
		
		\begin{remark}[Lemma 59 \cite{fauck2016rabinowitz}]\label{Index formula}
			Let $\Psi_1,\Psi_2,\Psi_3:[0,T]\to Sp(2)$ be the following paths:
			\[\Psi_1(t)=e^{it}, \quad,\Psi_2(t)=e^{-it},\quad \Psi_3(t)=\textrm{diag}\big(e^{f(t)},e^{-f(t)}\big),f\in C^1(\mathbb{R}).
			\]
			Then, their Conley-Zehnder indices are given as follows:
			\begin{align*}
			\mu_{CZ}(\Psi_1)&=\Bigg\lfloor\frac{T}{2\pi}\Bigg\rfloor+\Bigg\lceil\frac{T}{2\pi}\Bigg\rceil,\\
			\mu_{CZ}(\Psi_2)&=\Bigg\lfloor\frac{-T}{2\pi}\Bigg\rfloor+\Bigg\lceil\frac{-T}{2\pi}\Bigg\rceil=-\mu_{CZ}(\Psi_1),\\
			\mu_{CZ}(\Psi_3)&=0.
			\end{align*}
		\end{remark}
		%%%%%%

		\subsubsection*{Trivialization}
		Suppose we have a symplectic manifold $(M,\omega)$ with $c_1(M)=0$ and $J$ is an $\omega-$compatible almost complex structure. Then the anti-canonical bundle of $M$ is the highest exterior power of $(TM,J)$, i.e,
	 $\kappa_J^*=\wedge^n(TM,J)$. The canonical bundle $\kappa_J$ is  the dual of $\kappa_J^*$. In the same manner, we can define the canonical bundle of a contact manifold $(C,\xi)$ with a choice of one form $\alpha$ and $d\alpha$-compatible almost complex structure on $\xi$. 
		
		A \textit{trivialization of the canonical bundle} is a bundle isomorphism $\Phi:\kappa_J\to M\times \mathbb{C}$. A \textit{trivialization} of $(\gamma^*TM,J)$( where $\gamma$ is a loop in $M$) is a bundle isomorphism $\Psi:\gamma^*TM\to S^1\times \mathbb{C}^n$. Such a trivialization has a one-to-one correspondence (up to homotopy) with the trivialization of $\gamma^{*}\kappa_J^*$ and hence the trivialization of the canonical bundle via:
		\[\det\nolimits_\mathbb{C}(\Psi):\wedge^n(\gamma^*TM)=\gamma^*\kappa_J^*\rightarrow S^1\times \mathbb{C}.
		\]
		
		For a 1-periodic Hamiltonian orbit $x$, we fix a trivialization of $(x^*TM,J)$ along $x$ as:
		\[\Psi:x^*TM\to S^1\times \mathbb{C}^n
		\]
		Suppose $\psi$ is the Hamiltonian flow and $d\psi_t:TM|_{x(0)}\to TM|_{x(t)}$ is its linearization, then define \[M_t(x):=\Psi_t\circ d\psi_t\circ \Psi_0^{-1}
		\]
		The Conley-Zehnder index of $x$ is defined as $\mu_{CZ}(x):=\mu_{CZ}(M_t(x))$.
		Similarly, we can define the Conley-Zehnder index of a Reeb orbit. In particular, let $(W,\lambda)$ be a Liouville domain and $(C:=\partial W,\xi:=\ker \lambda|_C)$ its boundary. We have $TM|_C=\xi\oplus<X_{Reeb}>\oplus<X>$, where $X_{Reeb},X$ are a Reeb vector field and a Liouville vector field, respectively. Since $<X_{Reeb}>=J<X>$ , we can identify $<X_{Reeb}>\oplus<X>$ with $\mathbb{C}$, i.e.
		$\gamma^*TM=\gamma^*\xi\oplus\mathbb{C}$, where $\gamma$ is a Reeb orbit.
		Due to the fact that Reeb flow preserves $X_{Reeb}$ and extends to the symplectization, we have 
		\[M_{t,M}(\gamma)=M_{t,\xi}(\gamma)\oplus \left[ {\begin{array}{cc}
			1 & 0 \\
			0 & 1 \\
			\end{array} } \right]
		\]
		where $M_{t,M}(\gamma)$ is the symplectic matrix associated to the linearization of the Reeb flow 
		with respect to a trivialization of $TM|_\gamma$, i.e,  $M_{t,\xi}(\gamma)$ is defined in the same manner. The product property of Conley-Zehnder index implies that $\mu_{CZ}(M_{t,M}(\gamma))=\mu_{CZ}(M_{t,\xi}(\gamma))$. Hence we will not specify which index we are referring to in the rest of this paper.

		Now consider a $G-$equivariant Liouville domain $(W,\lambda)$. Suppose the group action is free and $|G|<\infty$. Then we have that the quotient map
		\[\pi_G: (W,\lambda)\rightarrow (W/G,\lambda)
		\]
		is a finite covering map. Each Reeb orbit $\gamma$ in $\p(W/G)$ then lifts to a fractional orbit $\widetilde{\gamma}$ in $\p W$. That is,   $\widetilde{\gamma}(t)=\gamma_0(t),t\in[0,T]$ for some closed Reeb orbit $\gamma_0$ in $\p(W)$. In particular, we can choose $\gamma_0$ with period of $|G|\cdot T$. 
		If we choose a $G-$equivariant trivialization for the canonical bundle $\kappa_W$, then such trivialization descends down to $\kappa_{W/G}$. Equivalently, if we choose $G$-equivariant trivialization of $\xi|_{\gamma_0}$, and $M_{t,\xi}(\gamma_0)$ is the matrix of the linearized map, then we have for some $M_G\in Sp(2n,\R)$,
		\[ M_G\cdot M_{t,\xi}(\gamma_0)=M_{t+T,\xi}(\gamma_0).
		\]
		where $M_G$ satisfies $M_G^{|G|}=M_{|G|\cdot T,\xi}(\gamma_0)$  is a constant matrix, which only depends on the homotopy class of our $G$-equivariant trivialization. In particular, $M_{T,\xi}(\gamma_0)=M_G$,
		so $\mu_{CZ}(M_{t,\xi}(\gamma_0)), t\in [0,T]$ is well defined since the
		Conley-Zehnder index is constant for any homotopy with fixed endpoints. We can therefore define the Conley-Zehnder index of such a fractional Reeb orbit
	 of $\gamma$  to be the Conley-Zehnder index of $M_{t,\xi}(\gamma_0), t\in[0,T]$.

		 As a consequence, we have
		\[\mu_{CZ}(\gamma)=\mu_{CZ}(\widetilde{\gamma}).
		\]

			\begin{lemma}\label{index for the cylinder component}
	Let $(\mathbb{R}\times S^1, d(rd\theta))$ be the symplectization of $(S^1,\theta)$.
	Choose the canonical trivialization of $T(\mathbb{R}\times S^1)=T\mathbb{R}\times TS^1$, then 
	all fractional Reeb orbits of  $(S^1,\theta)$ have Conley-Zehnder index (Robbin-Salamon index) zero, with respect to any cyclic group action rotating the cylinder.
\end{lemma}
\begin{proof}
	Since Reeb flow preserves $(\partial_r,\partial_\theta)$, so the matrix for linearized return map is\[ M(t)=\left[ {\begin{array}{cc}
		1 & 0 \\
		0 & 1 \\
		\end{array} } \right].
	\]Therefore, the Conley-Zehnder index is zero.
\end{proof}		
		
		The following lemma gives a formula for the Reeb vector field in terms of the Hamiltonian and Liouville vector field. 
		\begin{lemma}\label{formula for reeb vector}
			Let $(W,\lambda)$ be a Liouville manifold. Suppose $H$ is a function on $W$ with $0$ as its regular value and the Liouville vector field $X$ is transverse to the 0-level set. Then $(\Sigma:=H^{-1}(0),\lambda)$ is a contact manifold whose Reeb vector field is given by $X_{Reeb}=\frac{X_H}{X(H)}$, where $X_H$ is the Hamiltonian vector field of $H$.
		\end{lemma}
		\begin{proof}
			$(\Sigma,\lambda)$ is well known to be contact. We only need to prove the latter part of the lemma. Since 
			\[\iota_{X_H}d\lambda|_{\Sigma}=-dH|_\Sigma=0\]
			and
			\[\iota_{X_H}\lambda=\iota_{X_H}\iota_{X}d\lambda=d\lambda(X,X_H)=dH(X)=X(H),
			\]
			it follows $X_{Reeb}=\frac{X_H}{X(H)}$.
		\end{proof}
\begin{lemma}[Lemma 5.20  \cite{mclean2016reeb}] \label{lemma:hamiltonianconleyzehnderindexcomparisonstandard}
	Let $(C,\xi)$ be a contact manifold with associated contact form $\alpha$
	and let $h : \R \to \R$ be a function with $h' > 0,h''>0$ and $h'(0) = 1$.
	Let $\widehat{C} := C \times \R$ be the symplectization of $C$ with symplectic form $d(e^r \alpha)$
	where $r$ parameterizes $\R$.
	Let $\gamma(t)$ be a Reeb orbit of $\alpha$ of period $L$ with a choice of  trivialization of the symplectic vector bundle $\oplus_{j=1}^{N}TM$ along this orbit.
	This choice of trivialization induces a choice of trivialization of $\gamma^* \oplus_{j=1}^{N} \xi$ in a natural way.
	Then the Hamiltonian $L h(e^r)$ has a $1$ periodic orbit x equal to $\gamma(Lt)$ inside $C \times \{0\} = C$ and its Conley-Zehnder index is equal to
	$\mu_{CZ}(\gamma)+\frac{1}{2}$.
	
\end{lemma}\label{remark: index difference}
	\begin{remark}
    Notice that the Hamiltonian vector field in  \cite{mclean2016reeb} differs from ours by a minus sign.
	We have 	\[M_{t,M}(x)=M_{t,\xi}(\gamma)\oplus \left[ {\begin{array}{cc}
		1 & 0 \\
		ah''t & 1 \\
		\end{array} } \right],
	\]for some constant $a>0$.
	
	If instead, $h''< 0$, then the index equals $\mu_{CZ}(\gamma)-\frac{1}{2}$. And
	if $h'<0$, then the Hamiltonian orbit goes in the opposite direction of the Reeb orbit, and the index differs by a minus sign.
	\end{remark}

		We will conclude this subsection with a lemma relating Morse index of critical point with Conley-Zehnder index of the corresponding constant Hamiltonian orbit.
		\begin{lemma}
			If $S$ is an invertible symmetric matrix with $||S||<2\pi$ and $\Psi(t)=\exp(tJ_0S)$, then \[\mu_{CZ}(\Psi)=n-Ind(S)
			\]
			where $Ind(S)$ is the number of negative eigenvalues of $S$.
		\end{lemma}
		\begin{corollary}[Corollary 7.2.2\textbf{ \cite{audin2014morse}}]\label{index of critical pt}
			Let $W$ be a symplectic manifold of dimension $2n$, let
			\[H:W\rightarrow \mathbb{R}
			\]
			be a Hamiltonian and $x$ be a critical point of $H$. We assume that $H$ is $\mathcal{C}^2$-small (in this case, we can choose a Darboux chart centered at $x$ such that the usual norm $||Hess_x(H)||<2\pi$). Then the Conley-Zehnder index $\mu_{CZ}(x)$ of $x$ as a periodic solution of the Hamiltonian system and its Morse index Ind(x) as a critical point of the function $H$ are connected by
			\[\mu_{CZ}(x)=n-Ind(x).
			\]
		\end{corollary}

		\subsection{Weinstein handle attachment and contact surgery}\label{ss:Weinstein handle & contact surgery}
		\subsubsection{Contact surgery}
		This section is already included in Chapter 6 of  \cite{geiges2008introduction}. We will  highlight the parts which should be paid attention to in this paper, namely, the trivialization of the conformal symplectic normal bundle.
		\begin{defn}
			Let $(M,\xi)$ be a contact manifold. A submanifold $L$ of $(M,\xi)$ is called an \textit{isotropic submanifold} if $T_pL\subset \xi_p$ for all point $p\in L$.
		\end{defn}
		Let $L\subset (M,\xi=\ker \alpha)$ be an isotropic submanifold in a contact manifold with cooriented contact structure. Let $(TL)^\perp\subset \xi_L$ be the subbundle of $\xi_L$ that is symplectically orthogonal to $TL$ with respect to the symplectic bundle structure $d\alpha|_\xi$.The conformal structure of this bundle does not depend on the choice of contact form and therefore $(TL)^\perp $ is determined by $\xi$. The fact $L$ is isotropic implies that $TL\subset (TL)^\perp$. So we have the following definition,
		\begin{defn}
			The quotient bundle\[
			CSN_M(L):=(TL)^\perp/TL\]
			with the conformal symplectic structure induced by $d\alpha$ is called the \textit{conformal symplectic normal bundle} of $L$ in $M$.
		\end{defn}
		So we have 
		\[\xi|_L=\xi|_L/(TL)^\perp\oplus(TL)^\perp/TL\oplus TL=TL\oplus\xi|_L/(TL)^\perp\oplus
		CSN_M(L).
		\]
		Let $J:\xi\to \xi$ be a complex bundle structure on $\xi$ compatible with the symplectic structure given by $d\alpha$. Then the bundle $\xi|_L/(TL)^\perp$ is isomorphic to $J(TL)$. So the
		contact structure has the following natural splitting on the isotropic submanifold:
		\begin{lemma}\label{spliting of the contact structure}
			\[\xi|_L=TL\oplus J(TL)\oplus
			CSN_M(L)
			\]
		\end{lemma}
		Therefore, if we fix a trivialization of $TL\oplus J(TL)$, then the trivialization of $CSN_M(L)$ is determined by the trivialization of $\xi|_L$.
		Now we can state the contact surgery theorem:
		\begin{theorem}[Theorem 6.2.5 \cite{geiges2008introduction}]\label{thm for contact surgery}
			Let $\Lambda^{k-1}$ be an isotropic sphere in a contact manifold $(M,\xi=ker \alpha)$ with a trivialization of the conformal symplectic normal bundle $CSN_M(\Lambda^{k-1})$. Then there is a symplectic cobordism from $(M,\xi)$ to the manifold $M'$ obtained from $M$ by surgery along $\Lambda^{k-1}$ with the natural framing. In particular, the surgered manifold $M'$ carries a contact structure that coincides with the one on $M$ away form the surgery region.
		\end{theorem}
		
		\begin{remark}
			The resulting contact structure on $M'$ is uniquely determined up to isotopy by the isotopic isotropy class of $\Lambda^{k-1}$ and the homotopy class of the trivialization of $CSN_M(\Lambda^{k-1})$.
		\end{remark}

		\subsubsection{Weinstein handlebodies}\label{weinstein handlebody}
	 For the purposes of this paper, we need to attach a handle to a Weinstein domain.
		We will follow Section 13 in \cite{cieliebak2006symplectic}.
		The standard handle of index $k$ will be the bidisk in $\mathbb{C}^n$:
		\[\bigg\{\sum_{j=1}^{k}x_j^2\leq (1+\epsilon)^2,\sum_{j=1}^{k}y_j^2+\sum_{j=k+1}^{n}|z_j|^2\leq \epsilon^2\bigg\},
		\]
		where $z_j=x_j+iy_j,j=1,2,\cdots,n$, are the complex coordinates in $\mathbb{C}^n$. In particular, the handle $H$ carries the standard complex structure $i$, along with the standard symplectic structure $\omega_{std}$.
		The symplectic form $\omega_{std}$ on $H$ admits a hyperbolic Liouville field
		\[X_{std}=\sum_{j=1}^{k}\big(-x_j\frac{\partial}{\partial x_j} + 2y_j\frac{\partial}{\partial y_j}\big)+\frac{1}{2}\sum_{l=k+1}^{n}\big(x_l\frac{\partial}{\partial x_l}+y_l\frac{\partial}{\partial y_l}\big).
		\]
		Let us denote by $\xi^-$ the contact structure $\alpha_{st}|_{\partial^-H}=0$ defined on $\partial^- H$ by the Liouville form $\alpha_{st}=\iota_{X_{std}}\omega_{std}$, where $\partial^- H:=\partial D_1^k\times D_\epsilon^{2n-k}$ is the \textit{lower boundary}. Notice that the bundle $\xi^-|_{\Lambda^{k-1}}$ canonically splits as $T\Lambda^{k-1}\oplus J(T\Lambda^{k-1})\oplus \epsilon^{n-k}$, where $\epsilon^{n-k}$ is a trivial $(n-k)-$dimensional complex bundle. We will denote by $\sigma_\Lambda$ the isomorphism 
		\[T\Lambda^{k-1}\oplus J(T\Lambda^{k-1})\oplus \epsilon^{n-k}\to \xi^-|_\Lambda.
		\] 
		
		Suppose we are given a real $k-$dimensional bundle $E$, a complex $n-$dimensional bundle $\tau,n\geq k$, and an injective totally real homomorphism $\phi:E\to \tau$. Then $\phi$ canonically extends to a complex homomorphism $\phi\otimes \C: E\otimes\C \to \tau$. If $\phi\otimes \C$ extends to a fiberwise complex isomorphism $\Phi: E\otimes\C \oplus \epsilon^{n-k}$ then $\Phi$ is called a \textit{saturation} of $E$ covering $\phi$. When $n=k$ the saturation is unique.
		
		Let $(V,\omega,X,\phi)$ be a Weinstein manifold, $p$ a critical point of index $k$ of the function $\phi$, $a<b=\phi(p)$ a regular value of $\phi$. Denote $W:=\{ \phi \leq a\}$. Suppose that the stable manifold of $p$ intersects $V\setminus Int W$ along a disc $D^k$, and let $\Lambda^{k-1}=\partial D^k$ be the attaching sphere. The inclusion $T\Lambda^{k-1}\hookrightarrow \xi$ extends canonically to an injective complex homomorphism $T\Lambda^{k-1}\oplus J(T\Lambda^{k-1})\hookrightarrow \xi$, while the inclusion $TD^k\hookrightarrow TV$ extends to an injective complex homomorphism $TD^k\oplus J(TD^k)\hookrightarrow TV$. There exists a homotopically unique complex trivialization of the conformal symplectic normal bundle $CSN_{\partial W}(\Lambda^{k-1})$ in $\xi$ which extends to $D^k$ as a trivialization of the conformal symplectic normal bundle to $D^k$ in $TV$. This trivialization provides a canonical isomorphism $\Phi_{D^k}:T\Lambda\oplus J(T\Lambda)\oplus \epsilon^{n-k} \to \xi|_{\Lambda^{k-1}}$, and  we will call this the  canonical saturation of the inclusion $\Lambda^{k-1} \hookrightarrow \partial W$.
		
		We have the following theorem on attaching a handle to a Weinstein domain:
		\begin{theorem}[Prop13.11 \cite{cieliebak2006symplectic}, \cite{weinstein1991contact}]\label{weinstein handle attaching}
			Let $(W,\omega,X,\phi)$ be a $2n-$dimensional Weinstein domain with boundary $\partial W$ and $\xi$ the induced contact structure $\{\alpha|_{\partial W}=0\}$ on $W$ defined by the Liouville form $\alpha=\iota_X\omega$. Let $h:\Lambda\to \partial W$ be an isotropic embedding of the $(k-1)-$sphere $\Lambda$. Let $\Phi:T\Lambda\oplus J(T\Lambda)\oplus \epsilon^{n-k} \to \xi$ be a saturation covering the differential $dh:T\Lambda\to \xi$. Then there exists a Weinstein domain $(\tilde{W},\tilde{\omega},\tilde{X},\tilde{\phi})$ such that $W\subset Int  \tilde{W}$, and
			\begin{itemize}
				\item[(i)]  $(\tilde{\omega},\tilde{X},\tilde{\phi})|_W=(\omega,X,\phi)$;
				\item[(ii)] the function $\tilde{\phi}|_{\tilde{M}\setminus Int W}$ has a unique critical point $p$ of index $k$.
				\item[(iii)] the stable disc $D$ of the critical point $p$ is attached to $\partial W$ along the sphere $h(\Lambda)$, and the canonical saturation $\Phi_D$ coincides with $\Phi$.
			\end{itemize}
			
			Given any two Weinstein extensions $(W_0,\omega_0,X_0,\phi_0)$ and $(W_1,\omega_1,X_1,\phi_1)$ of $(W,\omega,X,\phi)$ which satisfy properties (i)-(iii), there exists a diffeomorphism $g$ fixed on $W$ such that $g:W_0\to W_1$ satisfying $(\omega_0,X_0,\phi_0)$ and $(g^*\omega_1,g^*X_1,g^*\phi_1)$ are homotopic in the class of Weinstein structures which satisfy (i)-(iii). In particular, the completion of these two Weinstein domains are symplectomorphic via a symplectomorphism fixed on $W$.
		\end{theorem}
		We say that the Weinstein domain $(\tilde{W},\tilde{\omega},\tilde{X},\tilde{\phi})$ is obtained from  $(W,\omega,X,\phi)$ by attaching a handle of index $k$ along an isotropic sphere $h:\Lambda\to \partial W$ with the given trivialization $\Phi$.

			\begin{defn}\label{def: weinstein_flexible}
			A Weinstein domain $(W^{2n}, \lambda, \phi)$ is  \textit{flexible} if there exist regular values $c_1, \cdots, c_{k}$ of $\phi$ such that $c_1 < \min \phi < c_2 < \cdots < c_{k-1} < \max \phi < c_{k}$ and for all $i = 1, \cdots, k-1$, $\{c_i \le \phi \le c_{i+1} \}$ is a Weinstein cobordism with a single critical point $p$ whose attaching sphere $\Lambda_p$ is either subcritical or a loose Legendrian in $(Y^{c_i}, \lambda|_{Y^{c_i}})$.  
		\end{defn} 
		
		Flexible Weinstein \textit{cobordisms} are defined similarly. Also, a Weinstein handle attachment or contact surgery is called flexible if the attaching Legendrian is loose. 
		So any flexible Weinstein domain can be constructed by iteratively attaching subcritical or flexible handles to $(B^{2n}, \omega_{std})$.  A Weinstein domain that is Weinstein homotopic to a Weinstein domain  satisfying Definition \ref{def: weinstein_flexible} will also be called flexible.
		Loose Legendrians have dimension at least $2$  
		so if $(Y_+,\xi_+)$ is the result of flexible contact surgery on $(Y_-, \xi_-)$, then by Proposition \ref{prop: c1equivalence}
		$c_1(Y_+)$ vanishes if and only if $c_1(Y_-)$ does.
		Finally, we note that subcritical domains are automatically flexible.

		Since they are built using loose Legendrians and subcritical spheres, which satisfy an h-principle, flexible Weinstein domains also satisfy an h-principle  \cite{cieliebak2012stein}.
		Again, the h-principle has an existence and uniqueness part:
		\begin{itemize}
			\item any almost Weinstein domain admits a flexible Weinstein structure in the same almost symplectic class
			\item any two flexible Weinstein domains that are almost symplectomorphic are Weinstein homotopic (and hence have exact symplectomorphic completions and contactomorphic boundaries).
		\end{itemize}
		
		\subsubsection{Formal structures}\label{sec: formal structures}
		There are also formal versions of symplectic, Weinstein, and contact structures that depend on just the underlying algebraic topological data. 
		For example, an \textit{almost symplectic structure} $(W, J)$ on $W$ is an almost complex structure $J$ on $W$; this is equivalent to having a 
		non-degenerate (but not necessarily closed) 2-form on $W$. 
		An almost symplectomorphism between two almost symplectic manifolds $(W_1, J_1), (W_2, J_2)$ is a diffeomorphism $\phi: W_1 \rightarrow W_2$ such that $\phi^*J_2 $ can be deformed to $J_1$ through almost complex structures on $W_1$. Equivalently, it also means that there is a family of non-degenerate 2-forms $\omega_t$ interpolating between $\omega_1$ and $\omega_2$.

		An \textit{almost Weinstein domain} is a triple $(W, J, \phi)$, where $(W,J)$ is a compact almost symplectic manifold with boundary 
		and $\phi$ is a Morse function on $W$ with no critical points of index greater than $n$ and maximal level set $\partial W$. An \textit{almost contact structure} $(Y, J)$ on $Y$ is an almost complex structure $J$ on the stabilized tangent bundle $TY \oplus \epsilon^1$ of $Y$. Therefore an almost symplectic domain $(W, J)$ has almost contact boundary $(\partial W, J|_{\partial W})$; it is an almost symplectic filling of this almost contact manifold. Therefore a family of almost symplectic structures give rise to a family of almost contact structures on the boundary:
		\begin{lemma}\label{lemma for almost contact}
		Almost symplectomorphic Liouville domains have almost contactomorphic boundaries.
		\end{lemma}
		Note that any symplectic, Weinstein, or contact structure can also be viewed as an almost symplectic, Weinstein, or contact structure by considering just the underlying  algebraic topological data. 
		
		Note that the first Chern class $c_1(J)$ is an invariant of almost symplectic, almost Weinstein, or almost contact structures.
		In this paper, we will often need to assume that $c_1(J)$ vanishes. The following proposition, which will be used several times in this paper, shows  that the vanishing of $c_1(Y, J)$ is often preserved under contact surgery and furthermore implies the vanishing of $c_1(W, J)$.

		\begin{prop}[Proposition 2.1 \cite{lazarev2016contact}]\label{prop: c1equivalence}
			Let $(W^{2n}, J), n \ge 3,$ be an almost Weinstein cobordism between $\partial_-W = (Y_-, J_-)$ and $\partial_+W = (Y_+, J_+)$. If $H^2(W,Y_-) = 0$, the following are equivalent:
			\begin{itemize}
				\item $c_1(J_-)=0, c_1(J_+)=0$
				\item  $c_1(J)=0$.
			\end{itemize}
			If  $\partial_-W =\emptyset$, the vanishing of 
			$c_1(J_+)$ and $c_1(J)$ are equivalent. 
		\end{prop}
		\begin{proof}
			Let $i_\pm: Y_\pm \hookrightarrow W$ be inclusions. Then $i_\pm^*c_1(J) = c_1(J_\pm)$ so the vanishing of $c_1(J)$ implies the vanishing of $c_1(J_-)$ and $c_1(J_+)$.  
			To prove the converse, consider the cohomology long exact sequences of the pairs $(W, Y_-)$ and $(W, Y_+)$: 
			$$
			H^2(W, Y_\pm; \mathbb{Z})
			\rightarrow 
			H^2(W; \mathbb{Z}) \xrightarrow{i_\pm^*} H^2(Y_\pm; \mathbb{Z}).
			$$
			By assumption, $H^2(W, Y_-; \mathbb{Z})$ vanishes and hence $i_-^*$ is injective. By Poincar\'e-Lefschetz duality, 
			$H^2(W, Y_+; \mathbb{Z}) \cong H_{2n-2}(W, Y_-;\mathbb{Z})$. Since $2n-2 \ge n+1$ for $n \ge 3$ and $W$ is a Weinstein cobordism, 
			$H_{2n-2}(W, Y_-;\mathbb{Z})$ vanishes and hence $i_+^*$ is also injective. Then if either $c_1(J_-) = i_-^*c_1(J)$ or $c_1(J_+) = i_+^*c_1(J)$ vanish, so does $c_1(J)$.
			
			If $\partial_-W = \emptyset$, 
			we just need the vanishing of $H^2(W, Y_+; \mathbb{Z})$, which holds for $n\ge 3$.
		\end{proof}

		\subsection{Morse-Bott case}
		The results of this section largely come from  \cite{mclean2016reeb}.
		\begin{defn}
			A \textit{Morse-Bott family of Reeb orbits of $(C,\alpha)$ of period $T$} is a closed path connected submanifold $B\subset C$ where $B$ is contained in the image of the union of closed Reeb orbits of period $T$,
			satisfying  $\ker(D\psi_T)|_B=TB$, where  if $\psi_t:B\to B$ is the Reeb flow of $\alpha$.
		\end{defn}

\begin{comment}
	\begin{defn}
		Suppose $B_T$ is a subset of $C$ so that:
		\begin{enumerate}
			\item $B_T$ is a union Reeb orbits of period $T$.
			\item There is a neighborhood $\mathcal{N}_{B_T}$ of $B_T$ and a constant $\delta>0$ so that
			any Reeb orbit
			with period in the interval $[T-\delta,T+\delta]$ meeting $\mathcal{N}_{B_T}$
			is in fact contained in $B_T$ and has period $T$.
		\end{enumerate}
		Then we say $B_T$ {\it is an isolated family of Reeb orbits of period} $T$.
		We will say that $\mathcal{N}_{B_T}$ is {\it an isolating neighborhood for} $B_T$.
	\end{defn}

	\begin{defn} \label{defn:morsebottfamily}
		A {\it pseudo Morse-Bott family} is an isolated family of Reeb orbits $B_T$ of period $T$ with the additional property
		that $B_T$ is path connected and for each point $p \in B_T$ we have
		$\text{Size}_p(B_T) := \text{dim ker}(D_p\psi_T|_\xi - \text{id})$ is constant along $B_T$
		(recall that $D_p\psi_t|_\xi : \ker(\alpha)_p \to \ker(\alpha)_{\psi_t(p)}$ is the restriction
		of the linearization of $\psi_t$ to $\ker(\alpha)_p$).
	\end{defn}
	An example of a pseudo Morse-Bott family of period $T$ is a Morse-Bott family of Reeb orbits of period $T$.
	\end{comment}
	We are interested in indices of Reeb orbits and so
	from now on we assume that we work with a fixed trivialization of a fixed power 
	of the canonical bundle of $(C,\alpha)$.

	Note that the Conley-Zehnder index of the
	period $T$ orbits starting in $B$ are all the same because $B$ is path connected.
	Hence we define the {\it Conley-Zehnder index of} $B$, $\mu_{CZ}(B)$, to be the Conley-Zehnder index of one of its period $T$ Reeb orbits.

		We can define an index closely related to the Conley-Zehnder index, called \textit{lower SFT index}, $lSFT(\gamma)$, as follows:
		\[lSFT(\gamma):=\mu_{CZ}(\gamma)-\frac{1}{2}\dim\ker(D_{\gamma(0)}\psi_T|_\xi-\textrm{id})+(n-3).
		\]
		Similarly, we have the following definition:
		\begin{defn}
			Let $K$ be a Hamiltonian on a symplectic manifold $(X,\omega_X)$ and $B$ is a set of fixed points of its time $T$ flow. We say that $B$ is \textit{isolated} if any such fixed point near $B$ is contained in $B$. suppose $B$ is a path connected topological space and we have fixed a symplectic trivialization of the canonical bundle of $TX$.
			Then every such Hamiltonian orbit has the same Conley-Zehnder index and we will write $\mu_{CZ}(B,K)$ for the Conley-Zehnder index. The set $B$ is said to be \textit{Morse-Bott} if $B$ is a submanifold and $\ker(D\psi_K^T-id)=TB$ along $B$ where $\psi_K^T:X\to X$ is the time $T$ Hamiltonian flow of $K$.  
		\end{defn}
		
		The following lemma is a technical lemma which relates the index of Reeb orbits in a contact hypersurface (which is a regular level set of a Hamiltonian) and the index of the corresponding Hamiltonian orbits.
		\begin{lemma}[Lemma 5.22 \cite{mclean2016reeb}]\label{Reeb--Hamiltonian index relation}
			Let $(W,\omega_W)$ be a symplectic manifold with a choice of symplectic trivialization of the canonical bundle of $TW$. Let $\theta_W$ be a 1-form satisfying $d\theta_W=\omega_W$, and $K$ be an Hamiltonian with the property that $b:=\iota_{X_{\theta_W}}dK>0$. This means $C_r:=K^{-1}(r)$ is a contact manifold with contact form $\alpha_r:=\theta_W|_{C_r}$. Let $B\subset W$ be a connected submanifold transverse to $C_r$ for each $r$ so that $B_r:=C_r\cap B$ is a Morse-Bott submanifold of the contact manifold $(C_r,\alpha_r)$ of period $L_r$, where $L_r$ smoothly depends on $r$. Suppose that $b = L_0$ along $B_0$
			and that $db(V) > \frac{d(L_r)}{dr}|_{r = 0}$ along $B_0$, where $V$ is a vector field tangent to $B$
			satisfying $dK(V) = 1$. Then $B_0$ is Morse-Bott for $K$ and $\mu_{CZ}(B_0,K)=\mu_{CZ}(B_0,\alpha_0)+\frac{1}{2}$.
		\end{lemma}
		\begin{remark}
			Our sign convention is different from McLean's in  \cite{mclean2016reeb} since we use $\omega(\cdot,X_H)=dH$.
			 So the condition on $b:=\iota_{X_{\theta_W}}dK>0$ differs by a minus sign. If $b\neq L_0$ along $B_0$, we have to either rescale $b$ or $L_t$.
		\end{remark}
         \begin{remark}\label{remark: index of orbits on cylinder}
         	In light of lemma~\ref{index for the cylinder component}, we have $\mu_{CZ}(\gamma,r^2)=\frac{1}{2}$, where $\gamma$ is any Morse-Bott manifold of Hamiltonian orbits.
         \end{remark}

		\section{ADC structures and positive idempotent group}
		
		We will define (strongly) asymptotically dynamically convex contact structure first. Then we introduce a new invariant called the \textit{positive idempotent group} base on $SH_*(W)$. It does not depend on the filling for ADC contact structures and therefore can be seen as a contact invariant. The proof will be deferred to section~\ref{section:indepnedence of I+}. In subsection~\ref{subsection:effect of conatct surgery}, we show that the (strongly) ADC property is preserved under subcritical contact surgery.

		\subsection{Reeb orbits and asymptotically dynamically convex contact structures}\label{section : def of ADC}
		Let's take a moment to look at the degree of Reeb orbits, which is essential for the definition of ADC contact structures. For any contact manifold $(\Sigma,\alpha)$ with $c_1(\Sigma,\xi)$, the canonical line bundle of $\xi$ is trivial, as will always be the case in this paper. After choosing a global trivialization of this bundle, we can assign an integer to each Reeb orbit $\gamma$ of $(\Sigma,\alpha)$-the reduced Conley-Zehnder index:\[|\gamma|:=\mu_{CZ}(\gamma)+n-3.\]For a general Reeb orbits, $|\gamma|$ depends on the choice of trivialization of the canonical bundle. However, if the Reeb orbit $|\gamma|$ is contractible in $\Sigma$, then the grading does not depend on the trivialization.
		We will consider both the contractible and non-contractible Reeb orbits in this paper.
		
		Let $\mathcal{P}^{<D}_\Phi(\Sigma,\alpha)$ be the set of Reeb orbits $\gamma$ of $(\Sigma,\alpha)$ satisfying $A(\gamma)<D$, where $\Phi$ is a specific trivialization of the canonical bundle. In a similar manner , we can define $\mathcal{P}^{<D}_0(\Sigma,\alpha)$ to be the set of contractible Reeb orbits $\gamma$ of $(\Sigma,\alpha)$ satisfying $A(\gamma)<D$. Here we dropped the subscript $\Phi$ since the degree of contractible Reeb orbits does not depends on the choice of trivialization.

		\begin{lemma}[Proposition 3.1 \cite{lazarev2016contact}]\label{prop: easy}
			For any $D, s>0$, there is a grading preserving bijection between 
			$\mathcal{P}^{< D}_\Phi(Y, s\alpha)$ and $\mathcal{P}^{< D/s}_\Phi(Y, \alpha)$.
		\end{lemma}
		\begin{proof}
			Note that $R_{s\alpha} = \frac{1}{s}R_\alpha$. So if $\gamma_\alpha:[0, T]\rightarrow Y$ is a Reeb trajectory of $\alpha$ with action $T$, then $\gamma_{s\alpha} = \gamma_{\alpha}\circ m_{\frac{1}{s}}: [0, s T] \rightarrow Y$ is a Reeb trajectory of $s\alpha$ with action $sT$; here  $m_{\frac{1}{s}}: [0, sT] \rightarrow [0, T]$ is multiplication by $\frac{1}{s}$. 
			The map $\gamma_\alpha \rightarrow \gamma_{s\alpha}$ is a bijection between the set of Reeb orbits. If $T < D/s$, then $sT < D$ and so it is a bijection between $\mathcal{P}_\Phi^{< D/s}(Y, \alpha)$ and $\mathcal{P}_\Phi^{< D}(Y, s\alpha)$.  This bijection is grading-preserving since the Conley-Zehnder index of a Reeb orbit is determined by the linearized Reeb flow on the trivialized contact planes $\xi$ but  does not depend on the speed of the flow.  
		\end{proof}
		
		We will also need the following notation. If $\alpha_1, \alpha_2$ are contact forms for $\xi$, then there exists a unique $f: Y \rightarrow \mathbb{R}^+$ such that $\alpha_2 = f \alpha_1$. We write $\alpha_2 > \alpha_1, \alpha_2 \ge \alpha_1$ if $f > 1, f\ge 1$ respectively. Note that if $\alpha_2 > \alpha_1, \alpha_2 \ge \alpha_1$, then for any diffeomorphism $\Psi: Y' \rightarrow Y$, we have
		$\Psi^*\alpha_2 > \Psi^*\alpha_1, \Psi^*\alpha_2 \ge \Psi^*\alpha_1$, respectively.
		
		%%%%%%%%%%%%%%%%%%%%%%%%%%%%%%%%%%%%%%%%%%%%%%%%%%%

		\begin{defn}\label{def-adc}
			A contact manifold $(\Sigma,\xi)$ is \textit{asymptotically dynamically convex } (\textit{strongly asymptotically dynamically convex with respect to $\Phi$ }) if there exists a sequence of non-increasing contact forms $\alpha_1\geq \alpha_2\geq \alpha_3\cdots $ for $\xi$  and increasing positive numbers $ D_1<D_2<D_3\cdots $going to infinity such that all elements of $ \mathcal{P}_0^{<D_k}(\Sigma,\alpha_k)$ ($\mathcal{P}^{<D_k}_\Phi(\Sigma,\alpha_k)$) have positive lower SFT index.
		\end{defn}
		\begin{remark}
			The ADC property defined in definition 3.6
			 \cite{lazarev2016contact} requires the non-degeneracy of $\alpha_i$. 
			Here we define the strongly ADC property (with respect to $\Phi$) using lower SFT index. Therefore the contact form $\alpha_i$ in the definition doesn't have to be non-degenerate. It is an immediate corollary of lemma~\ref{mark's perturbation}, also see 
			remark 3.7 (2) of \cite{lazarev2016contact}.
		\end{remark}
	\begin{lemma}[Lemma 4.10 \cite{mclean2016reeb}]\label{mark's perturbation}
			Let $\gamma$ be any Reeb orbit of $\alpha$ of period T and define $K:=\dim \ker(D\psi_T |\xi(\gamma(0))-id)$. Fix some Riemannian metric on C. There is a constant $\delta > 0 $ and a neighborhood N of
			$\gamma(0) $ so that for any contact form $\alpha_1$ with $||\alpha-\alpha_1||_{\mathcal{C}^2}<\delta$ and any Reeb orbit $\gamma_1$ of $\alpha_1$ starting
			in N of period in $[T-\delta, T + \delta]$ we have  $ \mu_{CZ}(\gamma_1) \in [\mu_{CZ}(\gamma)-\frac{K}{2}, \mu_{CZ}(\gamma) + \frac{K}{2}].
			$
		\end{lemma}
		\subsection{Positive idempotent group $I_+$}\label{ss:definition of positive idempotent group}
			Now we consider a strongly asymptotically dynamically convex contact manifold $(\Sigma,\xi,\Phi)$ with Liouville filling $(W,\lambda)$. We have the following result due to Lazarev:
		
		\begin{theorem}[Proposition 3.8 \cite{lazarev2016contact}]\label{Lazarev main}
			If $(\Sigma,\xi,\Phi)$ is a strongly asymptotically dynamically convex contact structure, then all Liouville fillings of  $(\Sigma,\xi,\Phi)$ have isomorphic $SH^+$.
		\end{theorem}
		\subsubsection{Definition of positive idempotent group $I_+$}
		We also want to define the ring structure. However as in Remark~\ref{NOproduct}, we can not define a product on $SH^+$.
		Having said that, we can use the pair-of-pants product on $SH_*(W)$ to define an invariant for $SH_*^+$ which is independent of the Liouville filling.
		
		First, let's recall the tautological short exact sequence:
		\[0\rightarrow SC_*^{< \epsilon}(W)\rightarrow SC_*(W)\rightarrow
		SC_*(W)/SC^{<\epsilon}_*(W)\rightarrow 0.
		\]
		We have long exact sequence:
		\begin{equation}\label{eqn:exact sequence}
		\cdots\to SH_*^{<\epsilon}(W)\to SH_*(W)\to SH_*^+(W)\to SH_{*-1}^{<\epsilon}(W)\to\cdots
		\end{equation}
		We also have $H^{n-*}(W,H)\cong SH_*^{<\epsilon}(W,H)$ since the admissible Hamiltonian $H$ is $\mathcal{C}^2$ small in $W$. Therefore we can replace $SH_*^{<\epsilon}$ terms in equation~\ref{eqn:exact sequence} by $H^{n-*}$, in particular, we have a long exact sequence 
		\begin{equation}\label{eqn: exact sequence for defn}
		\cdots\to H^0(W)\to SH_n(W)\to SH_n^+(W)\to H^1(W)\to\cdots
		\end{equation}
		 In fact, the map $H^0(W)\to SH_n(W)$ in equation~\ref{eqn: exact sequence for defn} is a ring homomorphism, see Appendix. A of  \cite{cieliebak2018symplectic}. Suppose $SH_*(W)\neq 0$, then $1_W$ does \textit{not} maps to the unit in $SH_*(W)$, where $1_W $ is the unit of $ H^0(W)$, by Theorem~\ref{Theorem unit is image of 1} (also see Lemma A.3 of  \cite{cieliebak2018symplectic}).
		 Therefore
		  $H^0(W)\to SH_n(W)$ is injective, and we will regard it as a subring of $SH_n(W)$. We can thus 
		  identify
		  elements in $SH_n(W)/H^0(W)$ with elements in $SH^+_n(W)$. In particular, $SH_n(W)/H^0(W)\cong SH^+_n(W)$ if $H^1(W)=0$.
		  
		  Now let's consider the subgroup of $SH_n(W)$ as follows:
		  \begin{equation}\label{def of I}
		  	I(W):=\{ \,\alpha \in SH_n(W)\,\big| \, \alpha^2-\alpha \in H^0(W)\}.
		  \end{equation}
		  Notice the group action here is "addition".
		  
		  Define the \textit{positive idempotent group} $I_+(W)$ by
		  \[I_+(W):=I(W)/H^0(W).
		  \]
		By the previous analysis, we can regard $I_+(W)$ as a subgroup of $SH_n^+(W)$. In the case $I_+(W)$ is finite, we can further define \textit{positive idempotent index} $i(W):=|I_+(W)|$.

	   \subsubsection{Properties of $I_+$}
		Since $H^0(W,\Z_2)\cong \Z_2$, $I_+(W,\Z_2)$ is determined by $I(W,\Z_2)$.
		Recall that $SH_*(W,\Z_2)$ has a $H_1(W,\Z)/Tors$ grading. The first observation is that elements in $H^0(W)$ have $H_1/Tors$ grading zero. Indeed, it's true for all elements in $I(W)$. Suppose $R$ is an algebra over $\mathbb{Z}_2$ which is graded by a finitely generated torsion-free abelian group $K$. This means that as a vector space, $R=\bigoplus\limits_{k \in K} R_k$ with the property that if $ a\in R_{k_1}, b\in R_{k_2}$ then $ab\in R_{k_1\cdot k_2}$.
		Define $I_0(R):=\{0,1\}$ and $I(R):=\{ x\in R|x^2-x\in I_0\}$.
		\begin{lemma}[Lemma 7.6 \cite{mclean2007lefschetz}]
			If $a\in I(R)$ then $a\in R_e$ where $e$ is identity of group $K$.
		\end{lemma}
		\begin{proof}
			We argue by contradiction. Suppose we have $a=a_{k_1}+\cdots+a_{k_n}$ where $k_i\in K$ and $a_{k_i}\in R_{k_i}$, $k_1\neq e$. Then $a^2=a_{k_1^2}+\cdots+a_{k_n^2}$. Since $K$ is torsion free,
			there is a group homomorphism $p: K\to \mathbb{Z}$ such that $p(k_1)\neq 0$. This map actually gives $R$ a $\mathbb{Z}$ grading. Let $b$ be an element in $R$, then it can be uniquely written as $b=b_1+\cdots+b_k$ where $b_k$ are non-zero elements of $R$ with grading $d_i\in \mathbb{Z}$. We can define a function $f$ as follows:
			\[f(b):=\min\{ |d_i|\neq 0\}
			\]
			Note that $f$ is well-defined only if at least one of the $d_i's$ is non-zero.
			And when it is well-defined, $f(b+1)=f(b)$ because $1\in R_e$ and has grading $0$.
			The assumption $p(k_1)\neq 0$ implies that $f(a)$ is well defined and positive. On the other hand, we have $a\in I(R)$, which means $a^2=a$ or $a^2=a+1$. Either way, it implies $f(a^2)=f(a)$, which contradicts the fact that $f(a^2)\geq 2f(a)$. 
		\end{proof}
		\begin{corollary}\label{nullhomologous of idempotents}
			Any element in $I(W)$ is null-homologous in $H_1(W,\Z)/Tors$.
		\end{corollary}
		Therefore, we can refine our definition of $I(W)$ to be
		\begin{equation*}
		I(W):=\{ \,\alpha \in SH^0_n(W)\,\big| \, \alpha^2-\alpha \in H^0(W)\}.
		\end{equation*}
		where $SH^0_*(W)$ is generated by all null-homologous Reeb orbits.
		In the case of strongly asymptotically dynamically convex contact manifolds (with respect to certain framing $\Phi$), different Liouville fillings have isomorphic positive idempotent group, as stated in Theorem~\ref{Main Thm}, the proof will be deferred to section~\ref{section:indepnedence of I+}.

		\subsection{Effect of contact surgery}\label{subsection:effect of conatct surgery}
		\begin{theorem}[Theorem 3.15 \cite{lazarev2016contact}, \cite{yau2004cylindrical}]\label{subcritical surgery of ADC}
			If $(Y_1^{2n-1},\xi_1), n\geq 2$, is an asymptotically dynamically convex contact structure and $(Y_2,\xi_2)$ is the result of index $k\neq 2$ subcritical contact surgery on $(Y_1,\xi_1)$, then $(Y_2,\xi_2)$ is also asymptotically dynamically convex.
		\end{theorem}

	Now we are in the position to prove Proposition~\ref{new prop}.
	\begin{proof}[Proof of Proposition~\ref{new prop}]
		Recall that $W_1$ is a flexible Weinstein domain, so $SH_*(W)=0$ (see  \cite{bourgeois2012effect}).
		By lemma~\ref{ADC boundary}, $\partial W_1$ is asymptotically dynamically convex and so is $W$ by assumption. Moreover, $\partial V_k$ is obtained by attaching a Weinstein $1-$handle to asymptotically dynamically convex contact manifold, therefore it is asymptotically dynamically convex by Theorem~\ref{subcritical surgery of ADC}.
		A well known fact is that subcritical surgery does not change symplectic homology as a ring, see Theorem~\ref{invariant of SH}. 
		We have
		\[ SH_n(W_i)\cong \underbrace{SH_n(W)\oplus\cdots\oplus SH_n(W)}_i \oplus \underbrace{SH_n(W_1)\oplus\cdots\oplus SH_n(W_1)}_{k-i}=\bigoplus_{j=1}^i SH_n(W)
		\]
		so we have 
		\[I(W_i)\cong \underbrace{I(W)\oplus\cdots\oplus I(W)}_i.
		\]
		Since $SH_*(W)\neq 0$ and is finite dimensional, $\{0,1_{W}\}\subset I(W)$. We therefore have 
		$2\leq | I(W)|<\infty$, so $ |I(W_i)|=|I(W)|^i$ are mutually distinct. Therefore, $|I_+(W_i)|\neq |I_+(W_j)| \textrm{ for }i\neq j$.
		
	\end{proof}

		\begin{theorem}[ \cite{lazarev2016contact}, \cite{yau2004cylindrical}]\label{theorem:contact surgery}
			Let $(\Sigma_1 ,\xi_1)$ be a strongly asymptotically dynamically convex contact structure with respect to $\Phi$, and $(\alpha_k,D_k)$ as in Definition~\ref{def-adc} and $(\Sigma_2,\xi_2) $ be the result of index 2 contact surgery on $\Lambda^1\subset\Sigma_1$ so that the trivialization $\Phi$ extends to the handle. Then $(\Sigma_2,\xi_2)$ is also strongly asymptotically dynamically convex with respect to $\Phi$. 
		\end{theorem}
		\begin{remark}
			Since the trivialization $\Phi$ of the canonical bundle of $(\Sigma_1 ,\xi_1)$ extends to the attaching handle, so by abuse of notation, the trivialization of the canonical bundle of $(\Sigma_2,\xi_2)$ which is obtained by extending $\Phi$ to the attaching handle is still denoted by $\Phi$. 
		\end{remark}
		\begin{prop}[Proposition 5.5 \cite{lazarev2016contact}, \cite{yau2004cylindrical}]\label{key}
			Let $\Lambda^{k-1}\subset (\Sigma_1^{2n-1},\alpha_1),n>1,$ be an isotropic sphere with $k<n$. For any $D>0$ and integer $i>0$, there exists $\epsilon=\epsilon(D,i)>0$ such that if $(\Sigma_2,\alpha_2)$ is the result of contact surgery on $U^\epsilon(\Lambda,\alpha)$ with respect to the trivialization $\Phi$, then there is a grading preserving bijection between  $\mathcal{P}^{<D}_\Phi(\Sigma_2,\alpha_2)$ and  $\mathcal{P}^{<D}_\Phi(\Sigma_1,\alpha_1)\cup \{\gamma^1,\cdots,\gamma^l\}$ where 
			$|\gamma^i|=2n-k-4+2i$.
			
		\end{prop}
		
		\begin{remark}
			The proof largely follows  \cite{lazarev2016contact} proposition 5.5, with only minor changes regarding the non-contractible Reeb orbits. The difference in the Strongly ADC case is that we need to choose the trivialization to  define the Conley-Zehnder index. 
			
		\end{remark}
		\begin{proof}[Proof of Proposition~\ref{key}]
			As explained in  \cite{yau2004cylindrical}, the surgery belt sphere $S^{2n-k-1}$ contains a contact sphere $(S^{2n-2k-1},\xi_{std})$. After taking appropriate sequence of contact forms on $(\Sigma_2,\xi_2)$, the Reeb orbits of $(\Sigma_2,\xi_2)$ correspond to the old Reed orbits of $(\Sigma_1,\xi_1)$, plus the new orbits of $(S^{2n-2k-1},\xi_{std})$ inside the belt sphere of action less than $D$. The correspondence is natural since the trivialization of the canonical bundle extends over the surgery. These new orbits corresponds to the iterations $\gamma^1,\cdots,\gamma^l$ of a single Reeb orbit $\gamma$, see  \cite{yau2004cylindrical}. Moreover, $\mu_{CZ}(\gamma^i)=n-k-1+2i$ and therefore $|\gamma^i|=2n-k-4+2i$. Meanwhile, by shrinking the handle, the action can be made arbitrarily small and therefore we can ensure that arbitrarily large iterations of $\gamma$ have action less than $D$.
		\end{proof}

	For $\Lambda \subset \Sigma$, Since $ J^1(\Lambda) \simeq T^*\Lambda\times\R$, choose a Riemannian metric on $\Lambda$. Let $U^\epsilon(\Lambda) \subset (J^1(\Lambda),\alpha_{std})$ be $\{ ||y||<\epsilon,|z|<\epsilon\}$,  
	the metric on $\Lambda$ to define $||y||$ on the fiber, $z$ is the coordinate on $\R$. If $\Lambda \subset (Y,\alpha) $ is Legendrian, let $U^{\epsilon}(\Lambda,\alpha) \subset (Y,\alpha)$ be a neighborhood of $\Lambda$ that is strictly contactomorphic to $U^{\epsilon}(\Lambda)$.

		\begin{prop}[Proposition 6.7 \cite{lazarev2016contact}]\label{prop6.7}
			Let $\alpha_1>\alpha_2$ be contact forms for $(\Sigma,\xi)$ and let $\Lambda \subset (\Sigma,\xi)$ be an isotropic submanifold with trivial symplectic conormal bundle. Then for any sufficiently small $\delta_1,\delta_2$, there exists a contactomorphism h of $(\Sigma,\xi)$ such that
			\begin{itemize}
				\item h is supported in $U^\epsilon(\Lambda,\alpha_1),h|_\Lambda=Id$, and $h^*\alpha_2<4\alpha_1$
				\item $h^*\alpha_2|_{U^{\delta_1}(\Lambda,\alpha_1)}=c\alpha_1|_{U^{\delta_1}(\Lambda,\alpha_1)}$ for some constant c (depending on $\delta_1,\delta_2$) 
				\item $h(U^{\delta_1}(\Lambda,\alpha_1))\subset U^{\delta_2}(\Lambda,\alpha_2)$.
			\end{itemize}
		\end{prop}

		\begin{prop}[Remark 6.5 \cite{lazarev2016contact}]\label{remark6.5}
			Let $\Lambda \subset (\Sigma_1^{2n-1},\xi_1),n>2$ be an isotropic sphere and $(\Sigma_2,\xi_2)$ be  the result of contact surgery on $\Lambda$ which extends the chosen trivialization $\Phi$ of the canonical bundle. Suppose $(\Sigma_1,\xi_1)$ is a strongly asymptotically dynamically convex contact structure with respect to the trivialization $\Phi$ and has $(\alpha_k,D_k)$ as in Definition~\ref{def-adc}. If $\alpha_k|_{U^\epsilon(\Lambda,\alpha_1)}=c_k\alpha_1|_{U^\epsilon(\Lambda,\alpha_1)}$ for some constants $\epsilon,c_k$, then $(\Sigma_2,\xi_2)$ is also strongly asymptotically dynamically convex with respect to the trivialization $\Phi$. 
		\end{prop}

		\begin{proof}[Proof of Theorem~\ref{theorem:contact surgery}]
			
			Now we will proceed exactly as Lazarev did, keeping in mind that we are dealing with the strongly ADC property. We can apply Proposition~\ref{prop6.7} so that the conditions of Proposition~\ref{remark6.5} are satisfied.
			
		\end{proof}
	
		\section{$I_+$ is an invariant of ADC contact manifolds }\label{section:indepnedence of I+}

	We will follow Lazarev's approach. Here we will use the procedure called stretching-the-neck.(For details, see Section 3.3 in~\cite{lazarev2016contact})

	\begin{prop}[Proposition 3.10 \cite{lazarev2016contact}]\label{lazarev proof for differential}
		Suppose that $(Z,\alpha)$ is strongly ADC with respect to the trivialization $\Phi$ and all elements of $\mathcal{P}^{<D}_\Phi(Z,\alpha)$ have positive degree. If $A_{H_+}(x_+)-A_{H_-}(x_-)<D$, then there exists $R_0 \in (0,1-\delta)$ such that for any $R\leq R_0$, all rigid $(H_s,J_{R,s})$-Floer trajectories are contained in $\widehat{W}\setminus V$.
	\end{prop}
	\begin{remark}
		If $H_s$ is independent of $s$, then the Floer trajectories define the differential; if $H_s$ is an decreasing homotopy, then $(H_s,J_{R,s})$-Floer trajectories define the continuation map.
	\end{remark}
	In P.Uebele's paper  \cite{uebele2015periodic}, the pair-of-pants product is defined for ``index-positive" contact manifold, where the symplectic homology used is actually Rabinowitz-Floer homology. Though as the paper points out, the ring  structure is not well defined on $SH^+$. However, at the chain level, if the pair-of-pants product is asymptotic to Hamiltonian orbits of positive action, then by the stretching-the-neck technique, we can prove the pair-of-pants does not enter the interior of the Liouville filling. In \cite{uebele2015periodic}, this is proved for index-positive contact manifolds. However, it is not true for Strongly ADC contact manifolds in general. That being said, the pair-of-pants does \textit{not} enter the interior of the filling when the indices of the Hamiltonian orbits of the asymptotes are high enough. To  be precise, we have the following:
	\begin{prop}\label{pair of pants product --nonberaking }
		Suppose that $(Z,\alpha)$ is strongly ADC with respect to the trivialization $\Phi$ and all elements of $\mathcal{P}^{<D}_\Phi(Z,\alpha)$ have positive reduced Conley-Zehnder index. Furthermore, let $A_H(x_i)<D/2(i=1,2,3)$ be non-constant Hamiltonian orbits such that $\mu_{CZ}(x_1)+\mu_{CZ}(x_2)-\mu_{CZ}(x_3)=n$ and $\mu_{CZ}(x_i)\geq n,i=1,2$, then there exists $R_0 \in (0,1-\delta)$ such that for any $R\leq R_0$, all pair-of-pants products are contained in $\widehat{W}\setminus V$.
	\end{prop}
	\begin{figure}
		\centering
		\usetikzlibrary{arrows}
		\begin{tikzpicture}
		\draw  (0,0)ellipse (1 and 0.25);
		\draw  (-0.3,3)ellipse (0.8 and 0.2);
		\draw  (5,3)ellipse (0.6 and 0.2);
		\draw [dashed] (3,-1)ellipse (0.4 and 0.1);
		\draw [dashed] (3,-1.5)ellipse (0.4 and 0.1);
		\draw [dashed] (5,-1)ellipse (0.4 and 0.1);
		\draw [dashed] (5,-1.5)ellipse (0.4 and 0.1);
		
		\draw  plot[smooth, tension=.7] coordinates {(-1.1,3)(-0.8,1.7)(-1,0)};
		\draw  plot[smooth, tension=.6] coordinates {(0.5,3) (0.5,2) (1,1.5) (2.5,1) (3.15,0.5) (3.3,0) (3.4,-1)};
		
		\draw  plot[smooth, tension=.8] coordinates {(1,0) (2,0.5) (2.5,0) (2.6,-1)};
		\draw  plot[smooth, tension=.7] coordinates {(2.6,-1.5) (3.2,-3) (4,-3.6) (4.8,-3) (5.4,-1.5)};
		
		\draw node at(-0.3,3.5) {$x_1$};
		\draw node at(0,-0.5) {$x_3$};
		\draw node at(5,3.5) {$x_2$};
		\draw node at(5.6,-1.3) {$\gamma_2$};
		\draw node at(2.4,-1.3) {$\gamma_1$};
		\draw [-latex] (1.5,4)node[above] { $\mu_{CZ}(x_1)-\mu_{CZ}(x_3)-\mu_{CZ}(\gamma_1)-n+3\geq 0$}--(0,1);
		\draw [-latex] (6,1.5)node[right] { $\mu_{CZ}(x_2)-\mu_{CZ}(\gamma_2)+3\geq 0$}--(5,1.4);
		\draw [-latex] (-1,-1)node[below] { $E=A_H(x_1)-A_H(x_3)-A(\gamma_1)$}--(-0.4,0.8);
		
		\draw  plot[smooth, tension=.7] coordinates {(3.4,-1.5) (3.5,-2.5) (4,-3) (4.5,-2.5) (4.6,-1.5)};
		\draw  plot[smooth, tension=.7] coordinates {(4.4,3) (4.5,0.5) (4.6,-1)};
		\draw  plot[smooth, tension=.7] coordinates {(5.6,3) (5.5,0.5) (5.4,-1)};
		\draw  plot[smooth, tension=.7] coordinates {(-1,3.5)};
		\end{tikzpicture}
		\caption{Top of Floer building is connected (such breaking does \textit{not} occur). Hamiltonian orbits are represented by continuous lines, Reeb orbits by dashed lines.}\label{Fig: connected top floer biulding}
	\end{figure}
	\begin{proof}
		The proof is a combination of proposition 3.10 of  \cite{lazarev2016contact} and  lemma 3.12 of  \cite{uebele2015periodic}. First of all, we have to rule out the breaking as in Figure~\ref{Fig: connected top floer biulding} (Similarly with $x_1$ and $x_2$ exchanged). 
		
		Suppose we have the breaking as in Figure~\ref{Fig: connected top floer biulding}, then the top level has positive dimension, and we have (see lemma 3.10 of  \cite{uebele2015periodic})
		\[\mu_{CZ}(x_1)-\mu_{CZ}(x_3)-\mu_{CZ}(\gamma_1)-n+3\geq 0
		\]
		and
		\[\mu_{CZ}(x_2)-\mu_{CZ}(\gamma_2)+3\geq 0.
		\]
		Then, since  \[\mu_{CZ}(x_1)+\mu_{CZ}(x_2)-\mu_{CZ}(x_3)=n,
		\]
		these conditions are reduced to 
		
		\[\mu_{CZ}(\gamma_1)\leq 3-\mu_{CZ}(x_2) \quad \text{and}\quad \mu_{CZ}(\gamma_2)\leq 3+\mu_{CZ}(x_2).\]
		In particular, \[\mu_{CZ}(\gamma_1)\leq 3-n.\]
		Meanwhile, the Floer energy of the top level would be \[0\leq E= A_H(x_1)-A_H(x_3)-A(\gamma_1).\]
		So 
		\[0<A(\gamma_1)\leq A_H(x_1)-A_H(x_3)< D.
		\]
		We have $\gamma_1 \in \mathcal{P}^{<D}_\Phi(Z,\alpha)$, which implies $\mu_{CZ}(\gamma_1)>3-n$, which is a contradiction.
		Now we know that the top of the Floer building is connected, so we can proceed as in the proof of proposition 3.10 of  \cite{lazarev2016contact}. We prove this by contradiction. Suppose the pair-of-pants product breaks after neck-stretching and $\gamma_k$ are the Reeb orbits in the top of the
		Floer building as in  \cite{lazarev2016contact}.
		The virtual dimension of the moduli space of the top Floer building is 
		\[ |x_1|+|x_2|-|x_3|-\sum |\gamma_k|<0.
		\]
		Contradiction.
	\end{proof}
	\begin{comment}
	
	Let's consider for $\alpha \in SH_*(W)\neq 0$, the \textit{spectrum value} of $\alpha$ is defined as 
	\[\nu(\alpha):=\inf\{ \,b\, | \,\alpha \in \,
	\Ima(SH_*^{(-\infty,b)}(W)\to SH_*(W)) \}.
	\]
	Clearly, we have $\nu(\alpha+\beta)\leq \max\{\nu(\alpha),\nu(\beta)\}$ by definition and 
	$\nu(\alpha\cdot\beta)\leq \nu(\alpha)+\nu(\beta)$ as a result of the fact that pair-of-pants product decreases action. The unit $e$ plays a special role. Indeed, we have $\nu(e)\leq 0$ since $SH_*^{\leq 0}(W)\to SH_*(W)$ is a map of rings with unit, also
	\[\nu(e)=\nu(e\cdot e)\leq 2\nu(e)
	\]
	Thus we have $\nu(e)=0$ or $\nu(e)=-\infty$.(Note that these conditions are independent of $\lambda$.)
	\end{comment}
	
	Now if we further require that the admissible Hamiltonian $H$ has a unique minimum (which is always possible and compatible with our requirements on admissible Hamiltonians), then the Floer chain complex $SC_*(W,H,J)=O_*(W,H,J)\oplus C_*(W,H,J)$, where $O_*(W,H,J)$ is generated by  all non-constant Hamiltonian orbits and $C_*(W,H,J)$ is generated by all constant Hamiltonian orbits (critical points of $H$). 
	Since $H$ is $\mathcal{C}^2$ small in $W$, the action of the critical points is small, and the Floer differential $d$ coincides with the Morse boundary operator $d_1$. We therefore have $(SC_*^{<\delta}(W,H,J),d)= (C_*(W,H,J),d_1)$ and $SC_*^+(W,H,J)=O_*(W,H,J)$.
	
	For degree reasons, $C_n(W,H,J)=\mathbb{Z}_2<p>$, where $p$ is the unique minimum of $H$. Note that $d(p)=d_1(p)= 0$ and we have the fact that $d(x+p)=0$ implies $d(x)=d(p)=0$.

	\begin{comment}
	We can define for the ring $SH_*(W)$ its a subring $I_0(W)$ as
	\begin{equation}
	I_0(W):=\{\,\alpha\in SH_n(W)\big|\, \alpha^2-\alpha=0,\, \nu(\alpha)\leq 0\,\}
	\end{equation}
	
	and
	\begin{equation}
	I(W):=\{ \,\alpha \in SH_n(W)\,\big| \, \alpha^2-\alpha \in I_0\}.
	\end{equation}
	Apparently $I_0$ and $I$ are groups.
	Notice that $I_0(W)=\{0,[p]\}$,where $[p]$ is the unit of ring $SH_*(W)$.In fact, $I_0(W)=SH_n^{\leq 0}(W)$. Notice that $SH_*^+(W)=SH_*(W)/SH_*^{\leq 0}(W)=SH_*(W)/I_0(W)$ we can define a subgroup called \textit{positive idempotent group} $I_+(W)$ of $SH_*^+(W)$ as follows:
	\[I_+(W):=\{ [\alpha]\big| \alpha^2-\alpha\in I_0(W) ,\alpha \in SH_*(W)\}=I/I_0
	\]

	Proposition~\ref{pair of pants product --nonberaking } tells us that as long as the inputs of pair-of-pants product have the required index, then its output is determined up to a difference of element generated by critical points. If we regard $SH_*^0,+(W)=SH_*^0(W)/SH_*^{0,\leq 0}(W)$, then the image under this quotient is independent of the filling.To be precise, we have the following proposition:
	
	\end{comment}

	\begin{figure}
		\centering
		\usetikzlibrary{decorations.pathmorphing, patterns,shapes}
		\begin{tikzpicture}
		\draw (0,6) arc (90:270:2cm);
		\draw (0,-2) arc (90:270:2cm);
		\draw[-] (0,2)--(10,2);
		\draw[-] (0,6)--(10,6);
		\draw[-] (0,-2)--(10,-2);
		\draw[-] (0,-6)--(10,-6);
		\draw[dashed](2,6)--(2,-6);
		\draw[dashed](6,6)--(6,-6);
		
		\draw  plot[smooth, tension=.7] coordinates {(-1.4,5) (-1.4,4.8) (-1.3,4.5) (-1,4.4) (-0.5,4.5) (-0.3,4.6)};

		\draw  plot[smooth, tension=.7] coordinates {(-1.4,4.8) (-0.8,4.9) (-0.5,4.5)};
		\draw  plot[smooth, tension=.7] coordinates {(-1.1,3.8) (-0.7,2.9) (0.1,2.5) (1.1,2.7)};
		
		\draw  plot[smooth, tension=.7] coordinates {(-1,3.5) (-0.2,3.4) (0.5,3.1) (0.8,2.6)};
		\draw  plot[smooth, tension=.7] coordinates {(-0.8,-2.8) (-0.3,-3.4) (0.6,-3.6) (1.2,-3.3)};
		\draw  plot[smooth, tension=.7] coordinates {(-0.6,-3.1) (0.4,-3.3) (0.6,-3.6)};
		\draw  plot[smooth, tension=.7] coordinates {(-1.2,-4.2) (-0.8,-4.7) (0,-4.5)};
		\draw  plot[smooth, tension=.7] coordinates {(-1,-4.5) (-0.6,-4.4) (-0.3,-4.6)};
		\draw  plot[smooth, tension=.7] coordinates {(0.3,-5.5) (1.4,-5.3) (1.5,-4.4)};
		\draw  plot[smooth, tension=.7] coordinates {(1.1,-5.4) (1.2,-5.1) (1.5,-5.1)};
		\draw (6.75,5.5) ellipse (0.1cm and 0.32cm);
		\draw (7,4) ellipse (0.1cm and 0.4cm);
		\draw (6.5,2.5) ellipse (0.1cm and 0.4cm);
		\draw (6.75,-2.5) ellipse (0.1cm and 0.32cm);
		\draw (7,-4) ellipse (0.1cm and 0.4cm);
		\draw (6.5,-5.5) ellipse (0.1cm and 0.4cm);
		\draw node (v1) at (0.9,4.7) {$p$};

		\draw  plot[smooth, tension=.7] coordinates {(6.8,5.2) (6.3,5) (6.3,4.5) (7,4.4)};
		\draw  plot[smooth, tension=.7] coordinates {(6.8,5.2-8) (6.3,5-8) (6.3,4.5-8) (7,4.4-8)};
		\draw  plot[smooth, tension=.7] coordinates {(6.7,5.8) (4.7,5.4) (3.4,4.3) (3.5,2.9) (4.8,2.2) (6.5,2.1)};
		\draw  plot[smooth, tension=.7] coordinates {(6.7,5.8-8) (4.7,5.4-8) (3.4,4.3-8) (3.5,2.9-8) (4.8,2.2-8) (6.5,2.1-8)};
		\draw  plot[smooth, tension=.7] coordinates {(7,3.6) (6.1,3.5) (5.6,3.2) (5.8,2.9) (6.5,2.9)};
		\draw  plot[smooth, tension=.7] coordinates {(7,3.6-8) (6.1,3.5-8) (5.6,3.2-8) (5.8,2.9-8) (6.5,2.9-8)};
		\node at (0.8,-2.6) {$q$};
		
		\draw [red,dashed] plot[smooth, tension=.7] coordinates {(6.7,5.8) (3.8,5.7) (2.3,5.6) (1.3,5.3) (v1)};
		\draw  [red,dashed]plot[smooth, tension=.7] coordinates {(6.8,5.2) (5.4,5) (5.1,4.6) (5.9,4.4) (7,4.4)};
		
		\draw [red,dashed] plot[smooth, tension=.7] coordinates {(7,3.6) (5.3,3.6) (2.8,3.7) (1.4,4) (v1)};
		\node at (7.2,-2.4) {$x_1^+$};
		\node at (7.5,-4) {$x_2^+$};
		\node at (6.95,-5.5) {$x_3^-$};
		\node at (7.2,-2.4+8) {$x_1^+$};
		\node at (7.5,-4+8) {$x_2^+$};
		\node at (6.95,-5.5+8) {$x_3^-$};
		\node at (0,4) {$W$};
		\node at (0,-4) {$V$};
		\node at (5,7) {$H_W$};
		\node at (5,-1) {$H_V$};
		
		\draw[decorate,decoration=zigzag](-2,6.5) -- (2,6.5) node (v2) {};
		\draw[decorate,decoration=zigzag](-2,6.5-8) -- (2,6.5-8) node (v3) {};
		\draw  plot[smooth, tension=.7] coordinates {(v2) (3.6,6.5) (6.3,6.8) (7.9,7.4) (9.4,8.7)};
		\draw  plot[smooth, tension=.7] coordinates {(v3) (3.6,6.5-8) (6.3,6.8-8) (7.9,7.4-8) (9.4,8.7-8)};
		\draw[-latex,dashed](5,1)node [below]{$\Sigma\times\{1\}$}--(6,2);
		\draw[-latex,dashed](5,0.5)--(6,-2);
		\draw[-latex,dashed](1,0.5)node [below]{$\Sigma\times\{R\}$}--(2,2);
		\draw[-latex,dashed](1,0)--(2,-2);
		\draw[-latex,dashed](2.7,0)node [right]{Neck-stretching}--(2,0);
		\draw[-latex,dashed](5.4,0)--(6,0);
		\draw[|<-](2,-6.5)--(4,-6.5)node[right]{$(H_W,J_W)\equiv(H_V,J_V)$};
		\draw[-](7.9,-6.5)--(10,-6.5);
		\draw[red] node at (9,3){$x_1^+\otimes_W x_2^+=x^-_3+p$};
		\draw[red] node at (9,-5){$x_1^+\otimes_V x_2^+=x^-_3$};
		\end{tikzpicture}
		\caption{Pair-of-pants product for different fillings $W$ and $V$ and the natural identification of $C_n^{<D}(W,H_W,J_W)$ and $C_n^{<D}(V,H_V,J_V)$.
			On the chain level, pair-of-pants product are the same, up to a difference in $O_n^{<D}$. Morally, the difference vanishes when elements are quotiented by $O_n^{<D}$; $I_+(W)$ is therefore isomorphic to $I_+(V)$.
		}\label{Fig:difference in pair of pants product}
	\end{figure}

	\begin{proof}[Proof of Theorem~\ref{Main Thm}]
		As shown in Figure~\ref{Fig:difference in pair of pants product}, let $W,V$ be two different Liouville fillings for a strongly ADC contact manifold $(\Sigma,\lambda)$.
		Suppose $H_W^D,H_V^D$ are Hamiltonians (as in Subsection~\ref{sss:Ad Hamiltonian}) whose slopes at infinity are $D \notin Spec(\Sigma,\lambda)$. We can further assume that they have unique minima which are denoted by $p,q$ respectively. Note that any element
		$x \in O_*(W,H_W,J_W)$ has action $\mathcal{A}_{H_W}(x)<D$. As shown above, $O_*(W,H_W,J_W)=SC_*^+(W,H_W,J_W)$.

		After neck-stretching, we can assume that 
		\[ (H_W,J_W)|_{\Sigma \times [R,\infty)}\equiv(H_V,J_V)|_{\Sigma \times [R,\infty)}\]
		So we have $O_*(W,H_W,J_W)=O_*(W,H_V,J_V)$.
		Proposition~\ref{lazarev proof for differential} shows that Floer cylinders with asymptotes in  $SC_*^+(W,H_W,J_W)$ are entirely contained in $\Sigma \times [R,\infty)$. Therefore Floer differentials of 	$O_*(W,H_W,J_W)$ and $O_*(W,H_V,J_V)$ coincide.
		We will suppress $W$ and $V$ in the notation and denote them by $(O_*(H,J),\p)$($(O_*^{<K}(H,J),\p)$ if it is filtered above by action $K$). We have the pair-of-pants product  $\otimes_W$ on 
		\[SC_n^{<D/2}(W,H^D,J)=O_n^{<D/2} (H^D,J)\oplus C_n^{<D/2}
		=O_n^{<D/2} (H^D,J)\oplus \Z_2<p>
		\]
		defined as
		\begin{align}
			SC_n^{<D/2}(W,H^D,J) \otimes SC_n^{<D/2}(W,H^D,J) &\rightarrow SC_n^{<D}(W,H^D,J)\\
			(x,y)&\mapsto x\otimes_W y.
		\end{align}
		By Proposition~\ref{pair of pants product --nonberaking }, $\otimes_W$ coincides with $\otimes_V$ on components in $O_n^{<D} (H^D,J)$, that is, for $x,y \in O_n^{<D/2} (H^D,J)$, $x\otimes_W y=z+\delta_W(x,y)$, where $z \in O_n^{<D} (H^D,J)$ and $\delta_W(x,y) \in \Z_2<p>$. Note that $\delta_W(x,y)$ is closed in $SC_n^{<D}(H,J)$. Likewise, we have $x\otimes_V y=z+\delta_V(x,y)$, where $z \in O_n^{<D} (H^D,J)$ and $\delta_V(x,y) \in \Z_2<q>$.
		Now for any $\alpha \in I^{<D/2}(W,H_W,J_W) \subset SH_n^{<D/2}(W,H_W,J_W)$, we have
		\[\alpha=[x+\epsilon p]_W=[x]_W+\epsilon[p]_W=[x]_W+\epsilon e_{H_W} \] where $x\in O_n^{D/2}(H,J),\epsilon=0\, \textrm{or}\, 1$. 
		$ x\otimes_W x=z+\delta_W(x,x)$ implies 
		\[ \alpha^2-\alpha=[x]_W^2+\epsilon^2 e_{H_W}^2-[x]_W-\epsilon e_{H_W}=[z-x+\delta(x,x)]_W=[z-x]_W+[\delta(x,x)]_W\]
		So $\alpha \in I^{<D/2}(W,H_W,J_W) $ is equivalent to  
		\[ [z-x]_W+[\delta(x,x)]_W \in H^0(W).\]
		But since $[\delta(x,x)]_W \in H^0(W)$, $\alpha \in I^{<D/2}(W,H_W,J_W) $ is equivalent to $[z-x]_W \in H^0(W)$. Hence for $x\in O_n^{<D/2}(H,J), \p(x)=0$ ($\p$ is Floer differential on $ O_n^{<D/2}(H,J)$),
		\[[x]_W^+ \in I_+(W,H_W,J_W) \Longleftrightarrow [z-x]_W^+ \in SH_n^{+,<D/2}(H,J)\]
		where $[y]_W^+ $ stands for the equivalence class of $y\in O_*(H,J)$ in $SH_n^+(W,H_W,J_W)$.
		We can prove the same results for $V$ similarly. Therefore we have an isomorphism between $I_+^{<D/2}(W,H_W,J_W)$ and $I_+^{<D/2}(V,H_V,J_V)$:
		$$ [x]^+_W \mapsto [x]^+_V.$$
		Since $SH_n^{+,<D/2}(W,H_W,J_W),  SH_n^{+,<D/2}(V,H_V,J_V)$ can be defined by $(\Sigma\times[R,\infty), H, J) $ as the Floer cylinder never enters the interior. Therefore we have the identity
		\[ SH_*^{+,<D/2}(W,H_W,J_W)\cong H_*(O_*^{<D/2}(H,J),\p)\cong SH_n^{+,<D/2}(V,H_V,J_V),
		\]
		the inclusion
		map $SH_n^{+,<D/2}(W,H_W,J_W) \to SH_n^{+}(W,H_W,J_W)$ commute with the above isomorphism,
		\[\begin{tikzcd}
		I_+^{<D/2}(W,H_W,J_W) \arrow[r,"i"] \arrow[d,"\cong"]
		& SH_n^{+,<D/2}(W,H_W,J_V) \arrow[r,"i"]\arrow[d ,"\cong"] & SH_n^{+}(W,H_W,J_V) \arrow[d,"\cong"] \\
		I_+^{<D/2}(V,H_V,J_V) \arrow[r ,"i"]
		& SH_n^{+,<D/2}(V,H_V,J_V)\arrow[r,"i"]&SH_n^{+}(V,H_V,J_V)
		\end{tikzcd}
		\]
		and we can therefore take the direct limit with respect to $H_W$. Since we already know $SH_*^+(W)$ is isomorphic to $SH_*^+(V)$ by Theorem~\ref{Lazarev main}, it follows that 
		$I_+(W)\cong I_+(V)$.

	\end{proof}
	
	\begin{remark}
		We can also proceed exactly as in proof of proposition 3.8 in  \cite{lazarev2016contact}. The key point is to use the \textit{essential complex} as defined in that proof. 
	\end{remark}

		\section{Brieskorn Manifolds}

		\subsection{Definition of Brieskorn manifolds}\label{brieskorn mfld}
		Let $\mathbf{a}=(a_0,a_1,\cdots,a_n)$ be an $(n+1)$-tuple of integers $a_i>1,\mathbf{z}:=(z_0,z_1,\cdots,z_n)\in \mathbb{C}^{n+1}$, and set $f(\mathbf{z}):=z_0^{a_0}+z_1^{a_1}+\cdots+z_n^{a_n}$, we define \textit{Brieskorn Variety} as
		\begin{equation}
		V_\fa(t):=\{(z_0,z_1,\cdots,z_n)\in \mathbb{C}^{n+1}| f(\mathbf{z})=t\} \quad \text{for each}\quad  t\in \C. 
		\end{equation}
		We will often suppress $\fa$ when it causes no confusion,
		and define $X_t^s=V(t)\cap B(s)$.
		
		Further, with $S^{2n+1}$ denoting the unit sphere in $\mathbb{C}^{n+1}$, we define the \textit{Brieskorn Manifold} as the intersection of Brieskorn Variety $V_\fa(0)$with the unit sphere:
		\[\Sigma(\mathbf{a}):=V_\fa(0)\cap S^{2n+1}.\]
		
		\begin{lemma}[Lemma 96 \cite{fauck2016rabinowitz}, Lemma 7.1.1 \cite{geiges2008introduction}]\label{lemma}
			$\Sigma(\mathbf{a})$ and $V_\fa(t),t\neq 0$ are smooth manifolds.
		\end{lemma}
		\begin{proof}
			We set $\rho(z):=||z||^2=\sum z_k\bar{z_k}$ and consider the maps
			\[f:\mathbb{C}^{n+1} \to \mathbb{C} \qquad and \qquad (f,\rho):\mathbb{C}^{n+1}\to \mathbb{C\times R}\]
			Since $V_\fa(t)=f^{-1}(t)$ and $\Sigma(\mathbf{a})=(f,\rho)^{-1}(0,1)$, it suffices to show that $t$ (respectively $(0,1)$) are regular values. With a little Wirtinger calculus (and using the fact that $f$ is holomorphic)
			we find the Jacobian matrix
			\[ D(f,\rho)=
			\begin{bmatrix}

			a_0z_0^{a_0-1} &\cdots & a_nz_n^{a_n-1} & 0 & \cdots & 0 \\
			0 & \cdots & 0 & a_0\bar{z_0}^{a_0-1} & \cdots &  a_n\bar{z_n}^{a_n-1}\\
			\bar{z_0}& \cdots& \bar{z_n} & z_0 & \cdots &z_n
			\end{bmatrix}
			\]
			
			For $\mathbf{z}\neq 0$ the first two rows of $D(f,\rho)$ are linearly independent, which implies that $\epsilon\neq0$ is a regular value of $f$. If $\mathbf{z}$ is a point where this matrix has rank smaller than 3, there exists a non-zero complex number $\lambda$ such that $\bar{z_k}=\lambda a_k z_k^{a_k-1}$ for all $k$  and
			hence
			\[
			\sum\limits_{k=0}^{n}\frac{z_k\bar{z_k}}{a_k}=\lambda\sum\limits_{k=0}^n z_k^{a_k}=\lambda\cdot f(\mathbf{z})
			\]
			This equality is incompatible with the conditions $\rho(\mathbf{z})=1$ and $f(\mathbf{z})=0$ for a point $\mathbf{z}\in \Sigma(\mathbf{a}) $.

		\end{proof}
	
		\subsection{Topology of Brieskorn manifolds}

		Now, we give some topological facts about Brieskorn manifolds without proof. 
		\begin{prop}[Theorem 5.2  \cite{milnor2016singular} ]\label{high connectedness of brieskorn manifold}
			A Brieskorn manifold $\Sigma(\mathbf{a})^{2n-1}$ is $(n-2)$-connected. 
		\end{prop}

		\subsection{Trivialization and Conley-Zehnder index}
		Let us consider on $\mathbb{C}^{n+1}$ the following Hermitian form given by
		\[<\xi,\zeta>_\mathbf{a}:=\frac{1}{2}\sum_{k=0}^{n}a_k\xi_k\bar{\zeta_k}.\]
		
		It defines a symplectic 2-form
		\[\omega_\mathbf{a}:=\frac{i}{4}\sum_{k=0}^{n}a_kdz_k\wedge d\bar{z_k}.\]
		
		Notice that $Y_\lambda(\mathbf{z}):=\frac{\mathbf{z}}{2}$ is a Liouville vector field for $\omega_\mathbf{a}$, with the corresponding 1-form 
		\[\lambda_\mathbf{a}:=\omega_\mathbf{a}(Y_\lambda,\cdot)=\frac{i}{8}\sum_{k=0}^{n}a_k(z_kd\bar{z_k}-\bar{z_k}dz_k).\]

		\begin{prop}[Proposition 97 \cite{fauck2016rabinowitz}, \cite{lutz1976structures}]\label{Lutz}
			The restriction $\alpha_a:=\lambda_a|_\Sigma$ is a contact form on $\Sigma(\mathbf{a})$ with Reeb vector field $R_\mathbf{a}$ given by
			\[R_\mathbf{a}=4i(\frac{z_0}{a_0},\frac{z_1}{a_1},\cdots,\frac{z_n}{a_n}).\]
			
		\end{prop}
		\begin{proof}
			The gradient of $f$  with respect to $<\cdot,\cdot>_\mathbf{a}$ is given by
			\[\nabla_\mathbf{a}f:=2(\bar{z_0}^{a_0-1},\bar{z_1}^{a_1-1},\cdots,\bar{z_n}^{a_n-1}).\] 
			The Liouville vector field $Y_V$ of the restricted 1-form $\lambda_\mathbf{a}|_{V_\fa(0)}$ with respect to the restricted symplectic form $\omega_\mathbf{a}|_{V_\fa(0)} $is given by 
			\[Y_V:=Y_\lambda-\frac{<\nabla_\mathbf{a}f,Y_\lambda>_\mathbf{a}}{||\nabla_\mathbf{a}f||^2_\mathbf{a}}\cdot\nabla_\mathbf{a}f.\]
			Note that $TV_\fa(t)=\ker df=\ker<\nabla_\mathbf{a}f,\cdot>_\mathbf{a}$, which shows that $Y_V\in TV_\fa(0)$. Furthermore, we have for any $\xi \in TV_\fa(0)$,
			\[\omega_\mathbf{a}(Y_V,\xi)=\omega_\mathbf{a}(Y_\lambda,\xi)-\frac{<\nabla_\mathbf{a}f,Y_\lambda>_\mathbf{a}}{||\nabla_\mathbf{a}f||^2_\mathbf{a}}\cdot\omega_\mathbf{a}(\nabla_\mathbf{a}f,\xi)=\lambda_\mathbf{a}(\xi)+\frac{<\nabla_\mathbf{a}f,Y_\lambda>_\mathbf{a}}{||\nabla_\mathbf{a}f||^2_\mathbf{a}}\cdot \underbrace{Im<\nabla_\mathbf{a}f,\xi>)_\mathbf{a}}_{=0}=\lambda_\mathbf{a}(\xi)\]
			So this indicates that $Y_V$ is the Liouville vector field for the pair $(\omega_\mathbf{a}|_{V_\fa(0)} ,\lambda_\mathbf{a}|_{V_\fa(0)})$. Now notice that $d\rho=\sum\limits_{k=0}^{n}\bar{z_k}dz_k+z_kd\bar{z_k}$ ($\rho$ is defined in the proof of lemma \ref{lemma}) and we have
			\[d\rho(Y_V)=\sum \frac{z_k\bar{z_k}}{2}-\frac{<\nabla_\mathbf{a}f,Y_\lambda>_\mathbf{a}}{||\nabla_\mathbf{a}f||^2_\mathbf{a}}\sum 2\bar{z_k}\cdot\bar{z_k}^{a_k-1}=\frac{\rho(\mathbf{z})}{2}
			-\frac{<\nabla_\mathbf{a}f,Y_\lambda>_\mathbf{a}}{||\nabla_\mathbf{a}f||^2_\mathbf{a}}\cdot2\bar{f(\mathbf{z})}=\frac{1}{2}.\]
			since $\rho(\mathbf{z})=1$ and $f(\mathbf{z})=0$. It follows that $Y_V$ points out of the unit sphere and hence out of $\Sigma(\mathbf{a})$ in $V_\fa(0)$. 
			It follows that  $\Sigma(\mathbf{a})$ is a contact hypersurface in $V_\fa(0)$.
			Now we are going to check that $R_\mathbf{a}$ is the Reeb vector field of $\alpha_\mathbf{a}$. For any $\mathbf{z}\in \Sigma(\mathbf{a})$, we have
			\[<R_\mathbf{a},\nabla_\mathbf{a}f>_\mathbf{a}=4i\sum_{k=0}^{n}z_k^{a_k}=4if(\mathbf{z})=0,\]
			\[d\rho(R_\mathbf{a})=\sum_{k=0}^{n}z_k(-4i)\bar{z_k}+\bar{z_k}4iz_k=0
			\]
			The two equations above shows that $R_\mathbf{a}$ is a tangent vector. We also have
			\[\alpha_\mathbf{a}(R_\mathbf{a})=\lambda_\mathbf{a}(R_\mathbf{a})=\frac{i}{8}\sum_{k=0}^{n}a_k(\frac{4i}{a_k}\bar{z_k}-\bar{z_k}\frac{4i}{a_k}z_k)=\rho(\mathbf{z})=1,\]
			\[\iota_{R_\mathbf{a}}d\alpha_\mathbf{a}=\frac{i}{4}\sum_{k=0}^{n}(4i\frac{z_k}{a_k}d\bar{z_k}-(-4i)\frac{\bar{z_k}}{a_k})=-\sum_{k=0}^{n}(z_kd\bar{z_k}+\bar{z_k}dz_k)=-d\rho.\]
			The latter form is zero for vectors in $T\Sigma(\mathbf{a})$, therefore, $R_\mathbf{a}$ is the Reeb vector field.
		\end{proof}
		
		\begin{prop}[Corollary 98 \cite{fauck2016rabinowitz}]\label{trivialization of the symplectic complement}
			The symplectic complement $\xi_\mathbf{a}^{\bot}$ with respect to $\omega_\mathbf{a}$ of the contact structure $\xi_\mathbf{a}:= \ker         \alpha_\mathbf{a}$ inside $\C^{n+1}$ is symplectically trivialized by the following 4 vector fields:
			\begin{itemize}
				\item 	$X_1:=\frac{\nabla_\mathbf{a}f}{||\nabla_\mathbf{a}f||_\mathbf{a}}$
				\item   $Y_1:=i\cdot X_1$
				\item   $X_2:=Y_V$
				\item   $Y_2:=R_\mathbf{a}$.
			\end{itemize}
			
		\end{prop}
		\begin{proof}
			 $X_1,Y_1$ generate the complex complement of $TV_\fa(0)$ while $X_2,Y_2$ generate the symplectic complement of $\xi_\mathbf{a}$ in $TV_\fa(0)$, so we have
			\[\omega_\mathbf{a}(X_1,X_2)=\omega_\mathbf{a}(X_1,Y_2)=\omega_\mathbf{a}(Y_1,X_2)=\omega_\mathbf{a}(Y_1,Y_2)=0.\] 
			Meanwhile we have 
			\[\omega_\mathbf{a}(X_1,Y_1)=1,\quad\omega_\mathbf{a}(X_2,Y_2)=\lambda_\mathbf{a}(R_\mathbf{a})=1.\]
			The latter equation comes from the proof of proposition~\ref{Lutz}.
		\end{proof}

		The Reeb vector field $R_\mathbf{a}=4i(\frac{z_0}{a_0},\frac{z_1}{a_1},\cdots,\frac{z_n}{a_n})$ generates the following flow:
		\[\psi_\mathbf{a}^t(\mathbf{z})=(e^\frac{4it}{a_0}\cdot z_0,\cdots, e^\frac{4it}{a_n}\cdot z_n)
		\]
		The submanifolds $\Sigma_T$ of period $T\in \pi\mathbb{Z}/2$ are given by
		\[\Sigma_T=\Big\{\mathbf{z}\in \Sigma(\mathbf{a})\,\Big|\,z_k=0 \, \text{ if } \, \frac{T}{a_k} \in \pi\mathbb{Z}/2\,\Big\}.
		\]
		$\Sigma_T$ is not empty if and only if the relation $\frac{T}{a_k} \in \pi\mathbb{Z}/2$ is satisfied by at least $2$ different $k$, as $\mathbf{z} \in \Sigma(\mathbf{a})$ has at least $2$ non-zero entries.
		Note that $\Sigma_T$ is the intersection $\Sigma(\mathbf{a})\cap V(\mathbf{a},T)$, where $V(\mathbf{a},T)$ denotes the complex linear subspace 
		\[V(\mathbf{a},T):=\Big\{\mathbf{z}\in \mathbb{C}^{n+1}\,\Big|\,z_k=0 \text{ if } \frac{T}{a_k} \notin \frac{\pi}{2}\mathbb{Z}\,\Big\}
		\]
		whose complex dimension is given by
		\[\dim_{\mathbb{C}} V(\mathbf{a},T):=\Bigg|\Big\{k\,\Big|\, 0\leq k\leq n,\,\frac{T}{a_k} \in \frac{\pi}{2}\mathbb{Z}\, \Big\}\Bigg|,
		\]
		where $|S|$ denotes the cardinality of the set $S$.
		We notice that $\Sigma_T$ is therefore isomorphic to the Brieskorn manifold $\Sigma(\mathbf{a}(T))$, where 
		\[ \mathbf{a}(T)=(a_0,\cdots,\hat{a_i},\cdots,a_n)
		\]
		is a subset of $\mathbf{a}$. Here $\hat{a_i}$ means the term $a_i$ is omitted, when 
		$\frac{T}{a_i} \notin \frac{\pi}{2}\mathbb{Z}$.
		The differential of $\phi_\mathbf{a}$ at time $t$ is given by
		\[D\psi_\mathbf{a}^t=diag\big(e^{4it/a_0},\cdots,e^{4it/a_n}\big)
		\]
		It follows that
		\[\ker(D_\mathbf{z}\psi_\mathbf{a}^T\big|_{T_\mathbf{z}\Sigma(\mathbf{a})}-id)=T_\mathbf{z}\Sigma(\mathbf{a})\cap V(\mathbf{a},T)=T_\mathbf{z}\Sigma_T
		\]
		Therefore $\Sigma_T$ is Morse-Bott submanifold.

		The calculation of the indices of all closed Reeb orbits can be found in various literature, see  \cite{kwon2016brieskorn},  \cite{ustilovsky1999contact}. We conclude this subsection with the following proposition: 
		\begin{prop}[ \cite{kwon2016brieskorn}, \cite{fauck2016rabinowitz}]\label{index calculation}
			Let $\gamma\in \Sigma(\mathbf{a})$ be a fractional Reeb of period $t$. We have
			\[\mu_{CZ}(\gamma)=\sum_{k=0}^{n}\Bigg(\Bigg\lfloor\frac{2t}{a_k\pi}\Bigg\rfloor+\Bigg\lceil\frac{2t}{a_k\pi}\Bigg\rceil\Bigg)-\Bigg(\Bigg\lfloor\frac{2t}{\pi}\Bigg\rfloor+\Bigg\lceil\frac{2t}{\pi}\Bigg\rceil\Bigg)
			\]
		\end{prop}
		
		\begin{proof}
			First we notice that the indices are canonically defined when $n\geq 4$, by Proposition~\ref{high connectedness of brieskorn manifold}.
			Recall the Reeb vector field in Proposition~\ref{Lutz}, $R_\mathbf{a}=4i(\frac{z_0}{a_0},\frac{z_1}{a_1},\cdots,\frac{z_n}{a_n})$. The associated Reeb flow is 	
			\[\psi_\mathbf{a}^t(\mathbf{z})=(e^\frac{4it}{a_0}\cdot z_0,\cdots, e^\frac{4it}{a_n}\cdot z_n).
			\]
			We regard this as a flow on $\mathbb{C}^{n+1}$ as opposed to $\Sigma(\mathbf{a})$. This perspective gives us the advantage of calculating the indices directly on $\mathbb{C}^{n+1}$. If we take the standard trivialization of $T\mathbb{C}^{n+1}$, then the linearized return map is
			\[D\psi_\mathbf{a}^t=diag\big(e^{4it/a_0},\cdots,e^{4it/a_n}\big)=:\Psi_t.
			\]
			By Proposition~\ref{trivialization of the symplectic complement}, we have the trivialization of the symplectic complement $\xi_\mathbf{a}^\perp$. The linearized return map of the flow on  $\xi_\mathbf{a}^\perp$  gives:
			\begin{itemize}
				\item $D\psi_\mathbf{a}^t(X_1(\mathbf{z}))=e^{4it}\cdot X_1(\psi_\mathbf{a}^t(\mathbf{z}))$,
				\item $D\psi_\mathbf{a}^t(Y_1(\mathbf{z}))=e^{4it}\cdot Y_1(\psi_\mathbf{a}^t(\mathbf{z}))$,
				\item $D\psi_\mathbf{a}^t(X_2(\mathbf{z}))=X_2(\psi_\mathbf{a}^t(\mathbf{z}))$,
				\item $D\psi_\mathbf{a}^t(Y_2(\mathbf{z}))=Y_2(\psi_\mathbf{a}^t(\mathbf{z}))$.
			\end{itemize}
			It follows that the linearized  map of $D\psi_\mathbf{a}^t$ on $\xi_\mathbf{a}^\perp$ under the prescribed trivialization is the diagonal matrix:
			\[
			\Psi_2^t:=
			\left[ {\begin{array}{cc}
				e^{4it} & 0 \\
				0 & 1 \\
				\end{array} } \right]
			\]
			A trivialization of $\xi_\mathbf{a}$ along the Reeb orbit gives us the linearization of $\Psi_1^t$ of $\psi_\mathbf{a}^t$ on  $\xi_\mathbf{a}$. Any trivialization of $\xi_\mathbf{a}$ and $\xi_\mathbf{a}^\perp$ combined gives rise to a trivialization of $T\mathbb{C}^{n+1}$, which is homotopic to the standard one.
			Therefore by the product property of the Conley-Zehnder index and using remark~\ref{Index formula}, we find that
			
			\begin{align*}
			\mu_{CZ}(\gamma)=\mu_{CZ}(\Psi_1)=&\mu_{CZ}(\Psi)-\mu_{CZ}(\Psi_2)\\
			=&\sum_{k=0}^{n}\Bigg(\Bigg\lfloor\frac{2t}{a_k\pi}\Bigg\rfloor+\Bigg\lceil\frac{2t}{a_k\pi}\Bigg\rceil\Bigg)-\Bigg(\Bigg\lfloor\frac{2t}{\pi}\Bigg\rfloor+\Bigg\lceil\frac{2t}{\pi}\Bigg\rceil\Bigg).
			\end{align*}
		\end{proof}
		
		\begin{lemma}\label{minimal index}
			Let $\mathbf{a}=(a_0,a_1,a_2,\cdots,a_n)$, where the $a_i's$ are positive integers, and $\sum\frac{1}{a_k}\geq 1$. Then the following function $f_\mathbf{a}:\mathbb{R}_+\to \mathbb{Z}$, 
			\[f_\mathbf{a}(x)=\sum_{k=0}^{n}\Bigg(\Bigg\lfloor\frac{x}{a_k}\Bigg\rfloor+\Bigg\lceil\frac{x}{a_k}\Bigg\rceil\Bigg)-\Big(\Big\lfloor x\Big\rfloor+\Big\lceil x\Big\rceil\Big)\]
			has a minimum, denoted by $m(\mathbf{a})$. In particular, if $\mathbf{a}=(2,2,2,a_1,\cdots,a_n)$, then $m(\mathbf{a})\geq 2$, where $a_k's$ are positive integers, $n\geq 2$.
		\end{lemma}
		\begin{proof}
			We notice that $2x-1<\lfloor x\rfloor+\lceil x \rceil <2x+1$, we have \[
			f_\mathbf{a}(x)>2\big( \sum_{k=0}^{n}\frac{1}{a_k}-1\big)x-n-1\geq -n-1,
			\]
			which proves the first part. For the second part, 
			we have 
			\[f_\mathbf{a}(x)=3\Bigg(\Bigg\lfloor\frac{x}{2}\Bigg\rfloor+\Bigg\lceil\frac{x}{2}\Bigg\rceil\Bigg)+\sum_{k=1}^{n}\Bigg(\Bigg\lfloor\frac{x}{p_k}\Bigg\rfloor+\Bigg\lceil\frac{x}{p_k}\Bigg\rceil\Bigg)-\Big(\Big\lfloor x\Big\rfloor+\Big\lceil x\Big\rceil\Big)
			\]
			Note that $f_\mathbf{a}(x+2)\geq f_\mathbf{a}(x)+2$, so the minimum is obtained in $x\in (0,2]$. On this interval, we have $f_\mathbf{a}(x)=3\Bigg(\Bigg\lfloor\frac{x}{2}\Bigg\rfloor+\Bigg\lceil\frac{x}{2}\Bigg\rceil\Bigg)+n-\Big(\Big\lfloor x \Big\rfloor+\Big\lceil x \Big\rceil\Big)$, which is 
			\[f_\mathbf{a}(x)=
			\begin{cases}
			2+n \quad x \in (0,1),\\
			1+n \quad x=1,\\
			n \quad x\in (1,2),\\
			4+n \quad x=2.
			\end{cases}
			\]
			hence our conclusion.
		\end{proof}
		%\begin{remark}\label{index for contractible orbits }
		%	If we define $F_\mathbf{a}(x)=f_\mathbf{a}(x)|_{\mathbb{Z}^+}$, and $M(\mathbf{a})$ be the minimum of $F_\mathbf{a}(x)$. Then $M(\mathbf{a})=m(\mathbf{a})+1$ for $\mathbf{a}=(2,2,2,2,2,p_1,\cdots,p_n)$.
		%\end{remark}

		\section{Exotic contact manifolds}

		\subsection{Liouville domains admitting group actions}\label{cover}
		We need to find a Liouville domain $(W,\lambda)$ with the contact manifold
		$\Sigma(\mathbf{a})$ as its boundary. While $V_\fa(0)$ has a singularity at the origin, $V_\fa(\epsilon)$ is smooth. Therefore we will follow  Alexander Fauck's approach \cite{fauck2016rabinowitz} to overcome this by constructing an interpolation between $V_\fa(0)$ and $V_\fa(\epsilon)$. First, we choose a smooth monotone decreasing cut-off function $\beta\in C^\infty(\mathbb{R})$ with $\beta(x)=1,x\leq \frac{1}{4}$ and $\beta(x)=0, x\geq \frac{3}{4}$. Then we define (we will often omit $\fa$)
		\[U_\fa(\epsilon):=\{\mathbf{z}\in \mathbb{C}^{n+1}|z_0^{a_0}+\cdots+z_n^{a_n}=\epsilon\cdot\beta(||\mathbf{z}||^2)\}.\]
		 Let 
		\[W_\epsilon^s:=U_\epsilon\cap B(s)\]
		we have
		\begin{prop}[Proposition 99 \cite{fauck2016rabinowitz}]\label{Liouville form}
			For sufficiently small $\epsilon$,$(X_\epsilon^1,\lambda)$ is a Liouville domain with boundary $(\Sigma(\mathbf{a}),\alpha_\mathbf{a})$ and vanishing first Chern class.
		\end{prop}

		Moreover, we have a cyclic group 
		\[C(L):=\{\, e^\frac{2\pi ki}{L} \in \C \, |\,  k\in\Z\, \}=<\zeta>
		\]
		 acting on $(\mathbb{C}^{n+1})^*$, which is generated by :
		\begin{align*}
		\zeta_*: (\mathbb{C}^{n+1})^* & \longrightarrow  (\mathbb{C}^{n+1})^*\\
		(z_0,z_1,\cdots,z_n) & \mapsto  (z_0\zeta^{b_0},z_1\zeta^{b_1},\cdots,z_n\zeta^{b_n})
		\end{align*}
		
		where $L:=\textrm{lcm}\,a_j, b_j:=L/a_j, \zeta:=e^\frac{2\pi i}{L}$. 
		We can easily see that the 1-form $\lambda_\mathbf{a}$ is $C(L)$-invariant.

		We can restrict this group action to the subsets of $(\mathbb{C}^{n+1})^*$ mentioned above and obtain a $C(L)-$action on the manifolds $X_\epsilon^s$ and $W_\epsilon^s$.
		By definition, $X_\epsilon^{1/2}= U(\epsilon)\cap B(1/2)=V(\epsilon)\cap B(1/2)=W_\epsilon^{1/2}$.
		We have the following proposition:
		\begin{prop}
			For sufficiently small $\epsilon>0$, there is a $C(L)$-equivariant isotopy between the following pairs of Liouville domains:
			\begin{itemize}
				\item $X_\epsilon^1$ and  $X_\epsilon^{1/2}$,
				\item$W_\epsilon^1$ and  and $W_\epsilon^{1/2}$.
			\end{itemize}
		\end{prop}
	
		\begin{proof}
			We only give a proof for the existence of a $C(L)$-equivariant isotopy between $W_\epsilon^1$ and  $W_\epsilon^{1/2}$. We can prove the same results for $X_\epsilon^1$ and $X_\epsilon^{1/2}$ verbatim.
			Consider the function $\rho (\mathbf{z})=||\mathbf{z}||^2$ on $V_\epsilon$. If  for sufficiently small $\epsilon$, the critical values of $\rho$ restricted to $W_\epsilon^1$ are less that $1/4$, then we are done, by lemma~\ref{morse}. Indeed, we have $f_\epsilon(\mathbf{z}):=f(\mathbf{z})-\epsilon\cdot\beta(||\mathbf{z}||^2)$ on $\mathbb{C}^{n+1}$ and its differential is given by
			\[Df_\epsilon=Df-\epsilon\cdot\beta^{'}(||\mathbf{z}||^2)\cdot D\rho\]
			so the map \[(f_\epsilon,\rho): \mathbb{C}^{n+1}\rightarrow \mathbb{C\times R}\] has Jacobian matrix $(Df-\epsilon\cdot\beta^{'}(||\mathbf{z}||^2)\cdot D\rho,D\rho)$, which has the same rank as $(Df,D\rho)$. So by the same argument in the proof of lemma~\ref{lemma}, if $\mathbf{z}$ is a point where the Jacobian is not full rank, then we have for some complex number  $\lambda$, $\bar{z_k}=\lambda a_k z_k^{a_k-1}$ for all $k$.
			For $||\mathbf{z}||\geq 1/2$, we have $|z_{k_0}|\geq \frac{1}{2\sqrt{n}}$ for some $k_0$, so
			\[|z_{k_0}|=|\lambda|\cdot a_{k_0}\cdot |z_{k_0}|^{a_{k_0}-1}\]
			i.e,
			\begin{equation}\label{equ for lambda}
		|\lambda|=\frac{|z_{k_0}|^{2-a_{k_0}}}{a_{k_0}} \leq \frac{(2\sqrt{n})^{a_{k_0}-2}}{a_{k_0}}\leq C(\mathbf{a})
			\end{equation}
		
			where $C(\mathbf{a}):= \underset{0\leq k \leq n}{\max} \{\frac{(2\sqrt{n})^{a_{k}-2}}{a_{k}}\} $ only depends on $\mathbf{a}$ and $n$.
			Meanwhile, we have
			\begin{equation}\label{equ}
			\sum\limits_{k=0}^{n}\frac{z_k\bar{z_k}}{a_k}=\lambda\sum\limits_{k=0}^n z_k^{a_k}=\lambda\cdot f(\mathbf{z})=\lambda\cdot\epsilon\beta(||\mathbf{z}||^2)
			\end{equation}

			Combining equations~(\ref{equ}) and (\ref{equ for lambda}), we have
			\begin{equation}\label{equ upper}
			\sum\limits_{k=0}^{n}\frac{z_k\bar{z_k}}{a_k}=\lambda\cdot\epsilon\beta(||\mathbf{z}||^2)\leq \epsilon\cdot C(\mathbf{a})
			\end{equation}

		    On the other, 
			\begin{equation}\label{equ 2}
			\sum\limits_{k=0}^{n}\frac{z_k\bar{z_k}}{a_k}\geq \frac{1}{\max\{a_j\}}\sum_{k=0}^{n}z_k\bar{z_k}\geq \frac{1}{\max\{a_j\}}\cdot ||\mathbf{z}||^2=\frac{1}{4\max\{a_j\}}
			\end{equation}
			
			Equations~(\ref{equ 2}) and (\ref{equ upper}) cannot hold for sufficiently small $\epsilon$ at the same time, and therefore the function $\rho$ has no critical points in $||\mathbf{z}||\geq 1/2$, hence all critical values are less than $1/4$.
		\end{proof}
		
		\begin{lemma}[Theorem 2.2.2 \cite{nicolaescu2011invitation}]\label{morse}
			Suppose finite group $G$ acts on a manifold $M$ and $f$ is a $G$-invariant exhausting function on $M$. Moreover, assume that no critical value of $f$ is contained in $[a,b]\subset \mathbb{R}$, then there is a $G$-equivariant isotopy $\phi_t$ between the sublevel sets $M^a:=f^{-1}((-\infty,a])$ and $M^b:=f^{-1}((-\infty,b])$, and $\phi_t$ coincides with Id outside a compact set.
		\end{lemma}
		
		\begin{proof}
			
			Since there are no critical values of $f$ in [a, b] and the sublevel sets 
			are compact, we
			deduce that there exists $\epsilon > 0$ such that
			\[\{a-\epsilon<f<b+\epsilon\}\subset M\setminus Crit(f).\]
			First we fix a gradient-like $G$-invariant vector field $Y$ and construct a compactly supported $G$-equivariant smooth function
			\[g:M\to [0,\infty)\]
			such that
			\[g(x)=
			\begin{cases}
			\frac{1}{|Yf|},\quad & a\leq f(x)\leq b,\\
			0,&f(x)\notin (a-\epsilon,b+\epsilon).
			\end{cases}
			\]
			We can now construct a $G$-invariant vector field $X:=gY$ on $M$ and we denote by
			\[\phi:\mathbb{R}\times M\to M,\quad (t,x)\to \phi_t(x)
			\]
			the flow generated by $X$. Clearly the flow commutes with the group action, so $\phi_t$ is $G$-equivariant. If $u(t)$ is an integral curve of $X$, then differentiating $f$ along $u(t)$ in the region $\{a\leq f\leq  b \}$ and get
			\[ \frac{df}{dt}=Xf=
			\frac{1}{Yf}Yf=1\] 
			This implies
			\[\phi_{b-a}(M^a)=M^b\]
			and $\phi_t$ is identity outside the region $\{a-\epsilon<f<b+\epsilon\}$.
		\end{proof}
		
		\begin{remark}\label{stein}
			
				 By proposition~\ref{Liouville form},   
				$(X_\epsilon^{1/2},\Phi_{t}^*\lambda )$ is a family of $C(L)-$equivariant Liouville structures.
			Then by corollary~\ref{liouville homotopy},  we have $(X_\epsilon^{1/2},\lambda)$ is $C(L)-$ equivariant Liouville isomorphic to $(X_\epsilon^1,\lambda)$. By the same token, $(W_\epsilon^1,\lambda)$ is  $C(L)-$ equivariant Liouville isomorphic to $(W_\epsilon^{1/2},\lambda)$ and therefore to $(X_\epsilon^1,\lambda)$.
		\end{remark}
	    Let $\phi_t(\mathbf{z}):=\frac{1}{8}\sum\limits_{j=0}^{n}c_j(t)|z_j|^2$, where $c_j(t)$ is a linear interpolation such that $c_j(0)=1,c_j(1)=a_j$. It's easy to check that $\phi_t$ is plurisubharmonic on $V_\fa(\epsilon)$. Indeed, $\phi_t$ is $i$-convex on $\mathbb{C}^{n+1}$ since $\Delta \phi_t>0$ and   $V_\fa(\epsilon)$ is a smooth complex submanifold. So $(V_\fa(\epsilon),i,\phi_t)$ are $C(L)-$ equivariant Stein manifolds.
	    
	    Since $X_\epsilon^1=\phi_0^{-1}((-\infty,1/8])$,
		 $(X_\epsilon^1,i,\phi_0)$ is a $C(L)-$equivariant Stein domain. Seen as a Liouville domain, $(X_\epsilon^1,-d^\C\phi_0)$ is $C(L)-$equivariant Liouville isomorphic to $(X_\epsilon^1,\lambda)$ as follows:
		\begin{prop}\label{Liouville equivalence for stein domain}
			There is a $C(L)$-equivariant Liouville homotopy between $(X_\epsilon^1,-d^\C\phi_0)$ and $(X_\epsilon^1,\lambda)$, for sufficiently small $\epsilon$.
		\end{prop}
		\begin{proof}
		Notice that for $\lambda=-d^\C\phi_1$,
			 it suffices to prove the critical points of $\phi_t$ are contained in a compact set $\{||\mathbf{z}||\leq 1/3\}$, then  $\nabla_{\phi_t}\phi_t$ will be transversal to the boundary, and $-d^\C\phi_t$ will be a family of $C(L)-$equivariant Liouville structures on $X_\epsilon^1$, so we can conclude the result by corollary~\ref{liouville homotopy}. In the following we are going to prove that all critical points satisfy $||\mathbf{z}||\leq 1/3$.
			  Consider the map
			\[(f,\phi_t):\mathbb{C}^{n+1}\rightarrow \mathbb{C\times R} 
			\]
			Its Jacobian matrix is 
			\[ D(f,\rho)=
			\begin{bmatrix}

			a_0z_0^{a_0-1} &\cdots & a_nz_n^{a_n-1} & 0 & \cdots & 0 \\
			0 & \cdots & 0 & a_0\bar{z_0}^{a_0-1} & \cdots &  a_n\bar{z_n}^{a_n-1}\\
			\frac{1}{8}c_0(t)\bar{z_0}& \cdots& \frac{1}{8}c_n(t)\bar{z_n} & \frac{1}{8}
			c_0(t) z_0 & \cdots &\frac{1}{8}c_n(t)z_n
			\end{bmatrix}
			\]
			If $\mathbf{z}$ is a point where this matrix has rank smaller than 3, there exists a non-zero complex number $\lambda\in \C$ such that $\frac{c_k(t)}{8}\bar{z_k}=\lambda a_k z_k^{a_k-1}$ for all $k$  and 
			\begin{equation}
			\sum\limits_{k=0}^{n}\frac{c_k(t)z_k\bar{z_k}}{8a_k}=\lambda\sum\limits_{k=0}^n z_k^{a_k}=\lambda\cdot f(\mathbf{z})=\lambda\cdot\epsilon
			\end{equation}
			For $||\mathbf{z}||>\frac{1}{3}$, we have $|z_r|>\frac{1}{3(n+1)}$ for some $0\leq r\leq n$.
			So we have
			\[\frac{c_r(t)}{8}\cdot |\bar{z_r}|=|\lambda| \cdot a_r \cdot |z_r^{a_r-1}|
			\]
			i.e,
			\begin{equation*}
			|\lambda|=\frac{c_r(t)}{8a_r|z_r|^{a_r-2}}<\frac{(3(n+1))^{a_r-2}}{8}\leq C
			\end{equation*}
			where $ C=\max \limits_{0\leq i\leq n}\frac{(3(n+1))^{a_i-2}}{8}$, only depends on $\fa$.
			On one hand,we have 
			\begin{equation}\label{equation <}
			\sum\limits_{k=0}^{n}\frac{c_k(t)z_k\bar{z_k}}{8a_k}=\lambda\sum\limits_{k=0}^n z_k^{a_k}=\lambda\cdot f(\mathbf{z})=\lambda\cdot\epsilon<C\cdot \epsilon.
			\end{equation}
		
			On the other hand, we have
			\begin{equation}\label{equation >}
			\sum\limits_{k=0}^{n}\frac{c_k(t)z_k\bar{z_k}}{8a_k}\geq \sum\limits_{k=0}^{n}\frac{|z_k|^2}{8a_k}\geq \frac{1}{72 \max\limits_{0\leq i\leq n}\{a_i\}}.
			\end{equation}
			
			So for $\epsilon$ small enough, equations~(\ref{equation <}) and (\ref{equation >}) cannot both hold, which implies the critical points of $\phi_t$ is contained in $\{||\mathbf{z}||\leq 1/3\}$.
			
		\end{proof}
		
	\begin{remark}\label{Liouville completion and stein domain}
	Since we have $\phi_0(\fz)=\frac{||\fz||^2}{8}$, $\nabla_{\phi_0}\phi_0=\sum\limits_{i=0}^{n}(z_i\p\bar{z_i}+\bar{z_i}\p z_i)/2$ is complete in $\C^{n+1}$. Therefore $\phi_0$ is a completely exhausting function on $V_\fa(\epsilon)$. By the proof of Proposition~\ref{Liouville equivalence for stein domain}, all critical points of $\phi_0$ are in the interior of $X_\epsilon^1$. It follows that $V_\fa(\epsilon)$ is the completion of $X_\epsilon^1$ by matching the corresponding trajectories of the Liouville fields. 
	\end{remark}

		\subsection{Topology of manifolds $M_0$ and $M_1$}
		
		Now let's consider $C(L)$-equivariant Stein manifold
		$(\C^*,i,(\log|z|)^2/2)$ where the $C(L)-$action is multiplication given by
		\[ \R\times (\R/2\pi\Z) \longrightarrow \C^*,\quad (r,\theta)\mapsto e^{r+\theta i}\,.
		\] 
	The map	gives rise to polar coordinates form of the same Stein manifold $(\R\times S^1,j,r^2/2)$ and the Liouville vector is  
		$r\p_r$, which is complete.

		 Now we consider the product of the Stein manifolds 	$(\C^*,i,(\log|z|)^2/2)$ and $(V_\fa(\epsilon),i,\phi_0)$. It has a free $C(L)$ action as follows: 
		\begin{align*}
		\zeta_*: V_\fa(\epsilon)\times \C^* &\longrightarrow V_\fa(\epsilon)\times \C^*\\
		(z_0,z_1,\cdots,z_n,\eta)&\mapsto(z_0\zeta^{b_0},z_1\zeta^{b_1},\cdots,z_n\zeta^{b_n},\eta\zeta),
		\end{align*} where $b_i=L/a_i,\zeta \in C(L)$.
	    The product function $\phi:=(\log|z|)^2/2)+\phi_0$ is a completely exhausting $J-$convex Morse function,  and the product Stein manifold is of finite type. By abuse of the notation, we use $\phi$ to denote the function on the quotient manifold as well. Also, $M_0:=\{ \phi\leq C\}$ is a Stein domain, where $C$ is greater than all critical values of $\phi$. Hence the completion $\widehat{M_0(\fa)}=(V_\fa(\epsilon)\times \C^*)/C(L)$ since $\phi$ is complete.
	    Oftentimes we will suppress $\fa$. If we consider the Weinstein structure  instead, the Weinstein domain can be cut out in other ways, as stated in the
		following lemma:
		\begin{lemma}\label{lemma for cutout weinstein domain}
		Suppose $(W,\lambda,\phi)$ is a finite type Weinstein manifold. Let $\psi:W\to \R$ be an exhausting Morse function. Suppose $X_\lambda$ is nondegenerate and gradient-like for $\psi$ outside $\{ \psi\leq 0\}$. Then $\{ \psi\leq 0 \}$ together with $\lambda$ is Liouville homotopic to a Weinstein domain $W_1:=\{ \phi \leq K\}$, for $K$ sufficiently large.
		\end{lemma}
	    \begin{proof}
	    Let $K$ satisfy 
	    \[\{ \psi\leq 0\}\subset W_1 \subset W_2:=\{ \psi\leq C\}
	    \]
	    for some large enough $C$ (conditions will be evident along the line of proof). Notice that $\{ \psi\leq 0\}$ is Liouville homotopic to $W_2$. Fix a smooth function $\rho$ (it can be constructed on the level sets of $\phi$) such that
	    \begin{itemize}
	    	\item $\rho=1$ in $W_1$, $\rho=0$ outside $W_2$.
	    	\item $X_\lambda(\rho)\leq 0$.
	    \end{itemize}
        Let $M:=\max\limits_{p\in W_2\setminus W_1}(\phi-\psi)$.
	    Now consider the function $f=\rho\phi+(1-\rho)(\psi+M)$. We will show that $f$ is Morse and $X_\lambda$ is gradient-like for $f$. We only need to verify $X$ is gradient-like in $W_2\setminus W_1$.
	    We have \[X_\lambda(f)=\rho X_\lambda(\phi)+(1-\rho)X_\lambda(\psi)+(\phi-\psi-M)(X_\lambda(\rho))\geq \rho X_\lambda(\phi)+(1-\rho)X_\lambda(\psi)>0
	    \]
	    So $X_\lambda$ is gradient-like for $f$ and $f$ doesn't have new critical points outside $W_1$. Because $f|_{W_1}=\phi|_{W_1}$, $f$ is Morse.
	    Hence $(_\lambda,f)$  is also a Weinstein structure on
	    $W_2$, and a linear interpolation between $f$ and $\phi$ gives rise to a family of Weinstein structures. In particular, it gives rise to a Liouville homotopy.
	    \end{proof}

	     In fact, we have an explicit form for the topology of $M_0$.  The following quotient map
		\begin{align*}
		\pi:V_\fa(\epsilon)\times \mathbb{C}^*&\rightarrow \mathbb{C}^{n+1}\setminus V_\fa(0)\\
		(z_0,z_1,\cdots,z_n,t)&\mapsto(z_0t^{b_0},z_1t^{b_1},\cdots,z_nt^{b_n})
		\end{align*}
		coincides with the $C(L)-$ action quotient.

		Therefore $\widehat{M_0(\fa)}$ (hence $M_0$) is diffeomorphic $\mathbb{C}^{n+1}\setminus V_\fa(0)$.We have the following proposition about $M_0$:
		\begin{prop}\label{fundamental group of M0}
			Let $M_0(\fa),\, n\geq 3$ be the manifold defined above. Then $\pi_1(M_0)=\mathbb{Z},H_i(M_0)=0,i\geq 2, i\neq n, n+1
			$.
		\end{prop}
		\begin{proof}
			It suffices to prove the results for $\mathbb{C}^{n+1}\setminus V_\fa(0)$. We have a deformation retraction 
		\[r:\mathbb{C}^{n+1}\setminus V_\fa(0)\to S^{2n+1}\setminus \Sigma(\mathbf{a}),
		\]
		and we have the Milnor fibration:
		\begin{align*}
		S^{2n+1}\setminus \Sigma(\mathbf{a})&\longrightarrow S^1\\
		(z_0,z_1,\cdots,z_n)&\mapsto \frac{f(\mathbf{z})}{||f(\mathbf{z})||}
		\end{align*}
		The fibers are homotopic to a bouquet of $n-$spheres, which is simply connected since $n\geq 3$, the long exact sequence gives us $\pi_1(M_0)=\mathbb{Z}$.
		 Meanwhile, $H_*(M_0)=H_*(S^{2n+1}\setminus \Sigma(\mathbf{a}))$, and for $1<i<2n$, by Alexander duality we have
		 \[ \tilde{H}_i(S^{2n+1}\setminus \Sigma(\mathbf{a}))= 
		 \tilde{H}^{2n-i}(\Sigma(\mathbf{a})).
		 \]
		 The conclusion follows Theorem~\ref{high connectedness of brieskorn manifold}.
		\end{proof}

		\begin{prop}\label{contractible after surgery}
			 Let $M_0$ be a manifold with $\pi_1(M_0)=\Z, H_i(M_0)=0,i\geq 2, i\neq n, n+1$.
			 Suppose $\gamma$ is a generator for $\pi_1(M_0)$ and $M_1$ is the result of attaching a 2-handle along $\gamma$. Then $\tilde{H}_i(M_1)=0, i\neq n, n+1$.
		\end{prop}

		\begin{proof}
		The attaching 2-handle kills the generator $[\gamma]$ so $\pi_1(M_1)=0$. Meanwhile, $H_i(M_0)=0,i\geq 2,i\neq n,n+1$ implies $H_k(M_1)=0,k\geq 3, k\neq n, n+1$ since attaching a 2-handle does not change higher homology. Let's denote the 2-handle by $H$. We have the Mayer-Vietoris sequence:
			\[\cdots\to H_2(H)\oplus H_2(M_0)\to H_2(M_1)\to H_1(M_0\cap H)\xrightarrow{i_*} H_1(H)\oplus H_1(M_0)\to H_1(M_1)\to \cdots
			\]
			Here $[\gamma]$ is the generator of both $H_1(M_0\cap H) $ and $H_1(M_0)$, so $i_*$ is isomorphism. Hence we have
			\[ \cdots \to 0\to H_2(M_1) \to \mathbb{Z}\xrightarrow{\cong } \mathbb{Z}\to 0 \to \cdots
			\]
			So $H_2(M_1)=0$. The conclusion follows.
		\end{proof}

		\begin{comment}
		\begin{lemma}[Corollary 2.30 \cite{mclean2007lefschetz}]\label{lemma for contractible}
			Let $M$ be a contractible Stein manifold of finite type. If $n:=\dim_\mathbb{C}M\geq 3$ then $M$ is diffeomorphic to $\mathbb{C}^n$.
		\end{lemma}
		
		\begin{proof}
			The proof is taken out of McLean  \cite{mclean2007lefschetz}. Let $(J,\phi)$ be the Stein structure associated with $M$. We can also assume that $\phi$ is a Morse function. For $R$ large enough, the domain $M_R:=\{\phi<R\}$ is diffeomorphic to the whole of $M$ as $M$ is of finite type. It suffices to show that the boundary of $\bar{M_R}:=\{\phi\leq R\}$ is simply connected, then the result follows from the $h-$cobordism theorem.
			The function $\psi:=R-\phi$ only has critical points of index $\geq n\geq 3$ because the function only has critical points of index $\leq n$. So $\bar{M_R}$ can be reconstructed by attaching handles of index $\geq 3$ with the help of a Morse function $\psi$. This does not change the fundamental group, hence $\partial \bar{M_R}$ is simply connected because $\bar{M_R}$ is. 
		\end{proof}
		\end{comment}
		
		\subsubsection{Handle attachment and trivialization}\label{specific trivialization}
		
		Now we need to fix a trivialization of the canonical bundle $\kappa_{\widehat{M_0}}$ of $(T\widehat{M_0},J)$. Since we have the $C(L)-$ equivariant quotient map $V_\fa(\epsilon)\times\C^* \to \widehat{M_0}$, it suffices to fix $C(L)-$ trivializations on both $V_\fa(\epsilon)$ and $\C^*$ since
		\[T(V_\fa(\epsilon)\times\C^*)=TV_\fa(\epsilon)\times T\mathbb{C}^* 
		\]
		Notice that the trivialization of the symplectic complement in Proposition~\ref{trivialization of the symplectic complement} is $C(L)-$equivariant, and the standard trivialization of $T\C^{n+1}$ is also $C(L)-$equivariant, 
		as long as $\sum_{i=0}^{n}\frac{1}{a_i}\in \Z$.
		Indeed, if we take $\Omega=dz_0\wedge dz_2\wedge\cdots \wedge dz_n$, then the $C(L)-$actions on $\Omega$ is \[\eta^*(\Omega)=e^{\frac{2\pi i}{a_0}}dz_0\wedge\cdots\wedge e^{\frac{2\pi i}{a_n}}dz_n=e^{2\pi i\sum\frac{1}{a_i} }\Omega=\Omega.\]
		 Therefore a $C(L)-$equivariant trivialization of $TV_\fa(\epsilon)$ exists.
		Since $V_\fa(\epsilon)$ is simply connected, the trivialization of $TV_\fa(\epsilon)$ is homotopically unique. We will take the natural trivialization of $T\mathbb{C}^*\to \mathbb{C}^*\times \mathbb{C}$, which determines the trivialization $\Phi$ of $T(V_\fa(\epsilon)\times\C^*)$ . We will also fix $\Phi$ for the rest of this paper, which will be crucial in two places:
		\begin{itemize}
			\item  Determining the framing for the Weinstein 2-handle attachment in Proposition~\ref{handle attachment for construction}.
			\item  Determining the trivialization for the calculation of Conley-Zehnder index in Proposition~\ref{main proposition}.
		\end{itemize}
		
		\begin{prop}\label{handle attachment for construction}
			There is a  contractible Weinstein domain $(M_1,\omega_1,X_1,\psi_1)$ obtained from
			the Weinstein domain $(M_0,-d d^\C\phi,\nabla_\phi\phi,\phi)$ by attaching a 2-handle such that the canonical saturation (see Subsection~\ref{weinstein handlebody}) coincides with the trivialization $\Phi$.
		\end{prop}
		\begin{proof}
			If we can find an isotropic circle in $M_0$ which generates the fundamental group, then by Theorem~\ref{weinstein handle attaching}, we can attach a Weinstein handle in such a way that the trivialization of the contact structure extends to the Weinstein handle body. The existence of such an isotropic circle is guaranteed by the $h-$principle in lemma~\ref{h-principle}, which states a subcritical embedding can be perturbed into an isotropic embedding.
		\end{proof}
	\begin{comment}
	  \begin{remark}
	  	By Proposition~\ref{contractible after surgery}, $M_1$ is contractible.
	  \end{remark}
	  \end{comment}

		Let $M$ be a contact manifold of dimension $2n+1$ and $V$ a smooth manifold of subcritical dimension, i.e. $\dim V \leq n$. Let $Mono^{emb}$ be the space of monomorphisms $TV\to TM$ which cover embeddings $V\to M$, and $Mono^{emb}_{isot}$ its subspace which consists of isotropic monomorphisms $F: TV\to TM$.  Let $Mono_{isot}^{emb}$ be the space of homotopies
		\[Mono^{emb}_{isot}=\{ F_t,t\in [0,1]| F_t \in Mono^{emb}, F_0=df_0, F_1 \in Mono_{isot}^{emb}\}.
		\]
		The space $Emb_{isot}$ of isotropic embeddings $V\to M$ can be viewed as a subspace of $Mono_{isot}^{emb}$. Indeed, we can associate to $f\in Emb_{isot}$ the homotopy $F_t\equiv df, t\in [0,1]$, in $Mono^{emb}_{isot}$.
		\begin{lemma}[Proposition 12.4.1 \cite{eliashberg2002introduction}]\label{h-principle}
			The inclusion
			\[Emb_{isot} \hookrightarrow Mono^{emb}_{isot}
			\]
			is a homotopy equivalence.
		\end{lemma}
		The above $h-$principle also holds in the relative and $\mathcal{C}^0-$dense forms.
		\begin{remark}\label{it is actually stein}
			By Theorem~\ref{from weinstein to stein}, $(M_1,\omega_1,X_1,\psi_1)$ is homotopic to a Stein domain through Weinstein structures. We denote the Stein structure by the same notation $(M_1,J_1,\phi_1)$.
		\end{remark}
	   \begin{comment}
		In light of proposition~\ref{fundamental group of M0}, $\pi_1(M_0)=\mathbb{Z}$ (when $\Sigma(\mathbf{a})$ is a sphere). Suppose $\gamma$ is the generator of  $\pi_1(M_0)$ and let $(M_1,\omega_1,X_1,\psi_1)$ be the Weinstein domain obtained by attaching a Weinstein 2-handle to $M_0$. It is contractible by
		proposition~\ref{contractible after surgery}. As remark~\ref{it is actually stein} shows above, $(M_1,\omega_1,X_1,\phi_1)$ admits a Stein structure. As a matter of fact, since $(M_0,J_0,\phi_0)$ is of finite type, so is $(M_1,J_1,\psi_1)$. Then by lemma~\ref{lemma for contractible}, we have the following:
		\begin{lemma}\label{topology of M_1}
			$M_1$ is diffeomorphic to $B^{2n}$.
		\end{lemma}
		\end{comment}

		\subsection{The Weinstein domain $M_0$}
		 We notice $(V_\fa(\epsilon),-d^\C\phi_0)=(\widehat{X_\epsilon^1},-d^\C\phi_0)$ while $(X_\epsilon^1,-d^\C\phi_0)$ is $C(L)-$equivariant Liouville homotopic to $(W_\epsilon^1,\lambda)$, 
		 and in light of lemma~\ref{lemma for cutout weinstein domain}, we can define different Weinstein domains in $(\widehat{W_\epsilon^1}\times \C^*,\lambda_0:=\lambda+rd\theta)$ by different functions.
		
		 First of all, we need the following technical proposition.

		\begin{prop}\label{functions for the smoothing}
		Let $(M,\lambda)$ be a $G$-equivariant Liouville domain,and $R$ be the coordinate for its cylindrical end. Assume $\phi$ is a $G$-equivariant Morse function on $M$ such that $X(\phi)<0$ near the boundary of $M$. Then for any $\epsilon>0$, there exists $\delta_1\gg \delta_2>0$ and a $G$-equivariant Morse function $f$(see Figure~\ref{Fig:Morse function f}) such that:
		\begin{itemize}
			\item $||1-f||_{\mathcal{C}^2}<\epsilon$ in the region $M\setminus \{R>1-\delta_1+2 \delta_2\}$.
			\item f and $\phi$ have same set of critical points, and the Morse indices are the same.
			
			\item f satisfies the equation 
			\begin{equation}\label{equ: def of f}
			\Big(\frac{f}{a}\Big)^2+\Big(\frac{R-(1-\delta_1)}{\delta_1}\Big)^6=1
			\end{equation}
			on the region $1-\delta_1+\delta_2<R\leq 1$, for some $0<a<1$.
			
		\end{itemize}
		\begin{figure}
			\centering
			\begin{tikzpicture}
		\draw[-latex](0,0)--(0,6)node[left]{$r$};
		\draw  plot[smooth, tension=.47] coordinates {(0,0)(0.3,0)(0.6,-0.2)(0.9,0.15)
			(1.4,-0.05)(1.7,0.08)(1.9,-0.09)(2.2,0)(7,0)};
		\draw[|-|](7,0)node[below]{$1-\delta_1$}--(7.8,0);
		\draw[-|](7.8,0)--(8.1,0);
		\draw [dashed][-latex] (8.1,-0.5)node [below]{$1-\delta_1+2\delta_2$}--(8.1,0);
		\draw[dashed](7.8,1.3)--(7.8,5.3);
		\draw[dashed](8.1,0)--(8.1,4.3);
		\draw [dashed][-latex] (7.8,0.5)node [above]{$1-\delta_1+\delta_2$}--(7.8,0);
		\draw[-|](8.1,0)--(9,0);
		\draw[-latex](9,0)node[below]{$1$}--(10,0);
		\draw node at (1.5,-0.4){interior of $W_\epsilon^1$};
		\draw node at (7.5,-1.35){cylindrical coordinate $R$};
		\draw[dashed](0,5)node[left]{$1$}--(9,5);
		\draw  plot[smooth, tension=.47] coordinates {(0.1,5.1)
			(0.3,4.98)(0.4,5)(0.45,5.04)(0.7,5)(0.8,5.1)(0.9,4.95)
			(1.4,5.05)(1.89,4.93)(2.5,5.08)(3.43,4.90)(4,5.07)(4.1,5)
			(4.15,4.92)(4.7,5)(5.0,4.92)(5.2,5.05)(5.3,5.08)(5.6,5.05)(6,4.95)
			(6.6,4.85)(7,5)};
		\draw[scale=1,domain=0:1.57,smooth,variable=\t] plot ({7+2*sin(\t r)^0.5},{5*cos(\t r)});
		
		\draw[|<-](7.8,5.3)--(8.2,5.3);
		\draw[- >| ](8.6,5.3)--(9,5.3);
		\draw [-latex](7,5.9)node[above]{ $f$ only depends on $R$, $f'(R)<0,(f(R)^2)''<0$}--(8.5,5.3);   
		\draw [-latex] [dashed](7.5,2)node [left]{$f(R)=(1-\epsilon_2)\sqrt{1-\Big(\frac{R-(1-\delta_1)}{\delta_1}\Big)^6}$}--(8.8,2);
		\draw[|<-](0,4.3)--(3.5,4.3);
		\draw[- >| ](4.6,4.3)--(8.1,4.3);
		\draw [-latex] (3.9,3.7)node [below]{$||1-f||_{\mathcal{C}^2}<\epsilon$}--(4.1,4.3);
		\end{tikzpicture}
			\caption{$G-$equivariant Morse function $f$.}\label{Fig:Morse function f}
		\end{figure}
	\end{prop}
	\begin{proof}
		First, we can fix the canonical collar of the boundary using the negative Liouville flow 
		\begin{align*}
		&\iota:(1-\epsilon_1,1]\times \partial M  \longrightarrow M\\
		&\iota^{*}\lambda=R\lambda,\quad \iota^{*} X=R\partial_R
		\end{align*}
		where $\epsilon_1>0$ is sufficiently small, so that $X(\phi)<0$ in the canonical collar and $R$ is the cylindrical coordinate. Notice that $R$ is $G$-equivariant and so is any function in $R$.
		
		Now let's fix a sufficiently small $\epsilon_1>\delta_1\gg \delta_2\gg \epsilon_2>0$ ( the exact constraints on $\delta_1,\delta_2,\epsilon_2$ will be clear along the proof),
		and an increasing bump function $\rho$ such that $\rho(R)=1$ for $R\geq 1$ and $\rho(R)=0$ for $R\leq 0$. 
		Let $\hat{\rho}(R):=\rho(\frac{R-(1-\delta_1)}{\delta_2})$,
		then we have
		\begin{equation}\label{equ:estimation}
		||\hat{\rho}||_\ccc\leq \frac{1}{\delta_2^2}||\rho||_\ccc
		\end{equation}

		Define a bump function $\hat{\rho}$ on $M$  to be $\hat{\rho(R)}$ on its canonical collar and extended by 0. Apparently $\hat{\rho}$ is $G$-equivariant.
		Let $h>0$ be a function of radial coordinate on $[1-\delta_1,1]\times \p M$ satisfying the conditions:
		\begin{equation*}
		\Big(\frac{h}{1-\epsilon_2}\Big)^2+\Big(\frac{R-(1-\delta_1)}{\delta_1}\Big)^6=1.
		\end{equation*}
		
		Then $h$ can be extended to a smooth function on $M$. 
		Without loss of generality, we can assume  $||1-\phi||_\ccc<\epsilon_2$. Otherwise we can simply replace $\phi$ by $1+c\phi$ for $c>0$ sufficiently small.
		We claim  the function 
		\[f=\phi\cdot(1-\hat{\rho})+h\hat{\rho}
		\]
		satisfies all conditions in this proposition.
		Firstly, $h$ is well-defined and $G$-equivariant, and since $f$ coincides with $h$ on the region $\{R\geq 1-\delta_1+\delta_2\}$, equation~\ref{equ: def of f} is satisfied.
		
		Secondly,  we only need to show that $f$ has no critical points in the region $\{1-\delta_1\leq R\leq 1-\delta_1+\delta_2\}$, for which we have
		\[\partial_R(f)=h'\hat{\rho}+h\hat{\rho}'+(1-\hat{\rho})\partial_R(\phi)-\phi\hat{\rho}'
		=(h-\phi)\hat{\rho}+h'\hat{\rho}+(1-\hat{\rho})\partial_R(\phi)<0
		\]
		since $R\partial_R(\phi)=X(\phi)<0$ and $h\leq 1-\epsilon_2\leq \phi$. 
		Therefore $f$ has no critical point in the canonical collar. Since outside the canonical collar $f\equiv\phi$, the second condition follows.
		
		Now we show that $f$ also satisfies the first condition.
		In the region $M\setminus \{R>1-\delta_1\}$, we have
		$f\equiv \phi$, so we only need to check the region $\{ 1-\delta_1\leq R\leq 1-\delta_1+2\delta_2\}$, where
		
		\begin{align*}
		||f-1||_{\mathcal{C}^2}&=||(\phi-1)+(h-\phi)\hat{\rho}||_{\mathcal{C}^2}\\
		&\leq ||\phi-1||_{\mathcal{C}^2}+||(\phi-h)\hat{\rho}||_{\mathcal{C}^2}\\
		&\leq \epsilon_2+2||(\phi-1)+(1-h)||_\ccc\cdot||\hat{\rho}||_\ccc\\
		&\leq\epsilon_2+\frac{1}{\delta_2^2}(||\phi-1||_\ccc+||1-h||_\ccc)||\rho||_\ccc\\
		&\leq \epsilon_2 +\frac{1}{\delta_2^2}(\epsilon_2+||1-h||_\ccc)||\rho||_\ccc 
		\end{align*}
		
		The Taylor expansion of $1-h$ at $R=1-\delta_1$ is:
		\[1- h((1-\delta_1)+t)=\epsilon_2+Ct^6+o(t^{11})\, ,\quad C=\frac{1-\epsilon_2}{2\delta_1^6}
		\]
		Therefore $||1-h||_\ccc\leq \epsilon_2+ C_1 \delta_2^6 $, for $t<2\delta_2\ll \delta_1$, where $C_1=C_1(\delta_1)$. Thus we have 
		\[||1-f||_\ccc\leq \epsilon_2+\frac{2\epsilon_2+C_1\delta_2^4}{\delta_2^2}<\epsilon
		\]
		The last inequality holds as long as $\epsilon_2\leq \delta_2^4$ and $\delta_2\ll \delta_1$.
	\end{proof}
	\begin{lemma}\label{remark: c1 small vector field}
     Let $f,\rho$ be defined as above, $g:=\hat{\rho}(\delta_2-R)$. Then
     $g\cdot X_f$ is $\mathcal{C}^1$ small, where $X_f$ is the Hamiltonian vector field of $f$ with respect to $d\lambda_0$.
	\end{lemma}
	\begin{proof}
	We only need to prove this in $\{ 1-\delta_1+\delta_2<R<1-\delta_1+2\delta_2\}$.
	Notice that \[||X_f||_\cc\leq ||1-f||_\ccc<K\delta_2^2,\]
	where $K$ is independent of $\delta_2$.
	Meanwhile, we have 
	\begin{align*}
	||g\cdot X_f||_{\mathcal{C}^{1}}&
	\leq |g|\cdot||X_f||+||dg||\cdot||X_f||+|g|\cdot||dX_f||\\
	&\leq (|g|+||dg||)\cdot(||X_f||+||dX_f||)\\
	&\leq ||g||_{\mathcal{C}^{1}}\cdot||X_f||_{\mathcal{C}^{1}}\\
	&\leq \frac{||\rho||_\cc}{\delta_2}\cdot K\delta_2^2\\
	&\leq K'\delta_2
	\end{align*}
	
	\end{proof}
		
		Suppose $G$ is a finite group and $M$ is a  $G-$manifold. Let $\mathcal{M}^G(M,\R)$ denote the set of $G-$equivariant Morse functions on $M$ and ${C}(M,\R)$ the set of smooth functions.
		\begin{lemma}[Density Lemma 4.8 \cite{wasserman1969equivariant}]\label{density lemma}
		 $\mathcal{M}^G(M,\R)$ is dense in ${C}(M,\R)$ with respect to the $C^k$ topology.
		\end{lemma} 
		\begin{remark}\label{morse function index at least n }
		Note that $(W_\epsilon^1,\lambda)$ is $G-$equivariantly Liouville isomorphic to $(X_\epsilon^1,-d^\C\phi_0)$. Since $\phi_0$ is $i-$convex on $X_\epsilon^1$ (and we can perturb it into a $G-$equivariant Morse function if necessary), the index of each critical point of $-\phi_0$ is at least $n$ (half of the dimension of a Stein Manifold). Therefore we can find such function $\phi'$ on $W_\epsilon^1$ as well.
		\end{remark}

		Apply proposition~\ref{functions for the smoothing} to $(W_\epsilon^1,\lambda)$, with $\phi'$ as in remark~\ref{morse function index at least n }. Then consider the function $F$ on the product Liouville manifold $(\widehat{W_\epsilon^1}\times \mathbb{R}\times S^1,\lambda_0:=\lambda+rd\theta) $ defined as:
		
		\begin{align}\label{equ:the defining equation}
		F: \widehat{W_\epsilon^1} \times \mathbb{R}\times S^1 &\rightarrow \mathbb{R},\\
		(p,(r,\theta))&\mapsto r^2-f(p)^2  \quad \text{for}\quad p \in W_\epsilon^1\\
		((q,R),(r,\theta)&\mapsto r^2-a^2\Bigg(1-\Big(\frac{R-(1-\delta_1)}{\delta_1}\Big)^6\Bigg) \quad \text{for}\quad (q,R)\in \partial W_\epsilon^1 \times (1-\delta_1+2\delta_2,\infty).
		\end{align}
		where $a= 1-\epsilon_2$.
		It is easy to check that $F$ is a smooth $C(L)$-equivariant function on $\widehat{W_\epsilon^1}\times \mathbb{R}\times S^1$, and $0$ is a regular value. Furthermore, the following lemma shows that the Liouville vector filed $Y:=Y_\lambda+r\p_r$( where $Y_\lambda$ is the Liouville field on $(W_\epsilon^1,\lambda)$) is gradient-like for $F$ on $\{ F\geq 0\}$. 
		
		\begin{lemma}\label{lemma for transversality}
			The Liouville vector field $Y$ of $(\widehat{W_\epsilon^1}\times \mathbb{R}\times S^1, \lambda+rd\theta )$ is gradient-like for $F$ outside
			 $W_0:=\{F\leq 0\}$.
		\end{lemma}
		
		\begin{proof}
			We will verify the statement on the regions $\{R>1-\delta_1+\delta_2 \}$ and $\{ R>1-\delta_1+2\delta_2\}^c$ separately.
			In the region $\{R>1-\delta_1+\delta_2 \}$, $Y=r\p_r+R\p_R$ with $F=r^2-a^2\Bigg(1-\Big(\frac{R-(1-\delta_1)}{\delta_1}\Big)^6\Bigg)$, the claim is trivial.
			In the region $\{ R>1-\delta_1+2\delta_2\}^c$, we have $Y=r\p_r+Y_\lambda$. Notice that this region is a product $W'\times (\R\times S^1)$, where $W'=\widehat{W_\epsilon^1}\setminus\{ R>1-\delta_1+2\delta_2\}$ is $W_\epsilon^1$ attached with a cylindrical cobordism. Now,
			\begin{equation}
			Y(F)=(r\p_r+Y_\lambda)(r^2-f^2)=2(r^2-fY_\lambda(f))
			\end{equation}
			Since $1-f$ is $\mathcal{C}^2$ small, the coordinate $r$ is nonzero in the region $W_0^c\cap \{ R>1-\delta\}^c$, and $W'$ is compact, we have $2(r^2-fY_\lambda(f))>0$.
			The conclusion follows.
		\end{proof}
		
		\begin{remark}\label{def of key block}
			Note that $(\{F\leq 0\},\lambda+rd\theta)$ is a $C(L)$-equivariant Liouville domain, and $C(L)$ acts freely on it. The quotient domain is Liouville homotopic to the Stein domain $(M_0,J,\phi)$, by Lemma~\ref{lemma for cutout weinstein domain}. Since the properties of interest are invariant under Liouville isomorphism, we will also denote the quotient domain$(\{F\leq 0\},\lambda_0)/C(L)$  by $(M_0,\lambda_0)$.
		\end{remark}

		\begin{remark}\label{containment}
			The region $(U:=\{R>1/2\}^c\cap \{|r|\leq 1/2\},\lambda)$ is a Liouville domain with corners. We can smooth out the corner with a $\mathcal{C}^\infty$-small perturbation. By abuse of notation, the boundary of this Liouville domain is denoted by $M=\{R=1/2\}\times\{|r|\leq1/2\}\cup \{R>1/2\}^c\times\{|r|=1/2\}$, with Liouville vector field $Y=R\partial_R+r\partial_r$. The time 1 flow of $Y$ sends $M$ to a new boundary $ \{R=e/2\}\times\{|r|\leq e/2\}\cup \{R>e/2\}^c\times\{|r|=e/2\}$, that is, $U\cup M\times[0,1]=\{R>e/2\}^c\cap \{|r|\leq e/2\}$. It's easy to check that
			$U\subset \{F\leq 0\}\subset U\cup M\times[0,1]$.
		\end{remark}
		\subsection{Strongly ADC property of $M_0$}\label{ADC property after smoothing}
		
		In this subsection, we will prove that the contact boundary of $(M_0,\lambda_0)$(as in remark~\ref{def of key block}) with respect to the trivialization $\Phi$ is strongly asymptotically dynamically convex.
		Let us first state what the framing is. Since $(\widehat{W_\epsilon^1}\times \mathbb{R}\times S^1, \lambda_0 )$ is a product, it suffices to choose the $G$-equivariant trivialization on both components, since it descends naturally to the quotient $(W_0,\lambda_0)$ (see subSection~\ref{specific trivialization}).
		We denote the boundary of $(M_0,\lambda_0)$ by $(\Sigma_0,\lambda_0)$.

		\begin{theorem}\label{main theorem}
			Let $F$ be the function of Lemma~\ref{lemma for transversality}. Then $(\Sigma_0,\lambda_0)$ satisfies the strongly ADC property with respect to a  trivialization $\Phi$, provided $\mathbf{a}$ satisfies the conditions $n\geq 3$ and $m(\mathbf{a})\geq 2$.
		\end{theorem}

		\begin{prop}\label{main proposition}
			For any $K>0$, there exits a $C(L)-$equivariant function $F$ as defined in \ref{equ:the defining equation} on the Liouville domain $(\widehat{W_\epsilon^1}\times \mathbb{R}\times S^1, \lambda_0 )$ with a chosen trivialization $\Phi$ such that
			\begin{itemize}
				\item[(1)] $\Sigma:=\{F=0\}$ is a regular level set and the Liouville vector field $Y$ points outwards along  $\Sigma$.
				\item[(2)] The quotient $\Sigma_0:=\Sigma/G$ has the property  that all elements of $\mathcal{P}^{<K}_\phi(\Sigma_0,\lambda_0)$ have lower SFT index at least $\min \{m(\mathbf{a})-3/2,n-5/2\}$.
			\end{itemize}
		\begin{figure}
		\centering
		\tikzstyle{myedgestyle} = [-latex]
		\begin{tikzpicture}
	\draw[-latex](0,0)--(0,6)node[left]{$r$};
	\draw  plot[smooth, tension=.47] coordinates {(0,0)(0.3,0)(0.6,-0.2)(0.9,0.15)
		(1.4,-0.05)(1.7,0.08)(1.9,-0.09)(2.2,0)(7,0)};
	\draw[|-|](7,0)node[below]{$1-\delta_1$}--(7.8,0);
	\draw[-|](7.8,0)--(8.1,0);
	\draw [dashed][-latex] (8.1,-0.5)node [below]{$1-\delta_1+2\delta_2$}--(8.1,0);
	\draw [dashed][-latex] (7,0.5)node [above]{$1-\delta_1+\delta_2$}--(7.8,0);
	\draw[-|](8.1,0)--(9,0);
	\draw[-latex](9,0)node[below]{$1$}--(10,0);
	\draw node at (1.5,-0.4){interior of $W_\epsilon^1$};
	\draw node at (7.5,-1.35){cylindrical coordinate $R$};
	\draw[dashed](0,5)node[left]{$1$}--(9,5);
	\draw  plot[smooth, tension=.47] coordinates {(0.1,5.1)
		(0.3,4.98)(0.4,5)(0.45,5.04)(0.7,5)(0.8,5.1)(0.9,4.95)
		(1.4,5.05)(1.89,4.93)(2.5,5.08)(3.43,4.90)(4,5.07)(4.1,5)
		(4.15,4.92)(4.7,5)(5.0,4.92)(5.2,5.05)(5.3,5.08)(5.6,5.05)(6,4.95)
		(6.6,4.85)(7,5)};
	\draw[scale=1,domain=0:1.57,smooth,variable=\t] plot ({7+2*sin(\t r)^0.5},{5*cos(\t r)});
	\draw[dashed](7.8,0.3)--(7.8,5.3);
	\draw[dashed](8.1,0)--(8.1,4.3);
	\draw[|<-](7.8,5.3)--(8.2,5.3);
	\draw[- >| ](8.6,5.3)--(9,5.3);
	\draw [-latex][dashed](10.9,4.2)node[below]{ (a) Fractional Reeb orbits $\widetilde{\gamma}$ }--(8.4,5.3);
	\draw node at(11,3.4){ with $lSFT\geq m(\mathbf{a})-3/2$};
	\draw node at(11,2){ (b) $\gamma$ is contractible};	
	\draw [-latex][dashed] (11.2,1.2)node [above]{ with $lSFT\geq m(\mathbf{a})-3/2$}--(9,0);  
	\draw [-latex] [dashed](3.9,3.5)node [below]{(c) Fractional Reeb orbits $\widetilde{\gamma}$ correspond to }--(3.9,4.3);
	\draw node at (3.9,2.6){critical  points of $f$, $lSFT\geq n-5/2$};
	\draw[|<-](0,4.3)--(3.5,4.3);
	\draw[- >| ](4.6,4.3)--(8.1,4.3);
	
	\draw node at (4.5,1.5) {$W_0$};
	\draw node at (3.5,5.3) {$\Sigma=F^{-1}(0)$};
	\end{tikzpicture}
		\caption{Lower SFT index of Fractional Reeb orbits in $\Sigma$.}\label{Fig:index of Reeb orbits}
     	\end{figure}
		\end{prop}

		\begin{proof}
			We will show that by choosing a proper $\mathcal{C}^2$-small function $f$ as in Proposition~\ref{functions for the smoothing}, the corresponding function $F$ satisfies the required conditions. 
			The first condition is satisfied by the construction of $F$, as proved in Lemma~\ref{lemma for transversality}, we only need to show the second condition is also satisfied. Recall the quotient map 
			\[ \pi: \Sigma \to \Sigma_0
			\]
			is an $L-$sheeted covering map. Therefore, the Reeb orbits in $\Sigma_0$ lift to fractional Reeb orbits in $\Sigma$. To be precise, if $\gamma(t), t\in [0,T]$ is a Reeb orbit in $\Sigma_0$, then the $L-$ fold Reeb orbit $\gamma(t), t\in [0,qT]$ can be lifted to a Reeb orbit $\widetilde{\gamma(t)}, t \in [0,LT]$ in $\Sigma$. It follows that the index of $\gamma$ in $\Sigma_0$ can be calculated through the index of $\widetilde{\gamma}$ in $\Sigma_0$. 
			We will proceed by investigating the Reeb orbits in three regions: 
			\begin{itemize}
				\item[(a)] $\Sigma\cap \{1>R>1-\delta_1+\delta_2\}$, where $\widetilde{\gamma}$ has constant $r,R$ coordinates.
				\item[(b)] $\Sigma\cap \{1=R\}$, where $\gamma$ is contractible, and $\gamma$ lifts to closed Reeb orbit $\widetilde{\gamma}$ in $\Sigma$.
				\item[(c)] $\Sigma\cap \{R>1-\delta_1+2\delta_2\}^c$, where $\widetilde{\gamma}$ has constant coordinate in the $W_\epsilon^1$ component.
			\end{itemize}
			
			We will show that all elements of $\mathcal{P}^{<K}_\phi(\Sigma_0,\lambda_0)$(see Figure~\ref{Fig:index of Reeb orbits}) can be lifted to fractional Reeb orbits either entirely contained in part (a), (b) or (c) and 
			\begin{itemize}
				\item[(a)] orbits in part(a) have lower SFT index at least $m(\mathbf{a})-3/2$ ;
				\item[(b)] orbits in part (b) have lower SFT index at least $m(\mathbf{a})-3/2$ ;
				\item[(c)] orbits in part (c) have lower SFT index at least $n-5/2$.
			\end{itemize}
			First, in region (a), by lemma~\ref{formula for reeb vector}, we have 
		\[X_{Reeb}=\frac{X_F}{Y(F)}=\frac{2r\partial_\theta-2fX_f}{2r^2-2fY_\lambda(f)}=\frac{2r\p_\theta-2ff'J\p_R}{2r^2-2Rff'}.
		\]
		
		So $X_{Reeb}$ has no $\p_R$ component, and therefore the Reeb flow in the region (a) has constant $R$ coordinate. So any Reeb orbits $\gamma$ intersecting $ \{R>1-\delta_1+\delta_2\} $ remains entirely in region (a).
			Let us begin the proof with a lemma:
		\begin{lemma}\label{lemma: Hamiltonian and Reeb case relation verified}
			With $W_0$ and $F$ defined as in Lemma~\ref{lemma for transversality}, the conditions in Lemma~\ref{Reeb--Hamiltonian index relation} are satisfied.
		\end{lemma}
		\begin{proof}
			
			We will  verify the conditions in three cases:
			\begin{itemize}
				\item[a.] in the region $W_0\setminus \{ R>1-\delta_1+2\delta_2\}$, where $||1-f||_\ccc<\epsilon$;
				\item[b.] in the region $\{ 1-\delta_1+\delta_2<R< 1\}$, where $Y=r\p_r+R\p_R$, and $f=f(R)$.
				\item [c.] in the region $R=1, r=0$.
			\end{itemize}
			First of all, note that $b=dF(Y)=2r^2-2fY_\lambda(f)>0$ by Lemma~\ref{lemma for transversality}.
			
			Case (a): Let $p$ be a critical point of $f$ 
			and define \[A:=\{(p,\sqrt{f(p)^2+t},\theta)\in W_\epsilon^1\times \mathbb{R}\times S^1\,\big|\, t\in (-\epsilon_0,\epsilon_0)\}.
			\]
			Define $C_t:=F^{-1}(t)$, which is transverse to $A$ as $\partial_r$ is transverse to it. Let
			\[A_t:=C_t\cap A=\{(p,\sqrt{f(p)^2+t},\theta)\in W_\epsilon^1\times \mathbb{R}\times S^1\}
			\]
		 and $L_t=\sqrt{f(p)^2+t}$, $b/L_0=2f(p)$. We can rescale $L_t$ by $2f(p)$,
			and with $V=\frac{\p_r}{2r}$,
			we have
			\[db(V)=2>2f(p)\frac{dL_t}{dt}\Big|_{t=0}=1
			\]
			
			Case(b): Let $B$ be a Morse-Bott manifold of the Brieskorn manifold $(\Sigma(\fa),\lambda)$,
			and $g(R)=-f^2(R)$. Then $F(t)=r(t)^2+g(R(t)), t\in (-\epsilon_0,\epsilon_0)$
			. We have the following:
			\[ dF=2rdr+g'(R)dR,\quad b=dF(Y)=2r^2+Rg'(R),\]
			\[X_{Reeb}=X_F/Y(F)=(2r\p_\theta+g'(R)J\p_R)/b\]
		    
			Now, define for any constant $a>0$ ($-1/a$ is the slope of tangent line of $F$ at $(r,R)$),
			\[A(a):=\{(q,R(t),\theta,r(t))\in B\times(1-\delta_1+\delta_2,1)\times S^1\times\R|r^2-f(R)^2=t, r=ag'(R), t\in (-\epsilon_0,\epsilon_0)\}
			\]
			Again let $C_t:=F^{-1}(t)$, which is transverse to $A(a)$, and 
			\[A_t:=C_t\cap A=\{(q,R(t),\theta,r(t))\, \}
			\]
			Then $A(t)$ is a Morse-Bott manifold in $C_t$.
			Since
			\begin{equation}\label{equ:period}
			L_t=b/2r(t)=r+\frac{Rg'(R)}{2r}=r+\frac{R}{2a},\quad b/L_0=2r(0)
			\end{equation}
			\begin{equation}\label{equ: derivative of period}
			2r\frac{dL_t}{dt}\Big|_{t=0}=2rr'+\frac{rR'}{a}
			\end{equation}
		  and on the other hand,  $V=r'\p_r+R'\p_R$,
		  $db=4rdr+(g'(R)+Rg''(R))dR$,
			we have that
			\[db(V)= 4rr'+R'g'(R)+RR'g''(R)\geq 2rr'+\frac{rR'}{a}=2r\frac{dL_t}{dt}\Big|_{t=0}
			\]
			since $r'>0, R'>0, g''(R)>0$.
			
			Case (c): Let $g=-f^2(R), B$ defined as above, define 	\[A:=\{(q,R(t),\theta,0)\in B\times\R\times S^1\times\R\,|\,g(R)=t, t\in (-\epsilon_0,\epsilon_0)\}
			\]
			Once more, $C_t:=F^{-1}(t)$, which is transverse to $A$ as $\partial_R$ is transverse to it, and 
			\[A_t:=C_t\cap A=\{(q,R(t),\theta,0)\in B\times\R\times S^1\times\R\,\}
			\]
			are pseudo Morse-Bott manifolds. Moreover,
			\[dF=g'(R)dR, \, b=dF(Y)=Rg'(R),\,
			L_t=R(t),\, b/L_0=g'(1)\] Here,
			\[ g'(1)\frac{dL_t}{dt}\Big|_{t=0}=g'(R)R'|_{R=1}=\frac{d}{dt}(g(R(t)))|_{t=0}=1
			\] 
			and with $V=\frac{\p_R}{g'(R)}$,
			we have
			\[db(V)=1+\frac{Rg''(R)}{g'(R)}>1=g'(1)\frac{dL_t}{dt}\Big|_{t=0}.
			\]
		\end{proof}

			Now let us compute the index of the Reeb orbits in region (a).
			Any Reeb orbit $\gamma$ can be lifted to a fractional Reeb orbit $\widetilde{\gamma}$ in the region  $\Sigma\cap \{R>1-\delta_1+\delta_2\}$, where $F=r^2-f(R)^2$. The Reeb orbit can be written as $\widetilde{\gamma}=(\gamma_1,\gamma_2)$, where $\gamma_1,\gamma_2$ are fractional Reeb orbits of $(\Sigma(\mathbf{a}),R\lambda)$ and $(S^1,r\theta)$, for fixed $r,R$,
			so by Lemma~\ref{lemma: Hamiltonian and Reeb case relation verified}, $$\mu_{CZ}(\gamma,F)=\mu_{CZ}(\gamma,\lambda_0)+\frac{1}{2}.$$
			Meanwhile, 
			\[ \mu_{CZ}(\gamma,F)= \mu_{CZ}(\gamma_1,-f(R)^2)+\mu_{CZ}(\gamma_2,r^2)
			\]
			follows the product property of Conley-Zehnder index. Note that
			 $$(-f(R)^2)'>0, (-f(R)^2)''>0,$$
			 therefore we have
			 \[\mu_{CZ}(\gamma_1,-f(R)^2)=\mu_{CZ}(\gamma_1,c_1\lambda)+\frac{1}{2}.
			 \]
			Moreover, by Remark~\ref{remark: index of orbits on cylinder},  \[
			\mu_{CZ}(\gamma_2,r^2)=\frac{1}{2}.
			\]
		 Notice $\gamma_1$ is a fractional Reeb orbits on the Brieskorn manifold $(\Sigma(\mathbf{a}),R\lambda)$, which has the same index as 
			$(\Sigma(\mathbf{a}),\lambda)$. By Lemma~\ref{index calculation} and Lemma~\ref{minimal index}, we then have $\mu_{CZ}(\gamma_1)\geq m(\mathbf{a}).$ 
			Putting all equations together:
			 \begin{align*}
			 lSFT(\gamma)&=\mu_{CZ}(\gamma,\lambda_0)-\frac{1}{2}\dim B+(n+1)-3\\
			 &\geq \big(\mu_{CZ}(\gamma,F)-\frac{1}{2}\big)-n+(n+1)-3\\
			 &\geq \big(\mu_{CZ}(\gamma_1,\lambda)+\frac{1}{2}\big)+\mu_{CZ}(\gamma_2,r^2)-\frac{5}{2}\\
			 &\geq m(\fa)+\frac{1}{2}+\frac{1}{2}-\frac{5}{2}\\
			&=m(\fa)-3/2.
			 \end{align*}

			For the region (b), the claim will be proved in Lemma~\ref{lemma for contractible orbits}.
			
			Now suppose $\gamma_0(t),t\in [0,T],T<K$ is a Reeb orbit in $\Sigma_0$ and can be lifted to a fractional Reeb orbit in region (c). Then the $L-$ fold Reeb orbit $\gamma(t):= \gamma_0(t), t\in [0,LT]$ can be lifted to a closed Reeb orbit $\widetilde{\gamma(t)}$ in this region. Let $g(R)$ be a smooth function defined in Lemma~\ref{remark: c1 small vector field}. So $g(R)=1$ for $R<1-\delta_1+\delta_2$ and $g(R)=0$ for $R>1-\delta_1+2\delta_2$. By abuse of notation, $g$ can be regarded as a function on $\Sigma\cap \{1-\delta_1+\delta_2<R<1-\delta_1+2\delta_2\}$. We extend $g$ to $\Sigma$ by a constant. Now define a new vector field $X=g\cdot X_{Reeb}$. Let $X_W$ be the projection of $X$ to $W_\epsilon^1$, i.e. \[X_W=g\cdot \frac{-X_f}{f-Y_\lambda(f)}=\frac{-1}{f-Y_\lambda(f)}\cdot gX_f.\] Since $(1-f)$ is $\mathcal{C}^2$-small, $||\frac{1}{f-Y_\lambda(f)}||_\cc<2$.
			By Lemma~\ref{remark: c1 small vector field}, $X_W$ is $\mathcal{C}^1$-small. Then by Corollary~\ref{small norm vector field}, for $f$ sufficiently $\mathcal{C}^2$-small, any periodic orbit of period less than $LK$ is a constant orbit, and therefore corresponds to a critical point of $f$. We claim that any such Reeb orbit $\gamma_0$ has lower SFT index at least $n-5/2$, which will be proved in Proposition~\ref{proof of index of critical point}. 
		\end{proof}

		\begin{remark}\label{noncontactible remark}
			Reeb orbits in $W_0$ can be graded by their $H_1/Tors$ class. Let's have a closer look at the Reeb orbits with $H_1/Tors$ grading $0$. In the proof of Proposition~\ref{main proposition}, the Reeb orbits in the regions (a) and (c) are never null-homologous.
		\end{remark}
		\begin{lemma}\label{lemma for contractible orbits}
			Any Reeb orbit $\gamma$ in region (b) is contractible in $\Sigma_0$, and its lower SFT index is at least $m(\mathbf{a})-3/2$.
		\end{lemma}
		\begin{proof}
			In region (b), $X_{Reeb}=J\partial_R$. In fact, $\Sigma\cap \{1=R\}=\Sigma(\mathbf{a})\times S^1$. The Reeb flow is stationary on $S^1$ and coincides with the Reeb flow on the Brieskorn manifold $\Sigma(\mathbf{a})$. Therefore, any Reeb orbit $\gamma$ is contractible. Suppose $\widetilde{\gamma}$ is a lift of $\gamma$. 
			Let $B\subset \Sigma(\fa)$ be a Morse-Bott manifold for $(\Sigma(\fa),\lambda)$.
			In light of Lemma~\ref{lemma: Hamiltonian and Reeb case relation verified}, we have
			\[\mu_{CZ}(B\times S^1, F)=\mu_{CZ}(B\times S^1, \lambda_0)+\frac{1}{2}.
			\]
			By the product property of Conley-Zehnder index, 
			\[\mu_{CZ}(B\times S^1, F)=\mu_{CZ}(B,-f^2(R))+\mu_{CZ}(S^1,r^2)=\mu_{CZ}(B,-f^2(R))+\frac{1}{2}.
			\]
			\begin{comment}
		    Since	$X_r|_{r=0}=0$, we have $\psi_t|_{S^1\times\R} =\textrm{id}$. So $\mu_{CZ}(S^1,r^2)=0$, and 
		    \[D\psi_t=D\psi_t|_{\Sigma\times\R}\oplus D\psi_t|_{S^1\times\R}=D\psi_t|_{\Sigma\times\R}\oplus\text{id}|_{S^1\times\R},
		    \]
		    so \[\dim\ker (D\psi_T-\text{id})|_{B\times S^1}=\dim B +2\,,\] where $T$ is the period of Reeb orbit. Thence
		    $D\psi_T-\text{id}$ has constant rank along $B\times S^1$. For Reeb orbits in a neighborhood of $B\times S^1$, the period $L= r+\frac{Rg'(R)}{2r}$, by equation~\ref{equ:period}. So for $r\neq 0$, the period of Reeb orbit near $B\times S^1$ is very large, so $B\times S^1$ is an isolated family of Reeb orbits hence is pseudo Morse-Bott.
		    \end{comment}
		    On the other hand,
		    \[ \mu_{CZ}(B,-f^2(R))=\mu_{CZ}(B,\lambda)+\frac{1}{2}\geq m(\fa)+\frac{1}{2}.\]
		    So we conclude that 
		    \[ \mu_{CZ}(B\times S^1, \lambda_0)= \mu_{CZ}(B\times S^1,F)-\frac{1}{2}\geq m(\fa)+\frac{1}{2}
		    \]
		    and
		    \begin{align*}
		   lSFT(\gamma)&=\mu_{CZ}(\gamma)-\frac{1}{2}\dim\ker(D_{\gamma(0)}\psi_T-\textrm{id})+(n+1-3)\\
		     &=\mu_{CZ}(B\times S^1, \lambda_0)-\frac{1}{2}(\dim B+1)+(n+1-3)\\
		   &\geq m(\fa)+1/2-n+n-2=m(\fa)-3/2.
		    \end{align*}
		\end{proof}
		\begin{remark}\label{contractible generator of spectral sequence}
			
			Let $MB(p), p\in \Z $ be the Morse-Bott manifold of return time $\frac{p\pi}{2}$ in the Brieskorn manifold, then $MB(p)\times S^1/C(L)$ is a Morse-Bott manifold in $\Sigma_0$. Conversely, any Morse-Bott manifold of contractible Reeb orbits in $\Sigma_0$ can be lifted to $\Sigma$. By Lemma~\ref{lemma for contractible orbits} and Remark~\ref{noncontactible remark}, the contractible Morse-Bott manifolds  in $\Sigma_0$ can be lifted to $\Sigma(\fa)\times S^1 \subset \Sigma$. Indeed, each Reeb orbit in $\Sigma_0$ has $L$ different lifts in $\Sigma$. 
			In terms of Morse-Bott manifolds of contractible Reeb orbits, we have
			 a one-to-one correspondence:
			\begin{align*}
			\pi: \Sigma(\fa)\times S^1&\rightarrow (\Sigma(\fa)\times S^1)/C(L)\subset \Sigma_0\\
			MB(p)\times S^1 &\mapsto (MB(P)\times S^1)/C(L).
			\end{align*}
			 The group action is trivial on the first factor, therefore \[(MB(p)\times S^1)/C(L)= MB(p)\times (S^1/C(L)) \cong MB(p)\times S^1\] 
			 and \[ \mu_{CZ}(MB(p)\times S^1,\lambda_0)=\mu_{CZ}(MB(p),\lambda)+\frac{1}{2}.
			 \]
			
		\end{remark}

		\begin{prop}\label{proof of index of critical point}
			As defined in the proof of part (c) of Proposition~\ref{main proposition}, the Reeb orbit $\gamma_0$ has lower SFT index at least $n-5/2$.
		\end{prop}
		\begin{proof}
			The $L-$fold iterate $\gamma(t)$ can be lifted to a Reeb orbit $\widetilde{\gamma(t)}$ in the Region (c). Its $W_\epsilon^1$ component is a critical point $p$ of $f$. Since the conditions of Lemma~\ref{Reeb--Hamiltonian index relation} are satisfied,
			\[\mu_{CZ}(B_0,\lambda_0)+\frac{1}{2}=\mu_{CZ}(B_0,F).
			\]
			Everything descends down to the quotient $M_0$. We will use the same notations for the quotient.
			We have the Hamiltonian orbit $\gamma_0=(p,\gamma_2)$, where $p$ is a constant orbit in $W_\epsilon^1$ while $\gamma_2$ is an orbit in $\mathbb{R}\times S^1$. The index is
			\[\mu_{CZ}(B_0,F)=\mu_{CZ}(p,-f^2)+\mu_{CZ}(\gamma_2,r^2)=\mu_{CZ}(p,-f^2)+\frac{1}{2}
			\]

			Since $f(p)\neq 0$, $Ind_p(f^2)=Ind_p(f)$, hence \[\mu_{CZ}(p,-f^2)=Ind_p(f^2)-n=Ind_p(f)-n\] by Corollary~\ref{index of critical pt}.
		 Since indices of critical points of $f$ is at least n,	so $\mu_{CZ}(p)\geq 0$ (see remark~\ref{morse function index at least n }). Thus lower SFT index
			\begin{align*}
			lSFT(\gamma)&=\mu_{CZ}(B_0,\lambda_0)-\frac{1}{2}\dim B_0 +(n+1) -3\\
			&=\mu_{CZ}(B_0,F)-\frac{1}{2}-\frac{1}{2}+(n+1)-3\\
			&\geq 0+\frac{1}{2}+n-3= n-5/2
			\end{align*}
			
			where the Morse-Bott manifold $B_0=S^1$.
		\end{proof}
		
		\begin{proof}[Proof of Theorem~\ref{main theorem}]
			Recall the definition of a strongly ADC contact manifold: there exists a sequence of non-increasing contact forms $\alpha_i$ and increasing positive numbers $D_i$ going to infinity such that all elements of $\mathcal{P}_\Phi^{<D_i}(\Sigma,\alpha_i)$ have positive lower SFT index.

			In light of Proposition~\ref{main proposition}, let $ K_i=K^i$($K$ is a fixed large number, the explicit conditions will be clear later in this proof),
			there exists a $C(L)-$ equivariant function $F_i$ such that  
			all elements of $\mathcal{P}_\Phi^{<K_i}(\Sigma_i,\lambda_0|_{\Sigma_i})$ have positive lower SFT index (since $\min\{m(\mathbf{a})-3/2,n-5/2\}>0$), where $\Sigma_i:=F_i^{-1}(0)/C(L)$ is the boundary of the quotient manifold. 
			
			By Remark~\ref{containment}, we notice that conditions of Corollary~\ref{corollary for non-increasing form } are satisfied, so there exists a contactomorphism $f_i:\Sigma_0\to \Sigma_{i+1}$ and a constant $C$ independent of $F_i$, such that
			\[\frac{1}{C}\cdot\lambda_0|_{\Sigma_0}<f_i^{*}(\lambda_0|_{\Sigma_{i}})<C\cdot\lambda_0|_{\Sigma}.
			\]
			So the non-increasing contact forms $\alpha_i$ can be defined as  $\alpha_{i}=\frac{1}{C^i}f_i^*(\lambda_0|_{\Sigma_{i}})<\alpha_{i-1}$, and $D_i:=K_i/C^i$, which goes to infinity as long as $K>C$. Then $\mathcal{P}_\Phi^{<D_i}(\Sigma_0,\alpha_i)=\mathcal{P}_\Phi^{<L_i}(\Sigma_i,\lambda_0|_{\Sigma_i})$, which shows that all elements have positive lower SFT index.
		\end{proof}

		We follow the idea of F.laudenbach in the proof of the following lemma.
		\begin{lemma}[Proposition 6.1.5 \cite{audin2014morse}, \cite{laudenbach2004symplectic}]
			Let $X$ be a vector field on $\mathbb{R}^{2n}$. If $||dX||_{L^2} < \frac{2 \pi}{L}$, the only periodic orbits with period less than L are constant orbits. 
		\end{lemma}
		\begin{proof}
			Consider the solution $u(t)$ of period $T\leq L$ and take its Fourier expansion as well as $\dot{u},\ddot{u}$.
			\[u(t)=\sum_{k}c_k(u)e^{2k\pi i t/T},\quad \dot{u}(t)=\sum_{k}\frac{2k\pi i}{T}c_k(u)e^{2k\pi i/T}
			\]
			So by Parseval's identity, we have 
			\[||\ddot{u}||^2_{L^2}=\sum \frac{4k^2\pi ^2}{T^2}|c_k(\dot{u})|^2\geq \sum_{k\ne 0}\frac{4\pi^2}{T^2}|c_k(\dot{u})|^2=\frac{4\pi^2}{T^2}||\dot{u}(t)||^2_{L^2}
			\]
			since $c_0(\dot{u})=0$. Hence,
			\[||\ddot{u}||_{L^2}\geq \frac{2\pi}{T}||\dot{u}(t)||_{L^2}.
			\]
			On the other hand, since  $\ddot{u}=(dX)(\dot{u})$,  $||dX||_{L^2} < \frac{2 \pi}{L}$, so
			\[||\ddot{u}||_{L^2}< \frac{2\pi}{L}||\dot{u}(t)||_{L^2} \leq \frac{2\pi}{T}||\dot{u}(t)||_{L^2}\]
			if $\dot{u}\ne 0$.
			Therefore $u(t)$ is a constant orbit.
		\end{proof}
		\begin{corollary}[ \cite{laudenbach2004symplectic}]\label{small norm vector field}
			If $M$ is a compact manifold with boundary and $X$ is a  vector field which vanishes in the neighborhood of the boundary. Then for any $L>0$, the flow generated by $X$ has no non-constant periodic orbit with period less than $L$ for sufficiently $\mathcal{C}^1$-small $X$.  
		\end{corollary}
		\begin{proof}
			First we get rid of the boundary by doubling $M$ (glue $M$ with itself along the boundary). Now that $X$ can be smoothly extended since it vanishes in a neighborhood of the boundary. Now consider the new closed manifold $\tilde{M}$. Let us fix a finite collection of compact charts $K_i$. Since $X$ is $\mathcal{C}^1$-small, every closed orbit with bounded period $T$ of the flow of $X$ has a small diameter($D\leq ||X||_{uniform}\cdot L)$, which implies the entire orbit remains in one of the charts $K_i$. The $\mathcal{C}^1$ norm is equivalent to the Euclidean norm so the lemma above applies.
		\end{proof}
		\begin{comment}
		\begin{lemma}\label{lemma for c-1 small}
			Let $f$ be a compactly supported smooth function and $X$ a smooth vector field. $f\cdot X$ is $\mathcal{C}^{1}-$small if $X$ is.
		\end{lemma}
		\begin{proof}

		\end{proof}
		\end{comment}
		
		\begin{lemma}\label{Liouville embedding}
			Let $(U,\lambda)$ be a Liouville domain,
			$(\widehat{U},\hat{\lambda})$ its completion, and $\Sigma_1:=\partial U$ be the contact boundary.
			Suppose we have a Liouville domain $(V,\hat{\lambda})$ such that $ U\subset V\subset U\cup\Sigma_1\times [0,M] $.Then there is a contactomorphism $\Psi$ 
			\[\Psi:(\Sigma_1,\lambda_2=\hat{\lambda}|_{\Sigma_1})\to 
			(\Sigma_2:=\partial V,\lambda_2=\hat{\lambda}|_ {\Sigma_2}).
			\]
			such that  $ \lambda_1\leq \Psi^{*}\lambda_2\leq e^M\lambda_1$.
		\end{lemma}
		\begin{proof}
			Since  \[U\subset V \subset \Sigma_1\times [0,M]  \]
			let $\psi$ be the flow generated by the Liouville vector field and 
			$t(p)$ be the time when the flow starting at $p \in \Sigma_1$ reaches $\Sigma_2$, i.e, $\psi_{t(p)}(p)\in \Sigma_2$. Then $M\geq t(p)\geq 0$.
			Let $\rho$ be a function on $\widehat{U}$  supported on $(-\epsilon,M+1)\times\Sigma_1$, such that $\rho((r,p))\equiv t(p)$ on the region $[-0,M]\times \Sigma_1$.
			Now consider the vector field $Y:=\rho\cdot\partial_r$ and we denote by
			\[\Psi:\mathbb{R}\times \widehat{U} \to \widehat{U},\quad (t,p)\mapsto \Psi_t(p)
			\]
			the flow generated by $Y$. Clearly we have $\Psi_1(\Sigma_1)=\Sigma_2$ and $\Psi^{*}\lambda_1=e^{t(p)}\lambda_2$. Now the  conclusion follows.
		\end{proof}
		\begin{remark}
			If the Liouville domains in the Lemma above are $G$-equivariant, then there is  $G$-equivariant contactomorphism satisfying the above statement.
		\end{remark}
		
		\begin{corollary}\label{corollary for non-increasing form }
			Let $(U,\lambda)$ be a Liouville domain, suppose we have two Liouville domains $V_1,V_2$ with $\Sigma_1=\partial V_1,\Sigma_2=\partial V_2$ such that
			$U\subset V_i\subset U\cup \partial U\times [0,M]$. Then there exists a contactomorphism $f$ and a constant $C$ independent of $V_i$, such that
			\[\frac{1}{C}\cdot\lambda|_{\Sigma_1}<f^{*}\lambda|_{\Sigma_2}<C\cdot\lambda|_{\Sigma_1}.
			\]
		\end{corollary}

		\section{Finiteness of positive idempotent group  }\label{Finite}
		We are going to show that the positive idempotent group $I_+(\Sigma_0)$ is finite. Let's recall the definitions: for any filling $W$ of $\Sigma_0$ such that $SH_*(W)\neq 0$, we have
	
		\[I(W)=\{\, \alpha \in SH_n^0(W)\,\big| \, \alpha^2-\alpha \in H^0(W)\,\}
		\]
		and $I_+(W)=I(W)/H^0(W)$, hence it suffices to prove $I(W)$ is a finite group. Indeed, for the Liouville filling $(M_0,\lambda_0)$ as in remark~\ref{def of key block}, $SH_k^0(M_0,\Z_2)$ is finite. We begin by introducing a spectral sequence which converges to $SH_*^0(M_0,\Z_2)$:
	
		\begin{theorem}[Theorem 5.4  \cite{kwon2016brieskorn}]\label{spectral sequence} Let$(W,\omega=d\lambda)$ be a Liouville domain satisfying the assumptions:
			\begin{itemize}
				\item[1] The Reeb flow on $\partial W$ is periodic with minimal periods $T_1\cdot \frac{\pi}{2},T_2\cdot \frac{\pi}{2},\cdots,T_k\cdot \frac{\pi}{2}$, where $T_k\cdot \frac{\pi}{2}$ is the common period, i.e. the period of a principal orbit. We assume that all $T_k$ are integers.
				\item[2] The restriction of the tangent bundle to the symplectization of $\partial W$,$T(\mathbb{R}\times \partial W)|_{\partial W}$, is trivial as a symplectic vector bundle, $c_1(W)=0$ and we have a choice of the trivialization of the canonical bundle.
				\item[3] There is a compatible complex structure $J$ for $(\xi:=\ker \lambda_{\partial W},d\lambda_{\partial W})$ such that for every periodic Reeb orbit $\gamma$ the linearized Reeb flow is complex linear with respect to some unitary trivialization of $(\xi,J,d\alpha)$ along $\gamma$.
			\end{itemize}
			For each positive integer $p$ define $C(p)$ to be the set of Morse-Bott manifolds with return time $p$, and for each Morse-Bott manifold $\Sigma\in C(p)$ put
			\[ \Delta(\Sigma)=\mu_{CZ}(\Sigma)-\frac{1}{2}\dim \Sigma/S^1,
			\]
			where the Robbin-Salamon index is computed for a symplectic path defined on $[0,p]$. Then there is a spectral sequence converging to $SH(W;R)$, whose $E^1-$page is given by
			\[E_{pq}^1=\begin{cases}
			\bigoplus\limits_{\Sigma\in C(p)}
			 H_{p+q-\Delta(\Sigma)}  (\Sigma;R)  & p>0\\
			H_{q+n}(W,\partial W;R)   & p=0\\
			0 & p<0.
			\end{cases}
			\]
		\end{theorem}

	\begin{remark}
	The above spectral sequence respects the $H_1$ grading. Therefore, to compute $SH_n^0(M_0)$, we only need to focus on the Morse-Bott manifolds of null-homologous Reeb orbits.
	\end{remark}
		
		 \begin{lemma}\label{finiteness of group}
			$SH_k^0(M_0,\Z_2)$ is finite for all $k$.
		\end{lemma}

		\begin{proof}[proof of lemma~\ref{finiteness of group}]
			Note that it suffices to find all the Morse-Bott manifolds.
			By Remark~\ref{contractible generator of spectral sequence}, the first page of the spectral sequence which converges to  $SH_*^0(M_0,\mathbb{Z}_2)$  is
			\[E_{pq}^1=\begin{cases}
			\bigoplus H_{p+q-\Delta(MB(p))}  (MB(p)\times S^1 ;\mathbb{Z}_2)  & p>0\\
			H_{q+n}(M_0,\partial M_0;\mathbb{Z}_2)   & p=0\\
			0 & p<0.
			\end{cases}
			\]
			The finiteness of $SH_k^0(M_0,\mathbb{Z}_2)$ follows from the following two facts: 
			first, there are only finitely many Morse-Bott manifolds $MB(p)$ satisfying $\Delta(MB(p))=k$, i.e.
			\[k=\mu_{CZ}(MB(p)\times S^1)-\frac{1}{2}(\dim(MB(p)\times S^1)/S^1)=f_\mathbf{a}(p)-\frac{1}{2}(\dim MB(p)-1).
			\]
			The above equation can only be satisfied by
			finitely many $p\in \frac{1}{2L}\Z$, and
			for any $p$ there is at most one Morse-Bott manifold with return time $p\pi/2$ in the Brieskorn manifold $\Sigma(\mathbf{a})$.
			Secondly, \[H_* (MB(p)\times S^1 ;\mathbb{Z}_2) =0\quad *<0\, \text{or}\, *>2n .\] and $H_* (MB(p)\times S^1 ;\mathbb{Z}_2)$ is finite dimensional for $0\leq * \leq 2n$. Therefore $SH_k^0(M_0,\Z_2)$ is finite for each $k$, since the dimension of $\bigoplus\limits_{p+q= k}E_{pq}^1$ is finite for each $k$.

		\end{proof}

		Now we are going to prove that $SH_*^0(M_0(\fa),\Z_2)\neq 0$, where $\fa$ is defined as in Remark~\ref{remark for sphere}.
		\begin{lemma}\label{lemma: nonvanishing}
		For $\fa=(2,2,2,\cdots,p_k)$, 
		we have $SH_*^0(M_0(\fa),\Z_2)\neq 0$, where $k+3=n,n>8, p_i's$ are sufficiently large integers.
		\end{lemma}
		\begin{proof}
			 It suffices to prove that $SH_{n-1}^0(M_0,\Z_2)\neq 0$.
			To that end we will focus on the total degree $p+q=n-2,n-1,n$ in the spectral sequence above.
			First of all, for $p=0$, we have 
			\[E_{0q}^1=H_{n+q}(M_0,\p M_0;\Z_2)=\begin{cases*}
			\Z_2, & $q=n$\\
			0,     & $q\neq n$.
			\end{cases*}
			\] 
			On the other hand,
		\begin{align*}
		\Delta(MB(p)\times S^1) &= \mu_{CZ}(MB(p)\times S^1)-\frac{1}{2}(\dim(MB(p)\times S^1)/S^1)\\
		&= f_\fa(p)-\frac{1}{2}(\dim MB(p)-1)
		\end{align*}
		where $p\pi/2, p\in \Z$ is the period.
		Meanwhile,
	\begin{align*}
	f_\mathbf{a}(p)&=3\Bigg(\Bigg\lfloor\frac{p}{2}\Bigg\rfloor+\Bigg\lceil\frac{p}{2}\Bigg\rceil\Bigg)+\sum\Bigg(\Bigg\lfloor\frac{p}{p_i}\Bigg\rfloor+\Bigg\lceil\frac{p}{p_i}\Bigg\rceil\Bigg)-\Big(\Big\lfloor p\Big\rfloor+\Big\lceil p\Big\rceil\Big)
	&\geq 3p+k-2p=p+n-3
	\end{align*}
	so $\Delta(MB(p)\times S^1)\geq p-4>n+1$ for any $p>n+5$, that is, for any Morse-Bott manifold to contribute to the homology of degree at most $n$, the period of such manifold is at most $n+5$.
	Thus, if we require $p_i>n+5$, then the only Morse-Bott manifolds could possibly contribute to total degree $p+q\leq n$ is $MB(p)\times S^1,\, p=2l,2l<n+5$ for some $0<l\in \Z$ (see  Subsection 5.5 \cite{kwon2016brieskorn}).
	Now that $p=2l,\, l<n$, we have $MB(p)=\Sigma(2,2,2)\cong \R\mathbb{P}^3$. 
	\[ H_i(\R\mathbb{P}^3\times S^1)=\begin{cases*}
	\Z_2, \quad &$i=0,4$\\
	\Z_2\oplus\Z_2, & $i=1,2,3$\\
	0, & otherwise
	\end{cases*}	
	\]
	In this case,
	\[\Delta(MB(2l)\times S^1)= 6l+n-3-4l-1=2l+n-4=p+n-4.
	\]
	So for $l>2$, $\Delta(MB(2l)\times S^1)>n$.
	For $l=1,2$, we have (see Figure~\ref{fig:Epq})
	\[E_{pq}^1=H_{q-(n-4)}(\R\mathbb{P}^3\times S^1;\Z_2)=\begin{cases*}
	\Z_2, & $q=n-4, n$\\
	\Z_2\oplus\Z_2, & $q=n-3, \, n-2,\, n-1$\\
	0,     &\text{otherwise}.
	\end{cases*}
	\] 
	Hence, $E_{2,n-3}^k(M_0,\Z_2)\neq 0$ stabilizes at the second page, so $SH_{n-1}^0(M_0,\Z_2)\neq 0$. It follows that $SH_*^0(M_0,\Z_2)\neq 0$. In particular, $SH_n^0(M_0,\Z_2)\neq 0$, since the unit lives in degree $n$.

	\begin{figure}
		\centering
		\begin{sseq}
			[ylabelstep=2,xlabelstep=2,
			ylabels={n-6;;n-4; ;n-2;;n; ;n+2;;n+4;;n+6},entrysize=0.6cm]
			{0...9}
			{-6...2}
			\ssmoveto 0 0
			\ssdropbull\ssname{a}
			
			\ssmoveto 2 {-4}
	\ssdropbull
	\ssmove 0 1
	\ssdropbull\ssname{b}
	\ssdropextension
	\ssdropbull 
     \ssdropextension
	 
	\ssmove 0 1
	\ssdropbull
	\ssdropbull
	\ssmove 0 1
	\ssdropbull
	\ssdropbull
	\ssmove 0 1
	\ssdropbull

		\ssmoveto 4 {-4}
	\ssdropbull\ssname{c}
	\ssdropextension
	\ssmove 0 1
	\ssdropbull
	\ssdropbull
	\ssmove 0 1
	\ssdropbull
	\ssdropbull
	\ssmove 0 1
	\ssdropbull
	\ssdropbull
	\ssmove 0 1
	\ssdropbull
\ssgoto c \ssgoto b \ssstroke[ color=red]\ssarrowhead

				\ssmoveto 8 {-4}
			\ssdropbull
			\ssmove 0 1
			\ssdropbull
			\ssdropbull
			\ssmove 0 1
			\ssdropbull
			\ssdropbull
			\ssmove 0 1
			\ssdropbull
			\ssdropbull
			\ssmove 0 1
			\ssdropbull
			
						\ssmoveto 6 {-4}
			\ssdropbull
			\ssmove 0 1
			\ssdropbull
			\ssdropbull
			\ssmove 0 1
			\ssdropbull
			\ssdropbull
			\ssmove 0 1
			\ssdropbull
			\ssdropbull
			\ssmove 0 1
			\ssdropbull
		\ssmoveto 6 {-6}
		\ssdrop{\phantom{d}}\ssname{d}
		\ssgoto a \ssgoto d\ssstroke[dashed]
		\end{sseq}
		\caption{$E_{pq}^2(M_0,\Z_2)=E_{pq}^1(M_0,\Z_2)$: $\dim E_{2,n-3}^2(M_0,\Z_2)=2$, so it can not be killed by $d_2$ (red arrow) since $\dim E_{4,n-4}^2(M_0,\Z_2)=1$ on the second page.}
		\label{fig:Epq}
	\end{figure}

	\end{proof}

			\begin{remark}
			Lemma~\ref{lemma: nonvanishing} shows that $I_+(\Sigma_0(\fa))$ is well-defined since $SH_*(M_0)\neq 0$. Furthermore, $I_+(\Sigma_0(\fa))$ is a finite group.
	     	\end{remark}

		Now we are ready to prove Theorem~\ref{main technical theorem}:
		first, we will take $\mathbf{a}=(2,2,2,p_1,\cdots,p_k)$ satisfying 
		\begin{itemize}
			\item $p_i> k+8,$
			\item $\sum\frac{1}{p_k}=\frac{1}{2}$.
		\end{itemize}
		Recall 
		\[U_\fa(\epsilon)=\{\mathbf{z}\in \mathbb{C}^{n+1}|z_0^{a_0}+\cdots+z_n^{a_n}=\epsilon\cdot\beta(||\mathbf{z}||^2)\},\]
		and  
		\[W_\epsilon^1=U_\epsilon\cap B(1),\quad \p W_\epsilon^1=\Sigma(\fa).\]
		Let $(M_0(\fa),\lambda_0)$ be defined as in Remark~\ref{def of key block}.
		Then we have the following facts:

	\begin{enumerate}
		\item $(M_0,\lambda_0)$ is strongly ADC;
	    \item $SH_n^0(M_0)\neq 0$ and is finitely dimensional. 
	    \item $H_i(M_0)=0, i>1,i\neq n,n+1$.
	\end{enumerate}
		The first claim is true due to Proposition~\ref{main theorem}. We only need to check the condition that $m(\fa)\geq 3$, which in turn is the result of Lemma~\ref{minimal index}. The second claim is proved in Lemma~\ref{finiteness of group}. 
		
		On the other hand, the Liouville vector field $Y_\lambda$ is gradient-like (Lemma~\ref{lemma for transversality})   for the function $F$ which we used to define the Weinstein domain.
		  Therefore, 
		it is Liouville homotopic to Stein domain $(M_0,J,\phi)$(Remark~\ref{def of key block}). $\pi_1(M_0)=\Z$ since $M_0$ is diffeomorphic to $\mathbb{C}^{n+1}\setminus V_\fa(0)$ (Proposition~\ref{fundamental group of M0}). Let $\gamma$ be an isotropic circle generating $\pi_1(M_0)$. Such $\gamma$ exists by the $h-$principle(Lemma~\ref{h-principle}).
		Let $M_1$ be Weinstein manifold obtained from $M_0$ by attaching a Weinstein 2-handle with respect to the trivialization $\Phi$ (Proposition~\ref{handle attachment for construction}). $M_1$ is of finite type because $M_0$ is. Furthermore, attaching 2-handle along $\gamma$ kills the fundamental group.  
		Now we are going to prove that$(M_1,\lambda_1,\psi_1)$ satisfies all conditions in Theorem~\ref{main technical theorem}
		\begin{proof}[proof of  Theorem~\ref{main technical theorem}]\label{proof 1.5}
		Indeed,we have the following facts about $(M_1,\lambda_1,\psi_1)$:
		\begin{enumerate}
			\item $(\p M_1, \lambda_1)$ is asymptotically dynamically convex;
			\item $SH_*(M_1)\cong SH_*(M_0)$ as rings.
			\item $\tilde{H}_i(M_1)=0, i\neq n,n+1$
		\end{enumerate}

		 The first statement is true because subcritical surgery preserves the ADC property, by Theorem~\ref{theorem:contact surgery}. The second statement is due to the fact that subcritical surgery doesn't change the ring structure of symplectic homology, see Theorem~\ref{invariant of SH}. The last statement on homology follows Proposition~\ref{contractible after surgery}.
		\end{proof}

\bibliographystyle{alpha}
\bibliography{reference1}
\end{document}